\documentclass[12pt]{article} 

\usepackage[utf8]{inputenc}
\usepackage[T1]{fontenc} 

\usepackage{geometry} 
\usepackage{graphicx} 

\usepackage{titlesec}

 \usepackage[parfill]{parskip} 

\usepackage[english]{babel}
\usepackage{amsmath}
\usepackage{booktabs} 
\usepackage{array} 
\usepackage{paralist} 
\usepackage{verbatim} 
\usepackage{subfig} 
\usepackage{ucs}
\usepackage{dsfont}
\usepackage{amsfonts}
\usepackage{mathabx}
\usepackage{amsthm}
\usepackage{thmtools}

\usepackage{bm}

\usepackage{tikz}
\usetikzlibrary{decorations.markings}
\usetikzlibrary{decorations.pathreplacing,calligraphy}
\usepackage{float}

\usepackage[colorlinks=true,pdfstartview=FitH,bookmarks=false]{hyperref}
\usepackage[all]{xy}

\usepackage{fancyhdr} 
\usepackage{sectsty}
\allsectionsfont{\sffamily\mdseries\upshape} 

\usepackage[nottoc,notlof,notlot]{tocbibind} 
\usepackage[titles,subfigure]{tocloft} 

\theoremstyle{plain}
\newtheorem{theorem}{Theorem}[section]
\newtheorem{definition}[theorem]{Definition}
\newtheorem{coro}[theorem]{Corollary}
\newtheorem{lem}[theorem]{Lemma}

\newtheorem{prop}[theorem]{Proposition}

\newtheorem{question}[theorem]{Question}

\theoremstyle{remark}
\newtheorem{rem}[theorem]{Remark}

\newtheorem{ex}[theorem]{Example}

\newcommand\N{\mathbb{N}}
\newcommand\Z{\mathbb{Z}}

\newcommand\R{\mathbb{R}}
\newcommand\C{\mathbb{C}}
\newcommand\D{\mathbb{D}}
\newcommand\Ss{\mathds{S}}

\newcommand\ra{\rightarrow}

\newcommand\FF{\mathcal{F}}

\newcommand\e{\mathrm{e}}

\newcommand{\ti}[1]{\widetilde{#1}}

\newcommand{\Fix}{\mathrm{Fix}}
\newcommand{\ind}{\mathrm{ind}}

\newcommand{\Sing}{\mathrm{Sing}}
\newcommand{\tin}{t\in[0,1]}

\newcommand{\Rec}{\mathrm{Rec}}
\newcommand{\id}{\mathrm{id}}

\newcommand{\gen}{\mathrm{gen}}
\newcommand{\inte}{\mathrm{int}}
\newcommand{\im}{\mathrm{Im}}
\newcommand{\Fr}{\mathrm{Fr}}
\newcommand{\Adh}{\mathrm{Adh}}
\newcommand{\HM}{\mathrm{HM}}
\newcommand{\HF}{\mathrm{HF}}
\newcommand{\CM}{\mathrm{CM}}
\newcommand{\CF}{\mathrm{CF}}

\date{December 2021}

\title{Barcodes for Hamiltonian homeomorphisms of surfaces}
\author{Beno\^it Joly}

\begin{document}

\maketitle

\abstract{In this article, the main goal is to give a dynamical point of view of Floer homology barcodes for Hamiltonian homeomorphisms of surfaces. More specifically, we describe a way to construct barcodes for Hamiltonian homeomorphisms of surfaces from graphs. We will define graphes associated to maximal isotopies of a Hamiltonian homeomorphism using Le Calvez's positively transverse foliation theory and to those graphs we will associate barcodes. In particular, we will prove that for the simplest cases, our constructions coincide with the Floer Homology barcodes.}

\tableofcontents

\section{Introduction}

\subsection{Goals and motivations}

\textbf{Main question}\\

In this article, we will think about the following question.

\begin{question}\label{question1}
Can we construct barcodes for Hamiltonian homeomorphisms of surfaces, equal to the Floer homology barcodes, using dynamical objects as Le Calvez's transverse foliations?
\end{question}

\textbf{Barcodes}\\

A barcode $B = (I_j)_{j\in \N}$ is a countable collection of intervals (or bars) of the form $I_j = (a_j , b_j ], a_j \in \R, b_j \in \R \cup \{+\infty\}$ which satisfy certain finiteness assumptions. \\


Let us desribe a first simple example of a barcode. We consider a Hamiltonian flow $(f_t)_{\tin}$ induced by an autonomous Hamiltonian function $H$. In this case, for $t$ small enough, the barcode of the Hamiltonian diffeomorphism $f_t$ is equal to the filtered Morse Homology $(\HM^t_*(H))_{t\in\R}$ of $H$. \\
To give more details about this example we explain how the filtered homology of $H$ can be interpreted as a collection of bars. The ends of these bars are in correspondence with the critical points of $H$ and can be classified as follows. 

\begin{itemize}
\item There are the \textit{death points} which are the critical points $x$ of $H$ ending some homology, meaning that the dimension of the vector spaces $(\HM^t_*(H))_{t\in\R}$ decreases at $H(x)$.
\item There are the \textit{birth points} which are the critical points $x$ of $H$ generating homology in $\HM_*^{H(x)}(H)$, meaning that the dimension of the vector spaces $(\HM^t_*(H))_{t\in\R}$ increases at $H(x)$. The value $H(x)$ of a birth point $x$ will be the begining of a bar.
\end{itemize}

The bars of a barcode can be described by the following classification of the birth points. 

\begin{itemize}
\item A birth point $x$ can be "homological" and associated to the semi-infinite bar $(H(x),+\infty)$ in the barcode if the homology it generates in $\HM_*^{H(x)}$ persists in the vector spaces  $(\HM_*^{t})_{t\geq H(x)}$. 
\item A birth point $x$ can be "bound to die" and paired wih a death point $y$ and associated to a finite bar $(H(x),H(y)]$ in the barcode if the homology it generates in $\HM_*^{H(x)}$ disapears in $\HM_*^{H(y)}$.
\end{itemize}

More generally filtered homologies are examples of \textit{persistence modules}. In fact, we will see that barcodes use to classify persistence modules up to isomorphisms. Roughly speaking, it is equivalent to consider a barcode as a set of bars or as a filtered homology.\\

\textbf{Barcodes in symplectic geometry}\\

Given a Hamiltonian function $(H_t)_{\tin}$ on a symplectic manifold $(M,\omega)$, we define the \textit{action functionnal} $A_H$ on the space of contractible loops of $M$ by\\
$$A_H(\gamma)=-\int_D u^*\omega +\int_0^1H_t(\gamma(t))dt,$$
where $u$ is an extension to the disk of the contractible loop $\gamma:\Ss^1\ra M$, that is, a map $u: D=\{z\in\mathbb{C} | |z|\leq 1\} \ra M$ such that $u(\e^{2i\pi t})=\gamma(t)$. If we suppose that $\pi_2(M)=0$, the function $A_H$ does not depend on the choice of $u$ and it will always be the case in this article. We will see in the preliminaries that for a Hamiltonian diffeomorphism $f$, the action function $A_H$ does not depend on the choice of the Hamiltonian function $H$ which induces $f$, hence it defines an action function $A_f$ asociated to $f$. Moreover, the critical points of the action functionnal $A_f$ are the trajectories of the contractible fixed points of $f$.\\

One may associate cannonically a barcode $B(\phi)$ to every Hamiltonian diffeomorphism $f$ using Hamiltonian Floer homology, see section \ref{introbarcodes}. The barcode $B(f)$ characterizes the filtered Floer complex associated to $f$ up to a quasi-isomorphism. The end-points of the bars are the critical value of the action functionnal and hence the barcode $B(f)$ contains all the Floer theoretic filtered invariants of $f$.\\

It is proven, see \cite{ROUX1}, \cite{POLT} for examples, that the barcode $B(f)$ depends continuously, with respect to the uniform topology, on $f$. When $f$ is a diffeomorphism $B(f)$ is interprated by the filtered Floer homology of $f$ and when $f$ is a Hamiltonian homeomorphism, by a limiting process we can define a barcode $B(f)$ of $f$. Hence, when $f$ is a Hamiltonian homeomorphism the dynamical information the barcode carries is no clear. Le Roux, Seyfaddini and Viterbo, see \cite{ROUX1}, used Le Calvez's foliation to extract dynamical information from the barcode of a Hamiltonian homeomorphism. The goal in this article will be different as we aim to use the ideas of this theory to construct barcodes instead of extracting information from the Floer homology barcodes.\\

The barcode $B(f)$ is an invariant of conjugacy which gives information about the structure of the set of fixed points and the \textit{spectral invariants} of $f$. The spectral invariants have been introduced by Viterbo \cite{VIT}. They have been used in numerous deep applications and their theory has been developped in many contexts, we can cite for example the work of Schwarz \cite{SCH} and Oh \cite{OH}. They are powerful tools which took an important place in the development of symplectic topology.\\

Recently, the notion of barcodes appears as a great tool to study $C^0$ symplectic geometry, let us cite for example the work of Buhovski, Humilière and Seyfaddini \cite{BUH2}, Jannaud \cite{JAN} and Le Roux Seyfaddini and Viterbo  \cite{ROUX3}.\\

\textbf{Hamiltonian homeomorphisms}\\

In symplectic geometry, we can define the notion of \textit{Hamiltonian homeomorphism} of a surface $\Sigma$ by taking the closure of the Hamiltonian diffeomorphisms of $\Sigma$. This definition comes from the Gromov-Eliashberg theorem \cite{ELIA} which states that if a sequence of symplectomorphisms of a symplectic manifold $(M,\omega)$ converges in the $C^0$ topology to a diffeomorphism then this diffeomorphism is a symplectomorphism as well.\\

For a Hamiltonian homeomorphism of a surface, we are not able to consider directly its Floer Homology as the construction requires at least a $C^2$ setting. Howeover, on surfaces, the barcode $B(f)$ depends continuously, in the uniform topology, on $f$ and moreover, extends to Hamiltonian homeomorphisms, see \cite{ROUX3} for more details.\\

The barcode of a Hamiltonian homeomorphism $f$ is defined by a limiting process and it is natural to wonder if it is possible to describe a direct construction.\\

 Moreover, the notion of Hamiltonian homeomorphism of surfaces is well-known in dynamical systems and has a dynamical interpretation thanks to the notion of rotation vectors. On a symplectic surface $(\Sigma,\omega)$, $\omega$ is an area form which induces a Borel probability measure $\mu$.  We will say that a homeomorphism $f$ of an oriented compact surface is \textit{Hamiltonian} if it is isotopic to the identity and preserves a Borel probability measure $\mu$ whose support is the whole surface and whose \textit{rotation vector} $\rho(\mu)$  is zero.\\

\textbf{Le Calvez's positively transverse foliations}\\

A key motivation for this article is to bring a dynamical interpretation of the barcodes for Hamiltonian homeomorphisms of surfaces. Taking this direction, we will give some constructions of barcodes, inspired by the Floer homology constructions, using Le Calvez's foliation theory.\\

Le Calvez's positively transverse foliations theory has many applications in the study of dynamical systems of surfaces. For example in the study of prime ends by Koropecki, Le Calvez and Nassiri \cite{KOR}, the study of homoclinic orbits for area preserving diffeomorphisms by Sambarino and Le Calvez \cite{SAM} or the results about the forcing theory of Le Calvez and Tal \cite{TAL1,TAL2}.\\ 

Nowadays, Le Calvez'spositively transverse foliations theory \cite{CAL4} represents one of the most important dynamical tool in the study of the dynamics of homeomorphisms of surfaces. This theory already found applications to Barcodes of Hamiltonian homeomorphisms of surfaces. For example, for a homeomorphism $f$ which preserves the area, Le Roux, Seyfaddini and Viterbo in \cite{ROUX3} used Le Calvez's foliations theory to extract dynamical informations of the barcode of $f$ without Kislev-Shelukhin's result \cite{SHE4}.\\

Here are some details about transverse foliations. Let us consider a homeomorphism $f$ on a surface. There are sets $X$ of fixed points of $f$, called \textit{maximal unlinked sets}, such that there exists an isotopy, called \textit{maximal isotopy}, from $\id$ to $f$ fixing all points of $X$ and which are maximal for the inclusion. \\

Le Calvez proved that given a maximal unlinked set $X$ of fixed points of $f$ and an isotopy $I$ fixing all the fixed points of $X$, there exist oriented foliations $\FF$ \textit{positively transverse} to the isotopy $I$. Roughly speaking, this means that, given a point in the complement of $X$, its trajectory along the isotopy $I$ is homotopic in $\Sigma\backslash X$ to a path transverse to $\FF$.\\

Moreover, if we suppose that $f$ is area-preserving, we will see in \ref{foliation} that those foliations are \textit{gradient-like}. To keep it simple, this means that we can see such a foliation as the gradient lines of a function defined on the surface. In particular, every leaf of a gradient-like foliation is an injective path, called a \textit{connexion}, between two singularities of $\FF$ and there is no cycle of connexions. \\
In the particular case where $f$ has finitely many fixed points, by a result of Wang \cite{WANG}, the notion of action function can be extended. A key point is that, for every leaf $\phi$ of $\FF$ the action function $A_f$ of $f$ satifsies $A_f(\alpha(\phi))>A_f(\omega(\phi))$. \\

To give an example, we can consider again a Hamiltonian diffeomorphism $f$ induced by an autonomous Hamiltonian function $H$ on a surface. The induced Hamiltonian flow is a maximal isotopy $I$ of $f$ and the gradient flow of $H$ is a gradient-like foliation positively transverse to $I$. In this case, the only maximal unlinked set of fixed points fixed by $I$ is the set of critical points of $H$.\\

\subsection{Results}

We describe briefly the results of the article. We provide distinct constructions of barcodes for Hamiltonian homeomorphisms of surfaces.\\

\textbf{First construction}\\

We will describe a first construction in section \ref{génériqueF} under some generic hypothesis which is inspired from the Morse and Floer homology constructions. We will consider a Hamiltonian homeomorphism $f$ of an oriented compact surface $\Sigma$ with a finite number of fixed points which are, in a sense, non degenerate and such that the set of fixed points is \textit{unlinked}, meaning that there exists an isotopy $I=(f_t)_{\tin}$ from $\id$ to $f$ fixing all the fixed points of $f$. By Le Calvez's theorem we can consider a gradient-like foliation $\FF$ transverse to $I$. We will suppose that $\FF$ satisfies some "generic" hypothesis which allows us to construct a chain complex inducing a filtered homology and then a barcode denoted $B_{\gen}(\FF)$. \\

We have the following theorem proved in section \ref{equality}.

\begin{theorem}\label{thm1intro}
The Barcode $B_{\gen}(\FF)$ does not depend on the choice of the foliation $\FF\in\FF_{\gen} (I)$.
\end{theorem}

In the case of a Hamiltonian diffeomorphism close enough of the identity and generated by an autonomous hamiltonian function we will obtain the following result.

\begin{theorem}\label{thm1'intro}
If we consider a Hamiltonian diffeomorphism $f$ with a finite number of fixed points which is $C^2$-close to the identity and generated by an autonomous Hamiltonian function then the barcode $B_{\gen}(\FF)$ is equal to the Floer homology barcode of $f$.
\end{theorem}

Let us give the idea of the construction. Since $f$ is area-preserving, we will see that there are three kinds of singularities for the foliation $\FF$: sinks, sources and saddle points. We will suppose that $\FF$ is in the set $\FF_{\gen}(I)$ of "generic" foliations positively transverse to $I$, meaning that there are finitely many leaves between sources and saddle points and between sinks and saddle points. 
In the Morse Homology theory, the chain complex is defined by counting modulo $2$ the number of trajectories between the critical points of a Morse function $f$. Following the same ideas we will be able to define a chain complex associated to $\FF$ by counting modulo $2$ the number of leaves between singularities of $\FF$ and more precisely the number of leaves between sinks and saddle points and between sources and saddle points. \\

A natural question appears.

\begin{question}
Can we generalize the construction to barcodes for every Hamiltonian homeomorphisms of surfaces?
\end{question}

\textbf{The map $\mathcal{B}$}\\

In general there is no natural way to construct a chain complex from a positively transverse foliation. The difficulties come from geometrical limitations of the foliations. \\

In section \ref{construction} we will construct an application $\beta$ which associates a barcode to triplets $(G,A,i)$ where $G$ is an oriented graph on the set of vertices $X$ equipped with an action function $A$ defined on $X$, meaning that for every edge $e$ of $G$ from $x$ to $y$ we have $A(x)>A(y)$, and an index function $i:X \ra\Z$. \\

All the next constuctions are based on the map $\mathcal{B}$.\\

\textbf{Second construction}\\

Given a Hamiltonian homeomorphism $f$, we will construct barcodes associated to maximal unlinked sets of fixed points of $f$.\\

Let us consider a maximal unlinked set $X$ of fixed points of $f$, a maximal isotopy $I=(f_t)_{\tin}$ fixing all the fixed points of $X$, and a gradient-like foliation $\FF$, positively transverse to $I$. We will begin by associating a graph $G(\FF)$ to the foliation $\FF$ whose set of vertices is equal to $X$ and for every couple $(x,y)$  of vertices there is an edge from $x$ to $y$ if there is a leaf $\phi$ of $\FF$ starting at $x$ and ending at $y$.\\

In section \ref{2barcode}, we will consider the barcode $\beta(G(\FF), A_f,\ind(\FF,\cdot))$, denoted $\bm{\beta}_\FF$, where $A_f$ is the action function of $f$ and $\ind(\FF,\cdot)$ the index function induced by $\FF$ and prove some useful properties. \\

To study the 	independence of the barcodes $\bm{\beta}_\FF$ on the choice of the foliation $\FF$, we will construct a more general barcode associated to $X$.\\

\textbf{Third construction}\\

In section \ref{order}, we will introduce an order on the fixed points of $X$. For two fixed points $x,y\in X$ we will say that $x>y$ if there exists an oriented path $\gamma$ from $x$ to $y$ of which any lift $\ti{\gamma}$ on the universal cover $\D^2$ of $\Sigma\backslash X$ is a \textit{Brouwer line} for the natural lift $\ti{f}$ of $f$, meaning that $\ti{\gamma}$ is the boundary of an attractor of $\ti{f}$. In the same ideas, we associate to this order a graph $G(>)$ whose set of vertices is equal to $X$ and for every couple of vertices $x$ and $y$ there is an edge from $x$ to $y$ if $x>y$.\\

We will consider the barcode $\bm{\beta}_>=(G(>),A_f,\ind(I,\cdot))$ which depends only on $X$.\\

\textbf{Theorems}\\

In the section \ref{equality} we prove the main results of this article about the previous consturctions. It is divided in two parts. \\

In the first part we will prove the following result.

\begin{theorem}\label{thm2intro}
The barcode $\bm{\beta}_\FF$ does not depend on the choice of $\FF$ and only depends on the maximal unlinked set of fixed points $X$.
\end{theorem}

Thereom \ref{thm2intro} will appear as a corollary of the following theorem.

\begin{theorem}\label{thm3}
For every foliation $\FF$ positively transverse to the isotopy $I$ we have
$$\bm{\beta}_>=\bm{\beta}_\FF.$$
\end{theorem}

The proof will be based on the geometrical properties of the gradient-like foliations from section \ref{2barcode} and the dynamical properties of the order on $\Sing(I)$ from section \ref{order}.\\

In the second part, we will prove the following result which enlighten the link between the barcode associated to a maximal set of fixed points and the first construction. 

\begin{theorem}\label{thm3'}
Let us consider a Hamiltonian homeomorphism $f$ on a compact surface $\Sigma$ whose set of fixed points is finite, unlinked, and such that each fixed point $x\in\Fix(f)$ is not degenerate.\\
We consider a maximal isotopy $I$ such that $\Sing(I)=\Fix(f)$ then for a foliation $\FF\in\FF_{\gen}(I)$ we have
$$B_{\gen}(\FF)=\bm{\beta}_\FF.$$
\end{theorem}

\textbf{Acknowledgments}\\

This parper results from my PhD thesis at the institut Mathématiques de Jussieu - Paris Rive Gauche,  Sorbonne University. and I would like to thanks my advisors Patrice Le Calvez and Frédéric Le Roux for their help and support.
I am also thankful to Rémi Leclercq and Barney Bramham for useful comments and suggestions from my PhD reviews.\\

\section{Preliminaries}\label{preliminaries}

\subsection{Morse homology and Floer homology}

We give a brief review of the Morse and Floer homology theories.  We refer to \cite{AUD} for more details about their constructions.\\

\textbf{Morse homology}\\

We consider a Morse function $H$ on a manifold $M$. The Morse chain complex of $H$, $\CM_*(H)$, is the vector space spanned by $\mathrm{Crit}(H)$ over the ground field $\Z/2\Z$. The boundary map of $\CM_*(H)$ counts certains \emph{broken geodesics} of a chosen adapted Riemaniann metric $g$ on $M$ which can be viewed as gradient lines of the gradient lines of $H$ induced by the metric $g$.\\
For any $t\in \R$, we can define $\CM_*^t:=\{\sum a_z z\in \CM_*(H) \ : \ H(z)<t\}$. One may prove that the Morse boundary map preserves $ \CM_*^t$ and hence we can define its homology $\HM^t_*(H)$. The set $(\HM^t_*)_{t\in\R}$ is called the \emph{filtered Morse homology} of $H$. \\

\textbf{Floer homology}\\

We consider a non degenerate time dependent Hamiltonian function $H$ on a symplectic manifold $(M,\omega)$. Symetrically as the Morse homology theory, the Floer chain complex of $H$, $\CF_*(H)$, is the vector space spanned by $\mathrm{Crit}(A_H)$ over the ground field $\Z/2\Z$ (or more generally over a field $\mathbb{F}$) where $A_H$ is the action function of $H$. The boundary map of $CF_*(H)$ counts certain solutions of a perturbed Cauchy-Riemann equation for a chosen $\omega$-compatible almost complex structure $J$ on $TM$, which can be viewed as isolated gradient flow lines of $A_H$.\\
For any $t\in \R$, we can define $\CF_*^t:=\{\sum a_z z\in \CF_*(H) \ : \ A_H(z)<t\}$. One may prove that the Floer boundary map preserves $\CF_*^t$ and hence we can define its homology $\HF^t_*(H)$. The set $(\HF^t_*)_{t\in\R}$ is called the \emph{filtered Floer homology} of $H$. \\

\subsection{Barcodes and persistence modules}\label{introbarcodes}

The notion of barcodes and persistence modules was used in topological data analysis, see for example G. Carlsson in \cite{CARL} or R. Ghrist in \cite{GHR}. Barannikov already noticed the existence of a filtration of the Morse homology in \cite{BAR} and we can find the notion of persistence modules in Usher's work \cite{USH1,USH2} but the barcodes have been introduced in symplectic topology by Polterovich and Shelukhin \cite{POLT}. The same year, without the terminology of the barcodes Usher and Zang published some results about the persistent homology in \cite{USH3}. 

Most of the following definitions and results are coming from \cite{ROUX3}. One can also refer to Chazal, De Silva, Glisse and Outdot's book \cite{CHAZ} or to \cite{EDEL}.\\

\textbf{Barcodes}\\

Let us consider a special family of intervals $\textbf{B}$  of the form $((a_j,b_j])_{j\in\{1,...,n\}}$, with $-\infty \leq a_j\leq b_j\leq +\infty$, where we allow trivial intervals of the form $(a,a]$. We say that two families are equivalent if removing all intervals of the form $(a,a]$ from them yields the same family.

\begin{definition}
A \emph{barcode} $B$ is an equivalence class of family of intervals $\textbf{B}$.
\end{definition}

By convenience, we will often identify a list of intervals with the corresponding barcode.\\

Let $a\leq b,c\leq d$ be four elements of $\R\cup \pm \{\infty$\}. We set $d((a,b],(c,d])=\max\{|c-a|,|d-b|\}$, with convention that $d(\infty,\infty)=0$. Note that if $c=d=\frac{a+b}{2}$, then $d((a,b],(c,d])=\frac{b-a}{2}$.\\

\begin{definition}
Let $B_1,B_2$ be barcodes and take representatives $\textbf{B}_1=(I_j^1)_{j\in J}$, $\textbf{B}_2=(I_k^2)_{k\in K}$. The \textit{bottleneck distance} between $B_1,B_2$, denoted by $\mathrm{d}_{\mathrm{bot}}(B_1,B_2)$, is the infimum of the set of $\epsilon$ such that there is a bijection $\sigma$ between two subsets $J',K'$ of $J,K$ with the property that for every $j\in J'$, $d(I^1_j,I^2_{\sigma(j)})\leq \epsilon$ and all the remaining intervals $I^1_j,I^2_k$ for $j\in J\backslash J'$, $k\in K\backslash K'$ have length less than $2\epsilon$.
\end{definition}

We will denote $\mathrm{Barcode}$ the set of barcodes in the next sections.\\

\textbf{Persistence module}\\
 
\begin{definition}
A \emph{persistence module} $V$ is a family $(V_t)_{t\in\R}$ of vector spaces equipped with morphisms $i_{s,t}: V_s\ra V_t$, for $s\leq t $, satisfying:
\begin{enumerate}
\item For all $t\in\R$ we have $i_{t,t}=\id$ and for every $s\leq t\leq u$ we have $i_{t,u}\circ i_{s,t}=i_{s,u}$,
\item There exists  a finite subset $F\subset \R$, often referred to the spectrum of $V$, such that $i_{s,t}$ is an isomorphism whenever $s,t$ belong to the same connected component of $\R\backslash F$,
\item For all $t\in\R$, $\underset{s \ra t, s<t}{\lim} V_s =V_t$ ; equivalently, for fixed $t$, $i_{s,t}$ is an isomorphism for $s<t$ sufficiently close to $t$.
\end{enumerate}
\end{definition}

Let us consider a persitence module $(V_t)_{t\in\R}$ equipped with the morphisms  $(i_{s,t})_{s\leq t}$. For any $t\in\R$, there exists $\epsilon$ such that $i_{s,u}:V_s\ra V_u$ is an isomorphism if $s,u\in(t-\epsilon,t]$ or if $s,u\in(t,t+\epsilon)$. Choose $t^-\in(t-\epsilon,t]$ and $t^+\in(t,t+\epsilon)$ and let $j(t)=\dim(\mathrm{Ker}(i_{t^-,t^+}))+\mathrm{codim}(\mathrm{Im}(i_{t^-,t^+}))$. Notice that $j(t)$ is zero except for $t$ in the spectrum of $\textbf{V}$. We say that $\textbf{V}$ is \emph{generic} if $j(t)\leq 1$ for all $t\in\R$.\\
 
\textbf{Functorial relations between the barcodes and the persistence modules}\\

To establish the link between the previous objects we consider two functors as follows.

\begin{enumerate}[(i)]
\item Consider an interval $I$ of the form $(a,b]$ and define $Q_t(I)= \Z/2\Z$, if $t\in I$, and $Q_s(I)=\{0\}$, if $s\notin I$.  $Q_s(I)$ is a persistence module, with $i_{s,t}$ equal to $\id$ if $s,t\in I$ and $0$ otherwise. For a set of intervals $\mathcal{I}$ for each $t\in\R$ we define 
$$Q_t(\mathcal{I})=\bigoplus_{I\in\mathcal{I}} Q_t(I).$$
\item We define a functor $\beta$ from the set of generic persistence modules into the set of barcodes which associate to a generic persistence module $\textbf{V}=(V_s)_{s\in\R}$ a barcode. We denote $(i_{s,t})_{s\leq t}$ the family of morphisms equipped with $\textbf{V}$. Let us consider the set of $t$ in the spectrum of $V$ such that $\mathrm{dim}(\mathrm{Ker}(i_{t^-,t^+}))=1$ and label its elements $b_1,...,b_n$. For each $b_j$, there exists  a unique $a_j\in\R$ with the following property: Let $x\in V_{b_j^-}$ represents a non-zero element in $\mathrm{Ker}(i_{t^-,t^+})$, the element $x$ is in the image of $i_{a^+_j,b^-_j}$ but $x$ is not in the image of $i_{a^-_j,b^-_j}$. We label the remaining elements of the spectrum of $\textbf{V}$ by $\{c_1,...,c_m\}$. The barcode $\beta(\textbf{V})$ consists of the list of intervals: 
$((a_j,b_j],(c_k,+\infty))$, where $1\leq j\leq n$ and $1\leq k\leq m$.\\
One may prove that the functor $\beta$ extends to the set of persistence modules, we refer to $\cite{ROUX3}$ for more details.\\

\end{enumerate}

The following theorem holds.
\begin{theorem}
The functors defined above satisfy the following properties.
\begin{enumerate}
\item $\beta\circ Q= \id_{\mathrm{Barcode}}$.
\item $\beta$ and $Q$ are isometries for the interleaving distance (see next definition).
\end{enumerate}
\end{theorem}

We define the interleaving distance.

\begin{definition}
Let $\textbf{V}=(V_s)_{s\in\R}$ and $\textbf{W}=(W_s)_{s\in\R}$ be two persistence modules, the pseudo-distance $d_{int}(\bf{V},\bf{W})$, called the \textit{interleaving} distance, is defined as the infimum of the set of $\epsilon$ such that there are morphisms $\phi_s:V_s\ra W_{s+\epsilon}$ and $\psi_s:W_s\ra V_{s+\epsilon}$ "compatible" with the $i_{s,t}, j_{s,t}$ in the following sense:
$$\xymatrix{
    V_{s-\epsilon} \ar[r]^{\phi_{s-\epsilon}} \ar[d]_{i_{s-\epsilon,t-\epsilon}}  & W_s  \ar[r]^{\psi_{s}} \ar[d]^{j_{s,t}} & V_{s+\epsilon} \ar[r]^{\phi_{s+\epsilon}} \ar[d]^{i_{s+\epsilon,t+\epsilon}} & W_{s+2\epsilon} \ar[d]^{i_{s+2\epsilon,t+2\epsilon}}\\
    V_{t-\epsilon} \ar[r]^{\phi_{t-\epsilon}}   & W_t  \ar[r]^{\psi_{t}}  & V_{t+\epsilon} \ar[r]^{\phi_{t+\epsilon}} & W_{t+2\epsilon} \\
  }$$
 where $\psi_s\circ\phi_{s-\epsilon}=i_{s-\epsilon,t+\epsilon}$ and $\phi_{s+\epsilon}\circ\psi_s=j_{s,s+2\epsilon}$ s.t. the diagrams commute for all $s\leq t$.
\end{definition}

\textit{The Morse example.}
To give a good idea of what a barcode is, we describe the case of a Morse function. Let $\Sigma$ be a compact surface and $H:\Sigma\ra \R$ a Morse function. The filtered Morse homology $(H_*(\{H<t\}))_{t\in\R}$ is a persistence module where the set $(i_{s,t})_{s\leq t}$ is given by the inclusions $i_{s,t} : H_*(\{H<s\}) \ra H_*(\{H<t\})$. .\\

\subsection{Positively transverse foliation's theory}\label{foliation}

From now we consider a connected, compact and oriented surface $\Sigma$ without boundary. Let $\mathrm{Homeo}(\Sigma)$ be the space of homeomorphisms of $\Sigma$ equipped with the topology of uniform convergence on $\Sigma$. For $f\in\mathrm{Homeo}(\Sigma)$, $\Fix(f)$ represents the set of fixed points of $f$.\\

\textbf{Isotopies and maximal Isotopies}\\

An isotopy is a continuous path $t\mapsto f_t$ from $[0,1]$ to $\mathrm{Homeo}(\Sigma)$. We say that $f\in \mathrm{Homeo}(\Sigma)$ is isotopic to the identity if there exists an isotopy $I=(f_t)_{t\in[0,1]}$ such that $f_0=\id$ and $f_1=f$. We denote by $\mathrm{Homeo}_0(\Sigma)$ the set of those homeomorphisms.\\

Given an isotopy $I=(f_t)_{\tin}$ from $\id$ to $f$, we can extend it to an isotopy defined on $\R$ by the periodic relation $f_{t+1}=f_t\circ f_1$. We define the set of singularities $\Sing(I)$ of $I$ as follows.
$$\Sing(I)=\{x\in\Sigma | \ \forall t\in [0,1], \ f_t(x)=x\}.$$
The complement of $\Sing(I)$ in $\Sigma$ is called the domain of $I$ and denoted $\mathrm{Dom}(I)$.\\

For a point $z\in\Sigma$, the arc $\gamma: [0,1] \ra \Sigma$ where for each t$\in[0,1]$, $\gamma(t)=f_t(z)$ is called the \textit{trajectory} of $z$ along the isotopy $I$. For every $n\geq 0$, we denote by $\gamma_n(z)$ the concatenation of the trajectories of $z,f(z),...,f^{n-1}(z)$.\\

We fix a homeomorphism $f\in\mathrm{Homeo}_0(\Sigma)$. A set $X\subset \Fix(f)$ is say to be \textit{unlinked} if there exists an isotopy $I=(f_t)_{t\in[0,1]}$ from $\id$ to $f$ such that $X$ is included in the set of singularities of $I$. \\
We denote by $\mathcal{I}(f)$ the set of couples $(X,I)$ such that $I$ is an isotopy from $\id$ to $f$ and $X\subset \Sing(I)$. The set $\mathcal{I}(f)$ is naturally equipped with a pre-order $\leq$, where 
$$(X,I)\leq (X',I'),$$ 
if $X\subset X'$ and for each $z\in \Sigma\backslash X$, its trajectory along $I'$ and $I$ are homotopic in $\Sigma\backslash  X$. The couple $(X',I)$ is called \textit{an extension} of $(X,I)$. An isotopy $I\in\mathcal{I}$ is called a \emph{maximal isotopy} in $\mathcal{I}$ if the couple $(\Sing(I),I)$ is a maximal element of $(\mathcal{I},\leq)$.\\

A recent result by Béguin, Crovisier and Le Roux \cite{ROUX4} asserts that for a homeomorphism $f\in\mathrm{Homeo}_0(\Sigma)$ isotopic to the identity there always exists a maximal isotopy (a weaker result was previously proved by O. Jaulent \cite{JAU}). We will often use Corollary 1.3 of \cite{ROUX4} which we write as the following theorem.

\begin{theorem}
Let us consider $f\in\mathrm{Homeo}_0(\Sigma)$. For each element $(X,I)\in \mathcal{I}(f)$ there is a maximal element $(X',I')\in \mathcal{I}(f)$ such that $(X',I')$ is an extension of $(X,I)$.
\end{theorem}

\textbf{Lefschetz index}\\

For a homeomorphism $f\in \mathrm{Homeo}(\Sigma)$ and an isolated fixed point $x$ of $f$, we define the Lefschetz index $\ind(f,x)$ of $x$ as follows. let $U$ be a chart centered at $x$ and we denote by $\Gamma$ a small oriented circle in $U$ around $x$. For $\Gamma$ sufficiently small, the map 
$$z\mapsto \frac{f(z)-z}{||f(z)-z||},$$
is well defined on $\Gamma$ and we denote by $\ind(f,x)$ the degree of this map.\\

\textbf{Rotation vectors}\\

Let $f\in \mathrm{Homeo}_0(\Sigma)$ be the time one map of an isotopy $I=(f_t)_{t\in[0,1]}$ from the identity to $f$. Among the many ways to define the \textit{rotation vector}, we restrict ourselves to \textit{positively recurrent points}. A point $z\in\Sigma$ is a \textit{positively recurrent point} of $f$ if for each neighborhood $U\subset \Sigma$ of $z$ there exists an integer $n\in\N$ such that $f^n(z)\in U$. The integer $n>0$ which is minimal for the previous property is called the \textit{first return time} and is denoted by $\tau(z)$. The set of positively recurrent points is denoted by $\Rec^+(f)$.\\

Let $z\in\Sigma$ be a positively recurrent point. Fix a $2$-ball $U\subset \Sigma$ containing $z$ and let $(f^{n_k}(z))_{k\geq 0}$ be a subsequence of the positive orbit of $z$ obtained by keeping the iterates of $z$ by $f$ that are in $U$. For any $k\geq0$, choose an arc $\gamma_k$ in $U$ from $f^{n_k}(z)$ to $z$. The homology class $[\Gamma_k]\in H_1(\Sigma,\Z)$ where $\Gamma_k$ is the concatenation of $\gamma_{n_k-1}(z)$ and $\gamma_k$ do not depend on the choice of $\gamma_k$. We say that $z$ has a rotation vector $\rho(z)\in \mathrm{H}_1(\Sigma,\R)$ if 
$$\lim_{l\ra +\infty} \frac{1}{n_{k_l}}[\Gamma_{k_l}]=\rho(z),$$
for any subsequence $(f^{n_{k_l}}(z))_{l\geq 0}$ which converges to $z$. Notice that the linking number of a periodic point $z^*$ of an orientation preserving homeomorphism of the plane is equal to the rotation number of $z^*$ in $\R^2\backslash \{z\}.$\\

In the case where $f$ preserves a Borel probability measure $\mu$, one applies Birkhoff's ergodic theorem to the first return map in $U$ and proves that $\mu$-a.e. point $z$ is positively recurrent and has a rotation vector $\rho(z)$. Moreover, the measurable map $\rho$ is bounded, and one may define the rotation vector of the measure
$$\rho(\mu)= \int_\Sigma \rho \ d\mu \in H_1(\Sigma,\R).$$

We say that $f\in\mathrm{Homeo}_0(\Sigma)$ is a \textit{Hamiltonian homeomorphism} if it preserves a Borel probability measure whose support is the whole surface and rotation vector is zero. We denote by $\mathrm{Ham}(\Sigma)$ the set of Hamiltonians on $\Sigma$.\\

\textbf{Action function of a Hamiltonian homeomorphism}\\

In this section we define dynamically the action function of a Hamiltonian homeomorphism $f$ of a compact surface $\Sigma$ with a finite number of fixed points. Notice that this definition extends the notion of action function defined for Hamiltonian diffeomprhisms. \\

Let us consider two unlinked fixed points $x,y\in\Fix(f)$ of $f$ and an isotopy $I=(f_t)_{t\in[0,1]}$ from $\id$ to $f$ such that $x,y\in \Sing(I)$.
Let $\gamma$ be a simple path from $x$ to $y$ and define the map $\rho_{f,\gamma}$ on $\Sigma$ by $\rho_{f,\gamma}(z)=\gamma\wedge I(z)$ where $I(z)$ is the trajectory of $z$ under the isotopy $I$ and $\gamma\wedge I(z)$ is the intersection number between $\gamma$ and $I(z)$. We define the difference of action between $y$ and $x$ by

\begin{equation}\label{action}
A_f(x)-A_f(y)= \int_\Sigma \rho_{f,\gamma}(z)dz,
\end{equation}
which does not depend on the choice of $\gamma$. Notice that in general for a homeomorphism $f$, the map $\rho_{f,\gamma}$ is not integrable. In our case, $f$ admits a finite number of fixed points and one may prove that the previous integral exists, see \cite{CAL4}.\\

Unfortunately, if we consider two fixed points $x,y\in\Fix(f)$ they may not be unlinked and the previous arguments fail and to define the action difference between $y$ and $x$ we have to consider the universal cover of $\Sigma$.  We refer to Wang's work \cite{WANG} for more details about the general definition.\\

\textbf{Positively transverse foliations} \\

Let us consider an oriented topological foliation $\FF$ on the complement of a compact set $X$ of a surface $\Sigma$. The set $X$ will be called the set of \textit{singularities} of $\FF$. An open flow box of $\FF$ is a couple $(V,h)$, where $V$ is an open set of $\Sigma$ and $h:V\ra (-1,1)^2$ is an orientation-preserving homeomorphism that sends the foliation $\FF|_V$ on the vertical foliation oriented with $y$ decreasing. Writing $p_1: \R^2\ra \R$ for the first projection, we say that an arc $\gamma:I\ra \Sigma$ is \emph{positively transverse} to the foliation $\FF$ if for every $t_0\in I$, there exists an open flow-box $(V,h)$ such that $\gamma(t_0)\in V$ and the map $t\mapsto p_1(h(\gamma(t)))$ defined in a neighborhood of $t_0$ is strictly increasing. We draw in Figure \ref{flow} an example of a flow box. \\

 \begin{figure}[H]
 \begin{center}
\begin{tikzpicture}[scale=0.8]

\draw [decoration={markings, mark=at position 0.75 with {\arrow{>}}}, postaction={decorate}] (-2,0) -- (-2,-3);
\draw [decoration={markings, mark=at position 0.75 with {\arrow{>}}}, postaction={decorate}] (-1.8,0) -- (-1.8,-3);
\draw [decoration={markings, mark=at position 0.75 with {\arrow{>}}}, postaction={decorate}] (-1.6,0) -- (-1.6,-3);
\draw [decoration={markings, mark=at position 0.75 with {\arrow{>}}}, postaction={decorate}] (-1.4,0) -- (-1.4,-3);
\draw [decoration={markings, mark=at position 0.75 with {\arrow{>}}}, postaction={decorate}] (-1.2,0) -- (-1.2,-3);
\draw [decoration={markings, mark=at position 0.75 with {\arrow{>}}}, postaction={decorate}] (-1,0) -- (-1,-3);
\draw [decoration={markings, mark=at position 0.75 with {\arrow{>}}}, postaction={decorate}] (-0.8,0) -- (-0.8,-3);
\draw [decoration={markings, mark=at position 0.75 with {\arrow{>}}}, postaction={decorate}] (-0.6,0) -- (-0.6,-3);
\draw [decoration={markings, mark=at position 0.75 with {\arrow{>}}}, postaction={decorate}] (-0.4,0) -- (-0.4,-3);
\draw [decoration={markings, mark=at position 0.75 with {\arrow{>}}}, postaction={decorate}] (-0.2,0) -- (-0.2,-3);
\draw [decoration={markings, mark=at position 0.75 with {\arrow{>}}}, postaction={decorate}] (0,0) -- (0,-3);
\draw [decoration={markings, mark=at position 0.75 with {\arrow{>}}}, postaction={decorate}] (1,0) -- (1,-3);
\draw [decoration={markings, mark=at position 0.75 with {\arrow{>}}}, postaction={decorate}] (0.8,0) -- (0.8,-3);
\draw [decoration={markings, mark=at position 0.75 with {\arrow{>}}}, postaction={decorate}] (0.6,0) -- (0.6,-3);
\draw [decoration={markings, mark=at position 0.75 with {\arrow{>}}}, postaction={decorate}] (0.4,0) -- (0.4,-3);
\draw [decoration={markings, mark=at position 0.75 with {\arrow{>}}}, postaction={decorate}] (0.2,0) -- (0.2,-3);

\draw[left,red] (-2,-2) node {$\gamma$};
\draw[->,color=red] (-2,-2) ..controls +(5,0.2) and +(-6,0.2).. (1,-1);

\draw[->] (1.5,-1.5) -- (2.5,-1.5);

\draw [decoration={markings, mark=at position 0.75 with {\arrow{>}}}, postaction={decorate}] (3,0) -- (3,-3);
\draw [decoration={markings, mark=at position 0.75 with {\arrow{>}}}, postaction={decorate}] (3.2,0) -- (3.2,-3);
\draw [decoration={markings, mark=at position 0.75 with {\arrow{>}}}, postaction={decorate}] (3.4,0) -- (3.4,-3);
\draw [decoration={markings, mark=at position 0.75 with {\arrow{>}}}, postaction={decorate}] (3.6,0) -- (3.6,-3);
\draw [decoration={markings, mark=at position 0.75 with {\arrow{>}}}, postaction={decorate}] (3.8,0) -- (3.8,-3);
\draw [decoration={markings, mark=at position 0.75 with {\arrow{>}}}, postaction={decorate}] (4,0) -- (4,-3);
\draw [decoration={markings, mark=at position 0.75 with {\arrow{>}}}, postaction={decorate}] (4.2,0) -- (4.2,-3);
\draw [decoration={markings, mark=at position 0.75 with {\arrow{>}}}, postaction={decorate}] (4.4,0) -- (4.4,-3);
\draw [decoration={markings, mark=at position 0.75 with {\arrow{>}}}, postaction={decorate}] (4.6,0) -- (4.6,-3);
\draw [decoration={markings, mark=at position 0.75 with {\arrow{>}}}, postaction={decorate}] (4.8,0) -- (4.8,-3);
\draw [decoration={markings, mark=at position 0.75 with {\arrow{>}}}, postaction={decorate}] (5,0) -- (5,-3);
\draw [decoration={markings, mark=at position 0.75 with {\arrow{>}}}, postaction={decorate}] (5.2,0) -- (5.2,-3);
\draw [decoration={markings, mark=at position 0.75 with {\arrow{>}}}, postaction={decorate}] (5.4,0) -- (5.4,-3);
\draw [decoration={markings, mark=at position 0.75 with {\arrow{>}}}, postaction={decorate}] (5.6,0) -- (5.6,-3);
\draw [decoration={markings, mark=at position 0.75 with {\arrow{>}}}, postaction={decorate}] (5.8,0) -- (5.8,-3);
\draw [decoration={markings, mark=at position 0.75 with {\arrow{>}}}, postaction={decorate}] (6,0) -- (6,-3);

\draw[->,color=red] (3,-2) ..controls +(1,0.2) and +(-0.5,0).. (6,-1);

\end{tikzpicture}
\end{center}
\caption{ }
\label{flow}
\end{figure}

For $z\in\Sigma$, we write $\phi_z$ the leaf passing through $z$ and $\phi_z^+$ for the positive half-leaf from $z$. We consider an isolated singularity $x$ of the foliation $\FF$, we can define the index $\ind(\FF,x)$ of $x$ for the foliation $\FF$ as follows. We consider a sufficiently small open chart $U$ containing $x$ and an orientation preserving homeomorphism $h: U\ra D\backslash\{0\}$ which sends $x$ to $0$. We denote $\FF_h$ the image of the foliation $\FF|_U$ by $h$ and we consider a simple loop $\Gamma: \Ss^1 \ra D\backslash\{x\}$, one may cover $\Gamma$ by a finite family $(V_i)_{i\in J}$ of flow-boxes of the foliation $\FF_h$ included in $D\backslash \{0\}$. We denote by $\phi^+_{V_i,z}$ the positive half-leaf from $z$ of the restricted foliation $\FF_h|_{V_i}$. We can find a continuous map $\psi$ defined from the loop $\Gamma$ to $D_1\backslash\{x\}$ such that $\psi(z)\in\phi_{V_i,z}^+$ for every $i\in J$ and any $z\in V_i$. The map 
$$\theta\ra \frac{\psi(\Gamma(\theta))-\Gamma(\theta)}{||\psi(\Gamma(\theta))-\Gamma(\theta)||},$$
is well defined on $\Gamma$ and $\ind(\FF,x)$ is the degree of this map.\\

We say that a singularity $x$ of an oriented foliation $\FF$ is a sink (resp. source) if there is a neighborhood $V$ of $x$ such that the omega-limit point (resp. the alpha-limit point) of each leaf $\phi$ of $\FF$ which is passing through $V$ is equal to $x$. The sinks and sources of a foliation $\FF$ have an index equal to $1$.\\

We say that a singularity $x$ of an oriented foliation $\FF$ is a saddle point of the foliation $\FF$ if the foliation is locally homeomorphic to a foliation as the one on Figure \ref{saddlesaddle}, we refer to \cite{ROUX1} for more details on saddle points. A maximal connected union of leaves such that their alpha (resp. omega)  limit is equal to $x$ is called an unstable (resp. stable) cone of $x$. A saddle point has $1-\ind(\FF,x)$ unstable cones and three stable cones which are alternated in the cyclic order. In Figure \ref{saddlesaddle} we draw an example of a foliation near a saddle point of index $-2$. 

\begin{figure}[htp]
\begin{center}
\begin{tikzpicture}[scale=0.5]

\draw [decoration={markings, mark=at position 0.5 with {\arrow{<}}}, postaction={decorate},red] (0,-3) -- (0,0);
\draw [decoration={markings, mark=at position 0.5 with {\arrow{<}}}, postaction={decorate},red] (0.3,-3) -- (0,0);
\draw [decoration={markings, mark=at position 0.5 with {\arrow{<}}}, postaction={decorate},red] (-0.3,-3) -- (0,0);

\draw [decoration={markings, mark=at position 0.5 with {\arrow{>}}}, postaction={decorate},blue] (0,3) -- (0,0);
\draw [decoration={markings, mark=at position 0.5 with {\arrow{>}}}, postaction={decorate},blue] (0.3,3) -- (0,0);
\draw [decoration={markings, mark=at position 0.5 with {\arrow{>}}}, postaction={decorate},blue] (-0.3,3) -- (0,0);

\draw [decoration={markings, mark=at position 0.5 with {\arrow{<}}}, postaction={decorate},blue] (0,0) -- (-3,-3);
\draw [decoration={markings, mark=at position 0.5 with {\arrow{<}}}, postaction={decorate},blue] (0,0) -- (3,-3);

\draw [decoration={markings, mark=at position 0.5 with {\arrow{<}}}, postaction={decorate},red] (-3,0.3) -- (0,0);
\draw [decoration={markings, mark=at position 0.5 with {\arrow{<}}}, postaction={decorate},red] (-3,0) -- (0,0);
\draw [decoration={markings, mark=at position 0.5 with {\arrow{<}}}, postaction={decorate},red] (-3,-0.3) -- (0,0);

\draw [decoration={markings, mark=at position 0.5 with {\arrow{<}}}, postaction={decorate},red] (3,0) -- (0,0);

\draw [decoration={markings, mark=at position 0.5 with {\arrow{>}}}, postaction={decorate}]  (0.5,3) ..controls +(-0.4,-3) and +(-3,0).. (3,0.2)  ;

\draw [decoration={markings, mark=at position 0.5 with {\arrow{<}}}, postaction={decorate}] (-3,0.5) ..controls +(3,-0.4) and +(0.4,-3).. (-0.5,3)  ;

\draw [decoration={markings, mark=at position 0.5 with {\arrow{<}}}, postaction={decorate}] (-3,-0.5) ..controls +(3,0.4) and +(3,3).. (-3,-2.7)  ;
\draw [decoration={markings, mark=at position 0.5 with {\arrow{<}}}, postaction={decorate}] (-0.5,-3) ..controls +(0.4,3) and +(3,3).. (-2.7,-3)  ;
\draw [decoration={markings, mark=at position 0.5 with {\arrow{<}}}, postaction={decorate}] (3,-0.2) ..controls +(-3,0) and +(-3,3).. (3,-2.7)  ;
\draw [decoration={markings, mark=at position 0.5 with {\arrow{<}}}, postaction={decorate}] (0.5,-3) ..controls +(-0.4,3) and +(-3,3).. (2.7,-3)  ;
\end{tikzpicture}
\end{center}
\caption{ }
\label{saddlesaddle}
\end{figure}

A leaf of an oriented foliation $\FF$ whose alpha-limit point and omega-limit point are distinct singularities of $\FF$ will be called a \textit{connexion}.\\

Let us consider $f\in \mathrm{Homeo}_0(\Sigma)$ and a maximal isotopy $I=(f_t)_{t\in[0,1]}$ from the identity to $f$. A foliation $\FF$ is said to be \textit{positively transverse} to the isotopy $I$ if $\Sing(I)=\Sing(\FF)$ and for every $z\in \mathrm{Dom}(I)$, the trajectory $\gamma(z)$ of $z$ is homotopic in $\mathrm{Dom}(I)$, relatively to its endpoints, to a path $\gamma$ which is positively transverse to the foliation $\FF$. The following fondamental result of Le Calvez \cite{CAL4} asserts that for each maximal isotopy $I$ there exists a positively transverse foliation to the isotopy $I$.

\begin{theorem}    
Let us consider a homeomorphism $f\in \mathrm{Homeo}_0(\Sigma)$ and an isotopy $I=(f_t)_{t\in[0,1]}$ from $\id$ to $f$, such that $\Sing(I)$ is a maximal unlinked set of fixed points of $f$. There exists a foliation $\FF$ which is positively transverse to the isotopy $I$.\\
\end{theorem}

We denote by $\FF(I)$ the set of foliations positively transverse to $I$. \\

\textbf{Gradient-like foliations}\\

We will use the following definition of gradient-like foliations.\\

\begin{definition}\label{gradient-like}
A foliation $\FF$ is said to be \emph{gradient-like} if
\begin{itemize}
\item The number of singularities is finite.
\item Every leaf defines a connexion.
\item There is no closed leaf.
\item There is no family $(\phi_i)_{i\in \Z/k\Z}$, $k\geq 1$ of leaves such that $\omega(\phi_i)=\alpha(\phi_{i+1})$, $i\in\Z/k\Z$.
\end{itemize}
\end{definition}

For the remaining, most of the positively transverse foliations we will meet will be gradient-like foliations. The notion of connexion will be generalized in section \ref{order}, but until this section, a connexion will always refer to a leaf of a gradient-like foliation. \\

We refer to \cite{CAL4} for the proof of the following important properties.

\begin{prop}\label{gradient-like}
Consider a Hamiltonian homeomorphism of a surface $\Sigma$ with a finite number of fixed points then for each maximal isotopy $I$ from $\id$ to $f$, a foliation $\FF$ positively transverse to $I$ is gradient-like. Moreover we have
\begin{itemize}
\item $\ind(\FF,x)\leq 1$ for every point $x\in \Sing(I)$. 
\item $\ind(\FF,x)=1$ for every sink or source $x\in \Sing(I)$
\item $\ind(\FF,x)=\ind(f,x)$ for every saddle point $x\in \Sing(I)$, where $\ind(f,\cdot)$ is the Lefschetz index.
\item For every leaf $\phi\in\FF$, the action function $A_f$ of $f$ satisfies $A_f(\alpha(\phi))>A_f(\omega(\phi))$.
\end{itemize}
\end{prop}

For the remaining, we can keep in mind that, for a gradient-like foliation, there are three kinds of singularities: sinks, sources and saddle points. \\

\begin{rem}
For a maximal isotopy $I$ of a Hamiltonian homeomorphism $f$ of a surface $\Sigma$ and a foliation $\FF\in\FF(I)$, if $\Sigma$ is not the sphere, then the index function $\ind(\FF,\cdot)$ defined on $\Sing(I)$ does not depend on the choice of $\FF\in \FF(I)$ and can be denoted $\ind(I,\cdot)$.
\end{rem}

Let us consider a gradient-like foliation $\FF$ of a surface $\Sigma$ and a leaf $\phi$ of $\FF$. By definition, the omega-limit set (resp. the alpha-limit set) of $\phi$ exists and is equal to a singleton $\{x\}$. To simplify the notations, $x$ will be called \emph{the omega-limit point} also denoted $\omega(\phi)$ (resp. \emph{the alpha-limit point} also denoted $\alpha(\phi)$) of $\phi$.\\

\begin{rem}
An autonomous Hamiltonian function $H$ on a surface $\Sigma$ induces a Hamiltonian isotopy $I=(f_t)_{\tin}$ from $\id$ to a Hamiltonian diffeomorphism $f$ which is a maximal isotopy for $f$.  Moreocer, the gradient lines of $H$ induce a gradient-like foliation $\FF$ which is positively transverse to $I$. 
\end{rem}

\textbf{Action function along a leaf of a transverse foliation}\\

We consider a maximal isotopy $I=(f_t)_{\tin}$ from $\id$ to $f$ and a foliation $\FF$ positively transverse to $I$ we can give a short proof of the last point of Proposition \ref{gradient-like} stated as the following lemma.

\begin{lem}\label{action-decrease}
For every leaf $\phi\in\FF$ we have $A_f(\alpha(\phi))>A_f(\omega(\phi))$.
\end{lem}

\begin{proof}
We set $x=\alpha(\phi)$ and $y=\omega(\phi)$. In this particular case, it is enough to prove that the difference of action function given by equation \ref{action} is positive. We give the ideas of the proof using Wang's work.\\

Let us consider a small open disk $U\subset \Sigma\backslash X$. For almost every point $z\in U$ we can define $\tau(z)$ the first return map. Meaning that $\tau(z)$ is the first integer $n>0$ such that $f^n(z)\in U$. We consider the loop $\Gamma(z)= I^{\tau(z)-1}(z) \gamma(z)$ where $\gamma(x)\subset U$ is a path which joins $f^{\tau(z)}(z)$ to $z$. The algebraic intersection of $\Gamma(z)$ and $\phi$ does not depend on the choice of $\gamma$ and is well defined. We will denote $\delta(z)=\Gamma(z)\wedge \phi$ the algebraic intersection number defined on $U$. \\

There is a finite number of fixed points, so Wang, \cite{WANG}, proved that $\delta/\tau$ is bounded and then integrable. We can define the limit of Birkhoff's average functions $\delta^*$ and $\tau^*$ of $\delta$ and $\tau$. One may prove that the function $\eta= \delta^*/\tau^*$ is defined almost everywhere on $U$ and does not depend on the choice of $U$. One may prove that we obtain a function defined almost everywhere on $\Sigma\backslash X$ such that its integral, by Birkhoff's theorem, is equal to the action difference between $x$ and $y$ given by equation \ref{action}. Thus, it is enough to prove that $\delta(z)$ is positive and not zero to obtain the result.\\

We consider the universal cover $\ti{\Sigma\backslash X}$ of $\Sigma\backslash X$, which can be identify as the open disk $\D$ and we denote by $\ti f$ the lift of $f$.

We fix a lift $\ti \phi$ of $\phi$ on $\D$, hence $\delta(z)$ is equal to the finite sum of the algebraic intersection numbers of $\ti\phi$ and the lifts of $\Gamma(z)$. Let us consider a lift $\ti\Gamma(z)$ of $\Gamma(z)$ whose algebraic intersection number with $\ti\phi$ is not zero. Roughly speaking, $\ti\Gamma(z)$ is an oriented path going from one side of $\D\backslash \ti\phi$ to another. Moreover, $\ti \phi$ is a Brouwer line for $\ti f$, so $\ti\Gamma$ is going from the right hand side of $\ti\phi$ to the left hand side of $\ti\phi$. Hence the intersection number $\ti \Gamma(z)\wedge \ti \phi>0$ and then we have $\delta(z)>0$. \\
\end{proof}

		\section{The simplest case of barcode for Hamiltonian homeomorphisms}\label{génériqueF}
		
In this section, we consider a Hamiltonian homeomorphism $f$ on a closed surface $\Sigma$ which satisfies the following assumptions. 

\begin{enumerate}
\item The set of fixed points is finite and unlinked, in particular every fixed point is contractible.
\item The fixed points have distinct action values.
\item For every $x\in \Fix(f)$, $\ind(f,x)$ is either $1$ or $-1$.
\end{enumerate}

Let $I=(f_t)_{t\in[0,1]}$ be a maximal isotopy from identity to $f$. Let $\FF\in \FF(I)$ be a positively transverse foliation associated to $I$ which satisfies the following assumptions.
\begin{enumerate}
\item There is no leaf joining two saddles points.
\item For every saddle point $x\in \Fix(f)$, there are exactly two unstable cones composed of one leaf whose alpha-limit point is $x$ and two stable cones composed of one leaf whose omega-limit point is $x$.
\end{enumerate}

We denote $A_f$ the action functional of $f$ defined on $\Fix(f)$. \\

We have the following result from \cite{CAL4}.

\begin{lem}
For every $x\in\Fix(f)$ and every $\FF\in\FF(I)$ we have $\ind(f,x)=\ind(\FF,x)$. So if $\FF\in\FF_{\mathrm{gen}}(I)$ then we have
\begin{itemize}
\item $\ind(f,x)=1$ if $x$ is a source or a sink of $\FF$.
\item $\ind(f,x)=-1$ if $x$ is a saddle pont of $\FF$.
\end{itemize}
\end{lem}

For the remainder of the section, we consider a foliation $\FF\in\FF_{\gen}(f)$. Recall that the foliation $\FF$ is gradient-like and we will use the analogy between Morse Theory and gradient-like foliations to construct a filtered homology from the foliation $\FF$. We define a graph associated to the foliation $\FF$ and we associate to this graph a chain complex in order to compute its filtered homology and obtain a persistence module.\\

Remember that for a fixed point $x$ of $f$, being a sink or a source of the foliation $\FF$ does not depend on the choice of $\FF\in\FF(I)$. We define the index $\ind_{CZ}(f,\cdot)$ on the set of fixed points of $f$ as follows. For $x\in\Fix(f)$ we set 
\begin{itemize}
\item $\ind_{CZ}(f,x)=0$ if $x$ is a sink of $\FF$,
\item $\ind_{CZ}(f,x)=1$ if $x$ is a saddle point of $\FF$,
\item $\ind_{CZ}(f,x)=2$ if $x$ is a source of $\FF$.
\end{itemize}

The notation $\ind_{CZ}$ of the index function refers to the Conley-Zehnder index function as they are equals under these assumptions. 

\begin{definition}\label{graphgeneric}
Let $G(\FF)$ be the graph whose set of vertices is the set $\Fix(f)$ and whose set of edges corresponds to the set of leaves $\phi$ of $\FF$ such that $\ind_{CZ}(f,\alpha(\phi))=\ind_{CZ}(f,\omega(\phi))-1$.
\end{definition}

For $i\in\N$ we consider the set  $\Fix_i(f)$ of fixed points $x\in\Fix(f)$ which satisfy $\ind_{CZ}(f,x)=i$. Note that $\Fix_i(f)=\emptyset$ if $i\geq 3$. We define a chain complex associated to the graph $G(\FF)$ following the ideas from Morse homology .\\

\begin{definition}\label{complexchain}
For $t\in\R$ and $i\in \N$, we define the chain complex 
$$C_i^t= \bigoplus_{\substack{z\in \Fix_i (f)\\ A_f(z)< t}}  \Z/2\Z \cdot z,$$ 
and the maps $\partial_i^t: C_i^t \ra C_{i-1}^t$ such that for every $z\in C_i^t$ 
$$\partial_i^t(z)=\sum_{b\in\Fix_{i-1}(f)} n(z,b)b,$$
where $n(z,b)$ is the number modulo $2$ of edges from $z$ to $b$ in $G(\FF)$. If $i$ is distinct from $1$ or $2$ then $\partial_i^t $ is equal to $0$ for every $t\in\R$.
\end{definition}

\begin{rem}
For a fixed point $z\in\Fix(f)$, if there exists an edge from $z$ to $b$ in $G(\FF)$ then by Proposition \ref{gradient-like} we have $A_f(z)>A_f(b)$ and for every $t>A_f(z)$, the element $\partial_i^t(z)$ belongs to $C_{ i-1}^t$. So the map $\partial_i^t$ is well defined.
\end{rem}

We obtain that, for every $t\in\R$, $(C_i^t,\partial_i^t)$ is a chain complex thanks to the following property.\\

\begin{prop}\label{partialnul}
For each $t\in\R$ and every $i\in\N$ the maps $\partial_i^t$ satisfy $\partial_i^t\circ \partial_{i+1}^t=0$.
\end{prop}

We prove Proposition \ref{partialnul} after the following definition and lemma.

\begin{definition}
Let $x\in\Fix(f)$ be a source of the foliation $\FF$, the subset $\bigcup \{ \phi\in\FF \ | \ \alpha(\phi)= x\}\cup\{x\}$ of $\Sigma$ will be called the repulsive basin of $x$ and denoted $W^u(x)$.
\end{definition}

For a source $x$ of $\FF$, we want to describe $W^u(x)$. Let $P_n\subset \C$ be the filled regular polygon of vertices $\e^{\frac{i\pi k}{n}}$, with $k\in \{0,...,2n-1\}$. We have the following lemma.\\

\begin{lem}\label{bordbassin}
Let $x\in\Fix(f)$ be a source of the foliation $\FF$. There exist $n\geq 1$ and a continuous map $d:P_n\ra \Sigma$ such that 
\begin{itemize}
\item $d(\inte(P_n))$ is the repulsive basin of $x$.
\item $d(\e^{\frac{i\pi k}{2n}})$ is a sink of $\FF$ if $k$ is even and a saddle point of $\FF$ if $k$ is odd.
\item The image of a side of $P_n$ is the closure of a leaf of $\FF$.
\end{itemize}
\end{lem}

\begin{proof}[Proof of Lemma \ref{bordbassin}]
Let us consider a source $x\in\Fix(f)$ of $\FF$. There exists a homeomorphism $h:\D\ra W^u(x)$ such that $h(0)=x$ and such that the leaves from $x$ are the images by $h$ of the segments $t\e^{i\theta}$, $\tin$. For $\theta\in [0,2\pi)$ we will denote $\phi_\theta$ the image by $h$ of the segment $t\e^{i\theta}$, $\tin$.\\

There are a finite number of angles $(\theta_k)_{k\in\Z/n\Z}$ such that the omega-limit point of $\phi_{\theta_k}$ is a saddle point $x_k$ of $\FF$.\\

Moreover the \textit{attractive basin} of a sink $x$ of a foliation is the union of $x$ and the  leaves whose omega-limit point is equal to $x$. The attractive basin of a sink is an open set. So, by connectedness, for every $k\in\Z/n\Z$ there exists a sink of $\FF$, denoted  $s_k$, such that for every leaf $\phi_\theta$ of angle $\theta\in (\theta_k,\theta_{k+1})$, the omega-limit point of $\phi_\theta$ is equal to $s_k$. We denote $U_k$ the union of the leaves $\phi_\theta$,  with $\theta\in (\theta_k,\theta_{k+1})$. We draw an example of such a set in Figure \ref{U}. The set $U_k$ is a topological open disk on $\Sigma$ whose boundary is the closure of four distinct leaves of $\FF$: the leaves $\phi_{\theta_k}$ and $\phi_{\theta_{k+1}}$, a leaf $\psi_k$ from $x_k$ to $s_k$ and a leaf $\phi_k$ from $x_{k+1}$ to $s_k$.  The existence of the leaves $\phi_k$ and $\psi_k$ is deduced from the dynamic of the foliation near the saddle points $x_k$ and $x_{k+1}$.

\begin{figure}[h]
\begin{center}
\begin{tikzpicture}
\newdimen\r
   \r=2.7cm
   \draw[red,decoration={markings, mark=at position 0.5 with {\arrow{>}}}, postaction={decorate}]  (0:\r)  -- (60:\r) ;
   \draw[red,decoration={markings, mark=at position 0.5 with {\arrow{<}}}, postaction={decorate}]  (60:\r)  -- (120:\r) ;
    
    \draw[red,decoration={markings, mark=at position 0.5 with {\arrow{<}}}, postaction={decorate}]  (0:\r)  -- (0,0) ;
   \draw[red,decoration={markings, mark=at position 0.5 with {\arrow{<}}}, postaction={decorate}]  (60:\r)  -- (0,0) ;
    \draw[red, decoration={markings, mark=at position 0.5 with {\arrow{<}}}, postaction={decorate}]  (120:\r)  -- (0,0) ; 
   \draw[red] (1.6,1) node{$U_k$};

   \draw[below] (0,-0.2) node{$x$};
   \draw (0,2.8) node{$\phi_k$};
   \draw (2.3,1.5) node{$\psi_k$};
    \draw[below] (1.5,0) node{$\phi_{\theta_k}$};
    \draw(-1,1) node{$\phi_{\theta_{k+1}}$};
    
   \foreach \x/\l/\p in
     { 60/{$s_k$}/above,
      120/{$x_{k+1}$}/above,
      360/{$x_k$}/right
     }
     \node[inner sep=1pt,circle,draw,fill,label={\p:\l}] at (\x:\r) {};
     
     \draw[red,decoration={markings, mark=at position 0.5 with {\arrow{>}}}, postaction={decorate}]  (0,0) ..controls +(2.5,0.2)  and +(1,-2).. (60:\r) ; 
     \draw[red,decoration={markings, mark=at position 0.5 with {\arrow{>}}}, postaction={decorate}]  (0,0) ..controls +(1.5,1)  and +(0,-1).. (60:\r) ; 
     
     \draw[red,decoration={markings, mark=at position 0.5 with {\arrow{>}}}, postaction={decorate}]  (0,0) ..controls +(-1,2.5)  and +(-2,-0.5).. (60:\r) ; 
     \draw[red,decoration={markings, mark=at position 0.5 with {\arrow{>}}}, postaction={decorate}]  (0,0) ..controls +(0,0.8)  and +(-1.5,-1).. (60:\r) ; 
\end{tikzpicture}
\end{center}
\caption{ }
\label{U}
\end{figure}

We obtain that the repulsive basin of $x$ is equal to the union $[\phi_{\theta_0}\cup U_0 \cup \phi_{\theta_1}\cup U_1 \cup ... \cup U_{n-1}]\cup \{x\}$. We define the map $d:P_n\ra \Sigma$ given by Lemma \ref{U} as follows. For every $k\in\{0,...,n-1\}$ we set 
\begin{itemize}
\item $d(\e^{\frac{i2k\pi}{2n}})=x_k$,
\item $d(\e^{\frac{i(2k+1)\pi}{2n}})=s_k$.
\end{itemize} 
The map $d$ naturally extends to $\partial P_n$ by sending the edges of the boundary of $\partial P_n$ alternatively, in cyclic order, to the leaves $\phi_k$, $k\in\Z/n\Z$ and the leaves $\psi_k$, $k\in\Z/n\Z$. Finally, the map $d$ extends naturally on the interior of $P_n$ as follows.\\

For $k\in\Z/n\Z$, we consider the slice $S_k$ of the polygon $P_n$ defined as the set of points of $P_n$ whose angle $\theta$ in polar coordinates satisfies $\theta\in [\frac{ik\pi}{n},\frac{i(k+1)\pi }{n}]$.  We extend $d$ by sending the slice $S_k$, $k\in\Z/n\Z$, of $P_n$ on the closure of the set $U_k$ defined previously. The map $d$ is well-defined and continuous. \\

Let us draw a repulsive basin of a source $x$ of the foliation $\FF$ in Figure \ref{basin-sink}. We represent the leaves of $U_0$ and its boundary in red in Figure \ref{basin-sink}.

\begin{figure}[h]
\begin{center}
\begin{tikzpicture}[scale=0.7]
\newdimen\r
   \r=2.7cm
   \draw[red,decoration={markings, mark=at position 0.5 with {\arrow{>}}}, postaction={decorate}]  (0:\r)  -- (60:\r) ;
   \draw[red,decoration={markings, mark=at position 0.5 with {\arrow{<}}}, postaction={decorate}]  (60:\r)  -- (120:\r) ;
   \draw[ decoration={markings, mark=at position 0.5 with {\arrow{>}}}, postaction={decorate}]  (120:\r)  -- (180:\r) ;
   \draw[ decoration={markings, mark=at position 0.5 with {\arrow{<}}}, postaction={decorate}]  (180:\r)  -- (240:\r) ;
   \draw[ decoration={markings, mark=at position 0.5 with {\arrow{>}}}, postaction={decorate}]  (240:\r)  -- (300:\r) ;
   \draw[ decoration={markings, mark=at position 0.5 with {\arrow{<}}}, postaction={decorate}]  (300:\r)  -- (360:\r) ;
   
    \draw[red,decoration={markings, mark=at position 0.5 with {\arrow{<}}}, postaction={decorate}]  (0:\r)  -- (0,0) ;
   \draw[red,decoration={markings, mark=at position 0.5 with {\arrow{<}}}, postaction={decorate}]  (60:\r)  -- (0,0) ;
   \draw[red, decoration={markings, mark=at position 0.5 with {\arrow{<}}}, postaction={decorate}]  (120:\r)  -- (0,0) ;
   \draw[ decoration={markings, mark=at position 0.5 with {\arrow{<}}}, postaction={decorate}]  (180:\r)  -- (0,0) ;
   \draw[ decoration={markings, mark=at position 0.5 with {\arrow{<}}}, postaction={decorate}]  (240:\r)  -- (0,0) ;
   \draw[ decoration={markings, mark=at position 0.5 with {\arrow{<}}}, postaction={decorate}]  (300:\r)  -- (0,0) ;
   
   \draw (-1.3,1) node{$U_1$};
   \draw[red] (1.6,1) node{$U_0$};

   \draw[below] (0,-0.2) node{$x$};
   \draw (0,2.8) node{$\phi_0$};
   \draw (2.3,1.5) node{$\psi_0$};
    \draw (-2.3,1.5) node{$\psi_1$};
   \foreach \x/\l/\p in
     { 60/{$s_0$}/above,
      120/{$x_1$}/above,
      180/{ }/left,
      240/{ }/below,
      300/{ }/below,
      360/{$x_0$}/right
     }
     \node[inner sep=1pt,circle,draw,fill,label={\p:\l}] at (\x:\r) {};
     
     \draw[red,decoration={markings, mark=at position 0.5 with {\arrow{>}}}, postaction={decorate}]  (0,0) ..controls +(2.5,0.2)  and +(1,-2).. (60:\r) ; 
     \draw[red,decoration={markings, mark=at position 0.5 with {\arrow{>}}}, postaction={decorate}]  (0,0) ..controls +(1.5,1)  and +(0,-1).. (60:\r) ; 
     
     \draw[red,decoration={markings, mark=at position 0.5 with {\arrow{>}}}, postaction={decorate}]  (0,0) ..controls +(-1,2.5)  and +(-2,-0.5).. (60:\r) ; 
     \draw[red,decoration={markings, mark=at position 0.5 with {\arrow{>}}}, postaction={decorate}]  (0,0) ..controls +(0,0.8)  and +(-1.5,-1).. (60:\r) ; 
\end{tikzpicture}
\end{center}
\caption{ }
\label{basin-sink}
\end{figure}
\end{proof}

\begin{rem}
The map $d$ may not be injective. 
\end{rem}

Now, we can give the proof of proposition \ref{partialnul}.

\begin{proof}[Proof of proposition \ref{partialnul}]
We consider $t\in\R$  since $\partial_i^t$ is $0$ for $i$ distinct from $1$ or $2$, it is enough to prove that for every source $s\in\Fix(f)$ we have $\partial_1^t\circ \partial_{2}^t(s)=0$.\\
Let $x\in\Fix(f)$ be a source of the foliation $\FF$. Using the same notations of the proof of Lemma \ref{bordbassin}, there exists an integer $n>0$ and $n$ leaves which were denoted $(\phi_{\theta_k})_{k\in\Z/n\Z}$ whose omega-limit points, denoted $(x_k)_{k\in\Z/n\Z}$ are exactly the saddle points of the foliation $\FF$ which are connected to $x$. So we have $\partial_{2}^t(s)=\sum_{k=0}^{n-1} \omega(\phi_{\theta_k})$. \\

Moreover, for every $k\in\Z/n\Z$ the leaves $\phi_{k-1}$ and $\psi_k$ of the proof of Lemma \ref{bordbassin} are exactly the leaves of the foliation $\FF$ whose alpha-limit point is $x_k$. So we can compute

\begin{align*}
\partial_1^t\circ \partial_{2}^t(s) &= \partial_1^t(\sum_{k=0}^{n-1}\omega(\phi_{\theta_k})), \\
 &= \sum_{k=0}^{n-1}\partial_1^t(\omega(\phi_{\theta_k})),\\
 &=  \sum_{k=0}^{n-1} \left( \omega(\phi_{k-1})+\omega(\psi_{k}) \right),\\
 &= \sum _{k=0}^{n-1} \omega(\phi_{k})+\omega(\psi_{k}),\\
 &=\sum _{k=0}^{n-1}  2 s_k,\\
 &=0.
\end{align*}

Hence we obtain the result of Proposition \ref{partialnul}. 
\end{proof}

\begin{definition}
The image of the persistence module $H_*((C_i^t,\partial^t_i)_{i,t})$ under the functor $\beta$  is called the barcode of $f$ for the foliation $\FF$ and we will denote it $B_{\gen}(f,\FF)$.
\end{definition}

\begin{rem}\label{bargeneric}
For a foliation $\FF\in\FF_{\gen}(I)$, each value $b$ of the action function $A_f$ is the end of a unique bar of the barcode $B_{\gen}(f,\FF)$.
\end{rem}

\begin{rem}[Similarities with the autonomous example] \label{MorseExample}
Let us consider an autonomous Hamiltonian function $H$ on the $2$-sphere. The function $H$ induces a Hamiltonian isotopy $I=(f_t)_{\tin}$. Let us suppose that the set of fixed points of the time one map $f$ of $I$ is equal to the set of critical points of $H$. Hence, in this case $\Fix(f)$ is unlinked. If we consider an adapted Riemannian structure on $\Sigma$, the gradient lines of $H$ induced by the Riemannian metric defines a foliation $\FF$ positively transverse to $I$. Moreover the action function $A_f$ is given by $A_f(x)=H(x)$ for every $x\in \Sing(\FF)$. \\
In this example $f$ has six fixed points, two sinks $p_1,p_2$ two saddle points $x_1,x_1$ and two sources $s_1,s_2$. We draw the graph $G(\FF)$ on the left side of the figure and the barcode $B_{\gen}(f,\FF)$, as intervals of $\R$, on the right side.

\begin{figure}[h]
\begin{center}
\begin{tikzpicture}[scale=0.7]
\draw[green,decoration={markings, mark=at position 0.5 with {\arrow{>}}}, postaction={decorate}] (0,-1) ..controls +(0.5,0) and +(-0.5,0).. (1,-2);
\draw[green,decoration={markings, mark=at position 0.5 with {\arrow{<}}}, postaction={decorate}] (1,-2) ..controls +(1,0) and +(0,-0.5).. (0.8,0);
\draw[green,decoration={markings, mark=at position 0.5 with {\arrow{<}}}, postaction={decorate}] (0.8,0) ..controls +(0,0.5) and +(1,0).. (1,2);
\draw [green,decoration={markings, mark=at position 0.5 with {\arrow{>}}}, postaction={decorate}](1,2) ..controls +(-0.5,0) and +(0.5,0).. (0,1);

\draw[green,decoration={markings, mark=at position 0.5 with {\arrow{>}}}, postaction={decorate}] (0,-1) ..controls +(-0.5,0) and +(0.5,0).. (-1,-2.5);
\draw[green,decoration={markings, mark=at position 0.5 with {\arrow{<}}}, postaction={decorate}] (-1,-2.5) ..controls +(-1,0) and +(0,-0.5).. (-0.8,0);
\draw[green,decoration={markings, mark=at position 0.5 with {\arrow{<}}}, postaction={decorate}] (-0.8,0) ..controls +(0,0.5) and +(-1,0).. (-1,2.5);
\draw[green,decoration={markings, mark=at position 0.5 with {\arrow{>}}}, postaction={decorate}] (-1,2.5) ..controls +(0.5,0) and +(-0.5,0).. (0,1);

\draw[->] (2,-3) -- (2, 3);
\draw[left] (2,3)  node{$H$};
\draw[dotted] (-1,-2.5) -- (2.5,-2.5);
\draw[dotted] (0,-1) -- (3,-1);
\draw[dotted] (1,-2) -- (3,-2);
\draw[dotted] (0,1) -- (3.5,1);
\draw[dotted] (1,2) -- (3.5,2);
\draw[dotted] (-1,2.5) -- (4,2.5);

\draw[below] (-1,-2.5)  node{$p_1$};
\draw[below] (0,-1)  node{$x_1$};
\draw[below] (1,-2)  node{$p_2$};
\draw[above] (-1,2.5)  node{$s_1$};
\draw[above] (0,1)  node{$x_2$};
\draw[above] (1,2)  node{$s_2$};

\draw (0,-3.5) node{$\FF$};
\draw[green,->] (-1,2.5) ..controls +(0,-1.5) and +(0,0.5).. (-0.5,0);
\draw[green] (-0.5,0) ..controls +(0,-0.5) and +(0,1.5).. (-1,-2.5);
\draw[green,->] (1,2) ..controls +(0,-1.5) and +(0,0.5).. (0.5,0);
\draw[green] (0.5,0) ..controls +(0,-0.5) and +(0,1.5).. (1,-2);
\draw[green,decoration={markings, mark=at position 0.5 with {\arrow{>}}}, postaction={decorate}] (0,1) ..controls +(0,-0.5) and +(-0.5,1).. (1,-2);
\draw[green,dotted,decoration={markings, mark=at position 0.5 with {\arrow{>}}}, postaction={decorate}] (0,1) ..controls +(0,-0.5) and +(0.5,1).. (-1,-2);

\draw (3.5,-3.5) node{$B_{\gen}(f,\FF)$};

\draw (2.5,-2.5) -- (2.5, 3);
\draw (2.5,-2.5) node{$\bullet$};

\draw (3,-2) -- (3,-1);
\draw (3,-2) node{$\bullet$};
\draw (3,-1) node{$\bullet$};

\draw (3.5,1) -- (3.5,2);
\draw (3.5,2) node{$\bullet$};
\draw (3.5,1) node{$\bullet$};

\draw (4,2.5) -- (4,3);
\draw (4,2.5) node{$\bullet$};

\draw (-3,-3.5) node{$G(\FF)$};
\draw[below] (-4,-2.5)  node{$p_1$};
\draw[below] (-3,-1)  node{$x_1$};
\draw[below] (-2,-2)  node{$p_2$};
\draw[above] (-4,2.5)  node{$s_1$};
\draw[left] (-3,1)  node{$x_2$};
\draw[above] (-2,2)  node{$s_2$};
\draw[below] (-4,-2.5)  node{$p_1$};
\draw[below] (-3,-1)  node{$x_1$};
\draw[below] (-2,-2)  node{$p_2$};
\draw[above] (-4,2.5)  node{$s_1$};

\draw (-4,-2.5)  node{$\bullet$};
\draw (-3,-1)  node{$\bullet$};
\draw (-2,-2)  node{$\bullet$};
\draw (-4,2.5)  node{$\bullet$};
\draw (-3,1)  node{$\bullet$};
\draw (-2,2)  node{$\bullet$};
\draw (-4,-2.5)  node{$\bullet$};
\draw (-3,-1)  node{$\bullet$};
\draw (-2,-2)  node{$\bullet$};
\draw (-4,2.5)  node{$\bullet$};
\draw (-3,1)  node{$\bullet$};
\draw (-2,2)  node{$\bullet$};

\draw [decoration={markings, mark=at position 0.5 with {\arrow{>}}}, postaction={decorate}](-4,2.5) -- (-3,1);

\draw[decoration={markings, mark=at position 0.5 with {\arrow{>}}}, postaction={decorate}] (-2,2) -- (-3,1);

\draw[decoration={markings, mark=at position 0.5 with {\arrow{>}}}, postaction={decorate}] (-3,1) -- (-4,-2.5);
\draw[decoration={markings, mark=at position 0.5 with {\arrow{>}}}, postaction={decorate}] (-3,1) -- (-2,-2);
\draw[decoration={markings, mark=at position 0.5 with {\arrow{>}}}, postaction={decorate}] (-3,-1) -- (-4,-2.5);
\draw[decoration={markings, mark=at position 0.5 with {\arrow{>}}}, postaction={decorate}] (-3,-1) -- (-2,-2);

\draw [decoration={markings, mark=at position 0.5 with {\arrow{>}}}, postaction={decorate}](-4,2.5) -- (-3,-1);

\draw[decoration={markings, mark=at position 0.5 with {\arrow{>}}}, postaction={decorate}] (-2,2) -- (-3,-1);

\end{tikzpicture}
\end{center}
\caption{ }
\label{Morse2}
\end{figure}

In this example the barcode $B_{\gen}(f,\FF)$ is equal to the filtered Morse homology of the function $H$.
\end{rem}

In section \ref{equality} we will prove the following two results.\\

\begin{prop}\label{independance-generique}
The barcode $B_{\gen}(f,\FF)$ defined for a foliation $\FF$ does not depend on the choice of $\FF\in\FF_{\gen}(f)$.
\end{prop}

Hence we can denote $B_{\gen}(f)=B_{\gen}(f,\FF)$ for any choice of $\FF\in\FF_{\gen}(f)$. With this notation, we have the following theorem.\\

\begin{theorem}\label{C2close}
If we consider a Hamiltonian diffeomorphism $f$ with a finite number of fixed points which is $C^2$-close to the identity and generated by an autonomous Hamiltonian function then the barcode $B_{\gen}(f)$ is equal to the Floer homology barcode of $f$.
\end{theorem}

We would like to prove in a near future the more general result.

\begin{question}
Is the result of Theorem \ref{C2close} holds if only consider a Hamiltonian homeormorphism $f$ whose set of fixed points is finite and unlinked? 
\end{question}

				\section{Construction of the map $\mathcal{B}$}\label{construction}

We give an algorithmic way to determine the barcode of certain type of finite graphs. This section is independent from the rest of the article and give a general construction of barcodes without having to be in a symplectic context. We consider the set $\mathcal{G}$ of elements $(G,A,\ind)$ such that $G$ is a finite oriented and connected graph equipped with a function, called \emph{action function}, $A:V \ra\R$ decreasing along the edges and a map $\ind:V\ra\Z$ where $V$ is the set of vertices of $G$. We construct a map 
$$\mathcal{B}:\mathcal{G}\ra \mathrm{Barcode}.$$

For an element $(G,A,\ind)\in\mathcal{G}$, and a vertex $x$ of $G$, we will say that $x$ is a \textit{sink} (resp. a \textit{source}) of the graph $G$ if there is no edge which begins with $x$ (resp. if there is no edge which ends with $x$). For any other vertex $x$ of $G$, we will say that $x$ is a \emph{saddle point} of the graph $G$. \\

For an element $(G,A,\ind)\in\mathcal{G}$, we could suppose that for a vertex $x$ of $G$ $\ind(x)$ is non positive if $x$ is a saddle point of $G$ and $\ind(x)$ is equal to $1$ if $x$ is a sink or a source of $G$ as it will always be the case in our future applications. Howeover, we do not need to make these assumptions to construct the map $\mathcal{B}$.\\

\begin{definition}
Let us consider an element $(G,A,\ind)\in\mathcal{G}$. For a subgraph $G'$ of $G$ we define 
\begin{align*}
L(G')&=\min\{ A(x) |x\in V\cap G'\}, \\
D(G')&=\max\{ A(x) |x\in V\cap G'\}.
\end{align*}
\end{definition}

Let us consider an element $(G,A,\ind)\in\mathcal{G}$ and let us denote by $V$  the set of vertices of $G$. For $t\in\R$, we define two subgraphs $G^-_t$ and $G^+_t$ as follows.

\begin{definition}
For $t\in\R$ we denote by $G^-_t$ the maximal subgraph  of $G$ whose set of vertices is $V\cap A^{-1}((-\infty,t))$. \\
Symmetrically, for $t\in\R$ we denote by $G^+_t$ the maximal subgraph  of $G$ whose set of vertices is $V\cap A^{-1}((t,+\infty))$.
\end{definition}

Let us consider $t\in\R$ such that there exists $x\in V$ satisfying $A(x)=t$. Since $V$ is finite, we can define the graphs $G_{t^+}^-=G^-_{t+\epsilon}$ and $G_{t^-}^+=G^+_{t-\epsilon}$ where $\epsilon>0$ satisfies $\im(A) \cap ((t-\epsilon,t+\epsilon))=\{t\}$.

\begin{definition}
Let us considet $t\in\R$.\\
We denote by $\mathcal{C}^-_t$ (resp. $\mathcal{C}^+_t$) the set of connected components of  $G_{t}^-$ (resp. of connected components of $G_{t}^+$). \\
We denote by $\mathcal{C}^-_{t^+}$ the set of connected components of $G^-_{t^+}$ and we denote by $\mathcal{C}^+_{t^-}$ the set of connected components of $G^+_{t^-}$.
\end{definition}

\begin{definition}
The inclusions $G_t^-\subset G_{t^+}^-$ and $G_t^+\subset G_{t^-}^+$ induce natural maps $j_t:\mathcal{C}^-_t\ra\mathcal{C}^-_{t^+}$ and $j'_t:\mathcal{C}^+_t\ra \mathcal{C}^+_{t^-}$ where for $C\in\mathcal{C}^-_t$, $j_t(C)$ is the connected component of $G_{t^+}^-$ which contains $C$ and for $C'\in\mathcal{C}^+_t$, $j_t'(C')$ is the connected component of $G_{t^-}^+$ which contains $C'$.
\end{definition}

Now, we can give the definition of the map $\mathcal{B}$. Given an element $(G,A,\ind)$ of $\mathcal{G}$ we describe the bars of $\mathcal{B}(G,A,\ind)$.\\

\textbf{The map $\mathcal{B}$}\\

The barcode $\mathcal{B}(G,A,\ind)$ is composed of the bars of the following four categories. \\

\textit{Category 0.}  The bars $(L(G),+\infty)$ and $(D(G),+\infty)$ are bars of $\mathcal{B}(G,A,\ind)$.\\

For every $t\in \im(A)$ there are three categories of bars
\textit{Category 1.} For each element $C$ of $\mathcal{C}^-_{t^+}$ such that $j^{-1}_t(C)$ is not empty, the barcode $\mathcal{B}(G,A,\ind)$ contains $\# j^{-1}_t(C)-1$ bars as follows.\\
We label $C_1,...,C_n$ the elements of $j^{-1}_t(C)$ and we choose $i_0\in[1,n]$ an integer such that $L(C_{i_0})=\min_{i\in[1,n]} L(C_i)$. \\
The bars of category $1$ associated to $t$ are the bars $(L(C_i),t]$ for $i\neq i_0$.\\

\textit{Category 2.} For each element $C'$ of $\mathcal{C}^+_{t^-}$ such that $j'^{-1}_t(C')$ is not empty, the barcode $\mathcal{B}(G,A,\ind)$ contains $\# j'^{-1}_t(C')-1$ bars as follows.\\
We label $C'_1,...,C'_n$ the elements of $j'^{-1}_t(C')$ and we choose $i_0\in [1,n]$ an integer such that $D(C'_{i_0})=\max_{i\in[1,n]} D(C'_i)$.\\
The bars of category $2$ associated to $t$ are the bars $(t,D(C'_i)]$ for $i\neq i_0$.\\

\textit{Category 3.} We define $k=\sum\{|\ind(x)| \ | \ x \text{ saddle point}, \  A(x)=t \}$.
Let us denote $k'$ equal to $k$ minus the number of bars of categories $1$ and $2$ associated to $t$. If $k'> 0$ then the bars of category $3$ associated to $t$ are $k'$ bars $(t,+\infty)$ and if $k'\leq 0$ there is no bar of category $3$ associated to $t$.

\begin{rem}
We refer to Proposition \ref{connexionsaddle} to enlight the definition of the bars of category $3$.
\end{rem}

\begin{rem}
By construction for every bar $I=(a,b]$ or $J=(c,\infty)$ in the barcode $\mathcal{B}(G,A,\ind)$ we have that $a,b$ and $c$ are values of the action function $A$.
\end{rem}

\begin{rem}
Let us consider a Morse function $H$ on a surface $\Sigma$. One may choose an adapted Riemanian metric on $M$ such that the gradient lines of $H$, defined for this metric, satisfy the Smale conditions and hence the Homology can be define using broken gradient trajectories as described in the section \ref{preliminaries}. Using these gradient lines, we can define a finite and oriented graph $G$ such that the critical points of $H$ are the vertices of $G$ and such that there exists an edge from $x$ to $y$ if there exists a gradient line from $x$ to $y$. Notice that the function $H$ decreased along the edges of $G$. One may prove that the barcode $\mathcal{B}(G,H,\ind_{Morse})$, where $\ind_{Morse}$ is the Morse index defined on the citical points of $H$, is equal to the Morse Homology barcode and does not depend on the adapted Riemanian metric. We will use this idea to construct barcodes associated to Hamiltonian homeomorphisms in the next sections.
\end{rem}

\textbf{Example} We compute the barcode of one example. We consider $(G,A,\ind)\in \mathcal{G}$ as follows.

\begin{figure}[h]
\begin{center}
\begin{tikzpicture}

\draw (0,0) node{$\bullet$};
\draw[left] (0,0) node{$x$};
\draw (0,1) node{$\bullet$};
\draw[left] (0,1) node{$z$};
\draw (0.5,-1.5) node{$\bullet$};
\draw[left] (0.5,-1.5) node{$y_2$};
\draw (-0.5,-1) node{$\bullet$};
\draw[left] (-0.5,-1) node{$y_1$};

\draw[decoration={markings, mark=at position 0.5 with {\arrow{>}}}, postaction={decorate}] (0,0) -- (0.5,-1.5);
\draw[decoration={markings, mark=at position 0.5 with {\arrow{>}}}, postaction={decorate}] (0,0) -- (-0.5,-1);
\draw[decoration={markings, mark=at position 0.5 with {\arrow{<}}}, postaction={decorate}] (0,0) -- (0,1);

\draw[->] (1,-1.7) -- (1,1.7);
\draw[dotted] (0.5,-1.5) -- (1,-1.5);
\draw[dotted] (-0.5,-1) -- (1,-1);
\draw[dotted] (0,0) -- (1,0);
\draw[dotted] (0,1) -- (1,1);
\draw[right] (1,-1.5) node{$A(y_2)$};
\draw[right] (1,-1) node{$A(y_1)$};
\draw[right] (1,0) node{$A(x)$};
\draw[right] (1,1) node{$A(z)$};

\draw[below] (0,-1.6) node{$G$};
\end{tikzpicture}
\end{center}
\end{figure}

The map $\ind$ satisfies $\ind(x)=-1$ and $\ind(y_1)=\ind(y_2)=\ind(z)=1$. The values of the map $A$ are represented on the vertical line on the right of the graph.\\

The bars of category $0$ are $(A(y_2),+\infty)$ and $(A(z),+\infty)$.\\

The vertex $x$ is the unique saddle point of the graph $G$. We describe the bars associated to $A(x)$ as follows.\\

The subgraph $G^-_{A(x)^+}$ has only one connected component $\mathcal{C^-}$ and $j_{A(x)}^{-1}(\{\mathcal{C}^-\})=G^-_{A(x)^-}$ has two connected components $\mathcal{C}=\{y_1\}$ and $\mathcal{C}'=\{y_2\}$. In this example we have $L(\mathcal{C})=A(y_1)> L(\mathcal{C}')=A(y_2)$ so by construction the bar $(A(y_1),A(x)]$ is the only bar of category $1$ of the barcode $\mathcal{B}(G,A,\ind)$.\\

The subgraph that $G^+_{A(x)^-}$ has only one connected component $\mathcal{C}^+$ and ${j'}_{A(x)}^{-1}(\{\mathcal{C}^+\})=G^+_{A(x)^+}$ has one connected component $\mathcal{C}=\{z\}$. So by construction there is no bar of category $2$ in the barcode  $\mathcal{B}(G,A,\ind)$.\\

The index of $x$ is equal to $-1$ and there is one bar of category $1$ and zero bar of category $2$ thus there is no bars of category $3$ in the barcode $\mathcal{B}(G,A,\ind)$. \\

Finally we obtain the barcode
$$\mathcal{B}(G,A,\ind)=\{(A(y_2),+\infty), \ (A(y_1),A(x)], \ (A(z),+\infty)\}.$$

\begin{rem}
This example corresponds to the barcode of a Morse function $H$ defined on the sphere and with exactly one source $z$, one saddle point $x$ and two sinks $y_1$ and $y_2$ where the graph corresponds to the connexions of the gradient lines of $H$ as in remark \ref{MorseExample}. 
\end{rem}

				\section{The barcode of a gradient-like foliation}\label{2barcode}

Let us consider a gradient-like foliation $\FF$, whose set of singularities $X$ is finite, defined on the complement of $X$ in a compact surface $\Sigma$. Recall that a gradient-like foliation is a foliation such that every leaf is a connexion and where there is no cycle of connexions, see section \ref{foliation} of the preliminaries for more details. In particular, the singularities of $\FF$ are isolated and are classified in three categories: the sinks, the sources and the saddle points. We suppose that the set of singularities $X$ of $\FF$ is equipped with an action function $A:X \ra \R$ such that for each leaf $\phi$ we have $A(\alpha(\phi))>A(\omega(\phi))$. \\

We will consider the oriented graph $G(\FF)$ of the foliation $\FF$ whose set of vertices is $X$ and for every couple of vertices $x$ and $y$ of $G(\FF)$ there exists an edge from $x$ to $y$ if and only if there exists a leaf $\phi$ in $\FF$ such that $\alpha(\phi)=x$ and $\omega(\phi)=y$. We want to study the barcode $\mathcal{B}(G(\FF),A,\ind(\FF,\cdot))$ associated to $\FF$ defined in section \ref{construction}.\\

Notice that the graph $G(\FF)$ is not constructed as the graph of a generic foliation as in section \ref{génériqueF} but it remains a finite oriented graph. Moreover. Moerover, for an edge between from a singularity $x$ to another singularity $y$, we don't have to suppose that $\ind(\FF,x)-\ind(\FF,y)=1$. The differences will be enlightened in section \ref{equality}.\\

In a first section we give some geometrical properties of the foliation $\FF$ and in a second section we prove some results about the barcode $\mathcal{B}(G(\FF),A,\ind(\FF,\cdot))$.

		\subsection{Geometric properties of a gradient-like foliation}\label{geometricsprop}

We introduce some useful definitions and notation.\\

\textbf{Saturated set.} A subset of $\Sigma\backslash X$ is said to be \emph{saturated} if it is equal to a union of leaves of $\FF$. We will use the fact that the closure in $\Sigma\backslash X$ of a saturated set is saturated.\\

\textbf{Chain of connexions.} A \emph{chain of connexions} in $\Sigma$ is a finite union of the closure of leaves $\psi_1,...,\psi_k$ of $\FF$ such that $\alpha(\psi_1)=x$,  $\omega(\psi_i)=\alpha(\psi_{i+1})$ for every $i\in[1,k-1]$ and $\omega(\psi_k)=y$. We will say that a chain of connexions is associated to the leaves $\psi_1,...,\psi_k$. If we consider two singularities $x$ and $y$ of $\FF$, we say that there is a chain of connexions from $x$ to $y$ if there exists a chain of connexions, associated to leaves $\psi_1,...,\psi_k$ such that $\alpha(\phi_1)=x$ and $\omega(\phi_k)=y$. In this case, $x$ will be called the \textit{starting point} of the chain and $y$ its \textit{ending point}.\\

\textbf{Trivialization.} Let us consider a saddle point $x\in X$ of $\FF$. We denote by $\Sigma^-_x$ the union of leaves $\phi\in\FF$ whose alpha-limit point is equal to $x$ and by $\Sigma^+_x$ the union of leaves $\psi\in\FF$ whose omega-limit point is equal to $x$. We will call a \emph{trivialization}  of  $\FF$ at $x$ a couple $(h,V)$ where $V$ is a neighborhood of $x$ such that $X\cap V=\{x\}$ and $h:V\ra \D$ is a map that sends the foliation $\FF|_V$ to the model foliation described in the appendix of \cite{ROUX1} proposition B.5.4. which we now describe. To simplify the notations we set $n=1-\ind(\FF,x)$. In this model foliation, for each leaf $\phi$ of $\FF$, $\phi\cap V$ is connected and $h$ sends $\Sigma^-_x$ to $n$ cones centered around the angles $\frac{2\pi (2k)}{2n}$ where $0\leq k\leq n-1$ and sends $\Sigma^+_x$ to $n$ cones centered around the angles $\frac{2\pi (2k+1)}{2n}$ where $0\leq k\leq n-1$. We have\\ 
\begin{itemize}
\item The two sets $\Sigma^-_x$ and $\Sigma^+_x$ are composed of $n$ connected components. Using the map $h$ we can label these connected components $(\sigma^-_k)_{0\leq k\leq n-1}$ and $(\sigma^+_k)_{0\leq k \leq n-1}$ in cyclic order around $x$.
\item The connected components $(\sigma^-_k)_{0\leq k\leq n-1}$ will be called the \emph{unstable cones} of $x $ and the connected components $(\sigma^+_k)_{0\leq k \leq n-1}$ will be called the \emph{stable cones} of $x$.
\item  Every stable cone and an unstable cone which are consecutive, in cyclic order, are separated by \emph{hyperbolic sectors} of $x$. We denote $U_k$, $0\leq k\leq 2n-1$ the hyperbolic sectors such that we have in cyclic order $V=\sigma^-_0\cup U_0 \cup \sigma^+_0\cup U_1 \cup...\cup U_{2n-1}$.\\
\end{itemize}

Notice that a stable or unstable cone of $x$ can be composed of a unique leaf, see Figure \ref{saddlesaddle} for example.\\

\textbf{Local model.} Let us consider a leaf $\phi$ of $\FF$. We describe in $\Sigma$ the leaves of $\FF$ which are close to $\phi$. It will be called the \emph{local model} near $\phi$. We parametrize the leaves of $\FF$ near $\phi$ by a small arc $\gamma: (-1,1)\ra \Sigma\backslash X$ transversal to $\FF$ such that $\gamma(0)\in \phi$. For every $t\in(-1,1)$ we denote $\phi_t$ the leaf of $\FF$ passing through $\gamma(t)$. \\

We prove the following property.\\ 

\begin{prop}
We have that the sets
$$\Gamma_\phi^+= \bigcap_{\epsilon>0}\overline{\bigcup_{t\in(0,\epsilon)} \phi_t}, \ \text{and } \Gamma_\phi^-= \bigcap_{\epsilon>0}\overline{\bigcup_{t\in(-\epsilon,0)} \phi_t},$$
are chains of connexions containing the leaf $\phi$ and there exists $s\in(0,1)$ such that 
\begin{itemize}
\item Every leaf $\phi_t$, $t\in(0,s)$ satisfies: $\alpha(\phi_t)$ is equal to the start of $\Gamma^+_\phi$ and $\omega(\phi_t)$ is equal to its end.
\item Every leaf $\phi_t$, $t\in(-s,0)$ satisfies: $\alpha(\phi_t)$ is equal to the start of $\Gamma^-_\phi$ and $\omega(\phi_t)$ is equal to its end.
\end{itemize} 
\end{prop}

\begin{proof}
We prove the result for the leaves $\phi_t$, $t>0$ as it is the same proof for the leaves $\phi_t$, $t<0$. \\

We start by studying the "future" of the connexions $(\phi_t)_{t\in(0,1)}$ and then, symmetrically, we study the "past" of these connexions.\\

We will prove that there exists a chain of connexions contained in $\Gamma_\phi^+$ passing through $\phi$ from $x$ to a singularity $y^+$ such that for every $t>0$ small enough we have $\omega(\phi_t)=y$. The omega-limit point $x_1$ of $\phi$ is a sink of $\FF$ or a saddle point. If $x_1$ is a saddle point then $\phi$ is in the interior of a stable cone of $x_1$ or in its boundary. We split the discussion into three cases.\\

\textit{Case 1.} Suppose that $x_1$ is a sink of $\FF$, since the set of leaves of which $x_1$ is the ending point is open, there exists $t_1\in(0,1]$ such that for each $t\in [0,t_1)$, $\omega(\phi_t)$ is equal to $x_1$ and the chain of connexions we are looking for is associated to the leaf $\phi$.  \\

\textit{Case 2.} We suppose that $x_1$ is a saddle point and the leaf $\phi$ is in the interior of a stable cone $\sigma^+$ of $x_1$.  There exists $t_1\in(0,1]$ such that for each $t\in [0,t_1)$, the leaf $\phi_t$ satisfies $\omega(\phi_t)=x_1$. The chain of connexions we are looking for is associated to the leaf $\phi$.\\

\textit{Case 3.} We suppose that $x_1$ is a saddle point and the leaf $\phi$ is in the boundary of a stable cone $\sigma^+$ of $x_1$. In this case, the leaf $\phi$ is in the boundary of the hyperbolic sector $U$ of $x_1$ preceding $\sigma^-$. We can consider a trivialisation of $\FF$ at $x_1$ and $t_1\in(0,1)$ such that each leaf $\phi_t$, $t\in(0,t_1]$, is a leaf of the sector $U$. The closure of the union of the leaves $(\phi_t)_{t\in(0,t_1]}$ contains a leaf $\phi_1$ of the unstable cone of $x$ which is adjacent to $U$. Notice that we have a chain of connexions associated to the leaves $\phi$ and $\phi_1$.\\

We do the same discussion about the omega limit point of $\phi_1$. If the omega limit point of $\phi_1$ corresponds to the case $1$ or the case $2$, then we stop the process and if we are in case $3$ then, we do the same discussion with the leaf $\phi_2$ provided by case $3$. If so, a chain of connexions is associated to the leaves $\phi,\phi_1,\phi_2$. Since the number of singularities of $\FF$ is finite, the process stops after a finite number of steps and we finally obtain a chain of connexions $\Gamma_{\text{future}}$ associated to a finite number of leaves $\phi,\phi_1,...,\phi_n$. We denote $y$ the omega-limit point of the leave $\phi_n$ and there exists $t_n\in(0,1)$ such that each leaf $\phi_t$ with $t\in(0,t_n]$ satisfies $\omega(\phi_t)=y^+$.\\

Let us draw an example of a chain of connexions provided by the previous process. In Figure \ref{local-model} the horizontal line represents the chain of connexions from $x$ to $\omega(\phi_n)$ and each line above represents a leaf $\phi_t$ where $t\in (0,1]$. It is enough to re-parametrize the trivialization $h$ to obtain the result we are looking for.\\
\begin{figure}[H]
\begin{center}
\begin{tikzpicture}[scale=2]\label{figure 1}
\draw [ 
           decoration={markings, mark=at position 0.5 with {\arrow{>}}},
           postaction={decorate}
           ]
           (0,0) -- (1,0) ;
 \draw [ 
           decoration={markings, mark=at position 0.5 with {\arrow{>}}},
           postaction={decorate}
           ]
           (1,0) -- (2,0) ;
\draw [ 
           decoration={markings, mark=at position 0.5 with {\arrow{>}}},
           postaction={decorate}
           ]
           (2,0) -- (3,0) ;
         
\draw (0,0) -- (0,1) ;

\draw [ 
           decoration={markings, mark=at position 0.5 with {\arrow{>}}},
           postaction={decorate}
           ]
           (0,1) -- (3,0) ;
\draw [ 
           decoration={markings, mark=at position 0.5 with {\arrow{>}}},
           postaction={decorate}
           ]
           (0,2/3) -- (3,0) ;
\draw [ 
           decoration={markings, mark=at position 0.5 with {\arrow{>}}},
           postaction={decorate}
           ]
           (0,1/3) -- (3,0) ;

\draw (0,0) node {$\bullet$};
    \node[below left] at (0,0) {$x$} ;
    \node[below] at (0.5,0) {$\phi$};
 \draw (3,0) node {$\bullet$};
    \node[below right] at (3,0) {$y^+$} ;
    \node[below] at (2.5,0) {$\phi_n$};	
    \node[left] at (0,1) {$t_n$} ;
    
\draw (1,0) node {$\bullet$} ;
\node[below right] at (1,0) {$x_1$};

\draw (2,0) node {$\bullet$} ;

\end{tikzpicture}
\end{center}
\caption{ }
\label{local-model}   
\end{figure}

By symmetrical arguments there is another chain of connexions $\Gamma_{\text{past}}$ from a singularity $z^+$ to $\omega(\phi)$, passing through $\phi$, such that every leaf $\phi_t$, $t>0$, satisfies $\alpha(\phi_t)=z^+$.\\

Moreover, by the previous construction, we obtain that $\Gamma_\phi^+$ is a chain of connexions and is equal to the union of $\Gamma_{\text{past}}$ and $\Gamma_{\text{future}}$. So $z^+$ is its starting point and $y^+$ its ending point.\\
\end{proof}

\begin{rem}
The space of leaves of a gradient-like foliation $\FF$ is a non-Hausdorff manifold. The chains of connexions correspond to the set of non separated leaves. 
\end{rem}
Let us draw two examples of a local model of a leaf $\phi$.\\
 
We consider a first example in Figure \ref{locallocal}. The leaves above $\phi$ represent the leaves $\phi_t$ with $t\in(0,1)$ and the leaves below $\phi$ represent the leaves $\phi_t$ with $t\in(-1,0)$. The chain of connexions $\Gamma^+_\phi$ is a chain from $z^+$ to $y^+$ passing through $\phi$ and the chain of connexions $\Gamma_\phi^-$ is a chain of connexions from $z^-$ to $y^-$ passing through $\phi$.

\begin{figure}[H]
\begin{center}
\begin{tikzpicture}[scale=2]
\draw [ red,
           decoration={markings, mark=at position 0.5 with {\arrow{>}}},
           postaction={decorate}
           ]
           (0,0) -- (1,0) ;
 \draw [ 
           decoration={markings, mark=at position 0.5 with {\arrow{>}}},
           postaction={decorate}
           ]
           (1,0) -- (2,0) ;
\draw [ 
           decoration={markings, mark=at position 0.5 with {\arrow{>}}},
           postaction={decorate}
           ]
           (2,0) -- (3,0.5) ;
   \draw [ 
           decoration={markings, mark=at position 0.5 with {\arrow{>}}},
           postaction={decorate}
           ]
           (1,0) -- (2,-0.5) ;      
   \draw [ 
           decoration={markings, mark=at position 0.5 with {\arrow{>}}},
           postaction={decorate}
           ]
           (-1,0) -- (0,0) ; 
  \draw [ 
           decoration={markings, mark=at position 0.5 with {\arrow{>}}},
           postaction={decorate}
           ]
           (-2,0.5) -- (-1,0) ; 
       
\draw [ 
           decoration={markings, mark=at position 0.5 with {\arrow{>}}},
           postaction={decorate}
           ]
           (-2,-0.5) -- (-1,0) ; 
           
\draw [ 
           decoration={markings, mark=at position 0.5 with {\arrow{>}}},
           postaction={decorate}
           ]
           (0,0.2) ..controls +(1,0) and +(-1,-0.5).. (3,0.5) ;
\draw [ 
           decoration={markings, mark=at position 0.5 with {\arrow{>}}},
           postaction={decorate}
           ]
           (0,-0.2)..controls +(1,0) and +(-1,0.3).. (2,-0.5) ;
\draw [ 
           decoration={markings, mark=at position 0.5 with {\arrow{>}}},
           postaction={decorate}
           ]
           (-2,0.5) ..controls +(1,-0.3) and +(-1,0).. (0,0.2);
          
\draw [ 
           decoration={markings, mark=at position 0.5 with {\arrow{>}}},
           postaction={decorate}
           ]
           (-2,-0.5) ..controls +(1,0.3) and +(-1,0).. (0,-0.2);

\draw (0,0) node {$\bullet$};
    \node[below left] at (0,0) {$x$} ;
    \node[red,below] at (0.5,0.05) {$\phi$};
 \draw (3,0.5) node {$\bullet$};
    \node[below right] at (3,0.5) {$y^+$} ;

\draw (1,0) node {$\bullet$} ;
\node[below right] at (1,0) {$x_1$};

\draw (2,0) node {$\bullet$} ;

\draw (-1,0) node {$\bullet$};
\draw (-2,0.5) node {$\bullet$};
\node[below] at (-2,0.5) {$z^+$};
\draw (-2,-0.5) node {$\bullet$};
\node[below] at (-2,-0.5) {$z^-$};

\draw (2,-0.5) node {$\bullet$};
\node[below] at (2,-0.5) {$y^-$};

\end{tikzpicture}
\end{center}
\caption{ }
\label{locallocal}    
\end{figure}

We draw a second example of a leaf $\phi$ in Figure \ref{local2}. We consider $\gamma$ in blue and the chains of connexions $\Gamma_\phi^-$ and $\Gamma_\phi^+$ are both equal to $\overline{\phi}$.

\begin{figure}[H]
\begin{center}
\begin{tikzpicture}[scale=2]\label{figure 1}
\draw [ red,
           decoration={markings, mark=at position 0.5 with {\arrow{>}}},
           postaction={decorate}
           ]
           (0,0) -- (1,0) ;
 \draw [ 
           decoration={markings, mark=at position 0.5 with {\arrow{>}}},
           postaction={decorate}
           ]
           (1,0) -- (2,0.5) ;

   \draw [ 
           decoration={markings, mark=at position 0.5 with {\arrow{>}}},
           postaction={decorate}
           ]
           (1,0) -- (2,-0.5) ;      

  \draw [ 
           decoration={markings, mark=at position 0.5 with {\arrow{>}}},
           postaction={decorate}
           ]
           (-1,0.5) -- (0,0) ; 
       
\draw [ 
           decoration={markings, mark=at position 0.5 with {\arrow{>}}},
           postaction={decorate}
           ]
           (-1,-0.5) -- (0,0) ; 
           
\draw[decoration={markings, mark=at position 0.5 with {\arrow{>}}}, postaction={decorate},blue] (0,0) ..controls +(0.2,0.2) and +(-0.2,0.2).. (1,0);
\draw[decoration={markings, mark=at position 0.5 with {\arrow{>}}}, postaction={decorate},blue] (0,0) ..controls +(0.2,-0.2) and +(-0.2,-0.2).. (1,0);
\draw[decoration={markings, mark=at position 0.5 with {\arrow{>}}}, postaction={decorate}] (0,0) ..controls +(0.5,0.5) and +(-0.5,0.5).. (1,0);
\draw[decoration={markings, mark=at position 0.5 with {\arrow{>}}}, postaction={decorate}] (0,0) ..controls +(0.5,-0.5) and +(-0.5,-0.5).. (1,0);

\draw[decoration={markings, mark=at position 0.5 with {\arrow{>}}}, postaction={decorate}] (-1,0.6) ..controls +(1,-0.5) and +(-0.5,0).. (0.5,0.5);
\draw[decoration={markings, mark=at position 0.5 with {\arrow{<}}}, postaction={decorate}] (2,0.6) ..controls +(-1,-0.5) and +(0.5,0).. (0.5,0.5);
\draw[decoration={markings, mark=at position 0.5 with {\arrow{>}}}, postaction={decorate}] (-1,-0.6) ..controls +(1,0.5) and +(-0.5,0).. (0.5,-0.5);
\draw[decoration={markings, mark=at position 0.5 with {\arrow{<}}}, postaction={decorate}] (2,-0.6) ..controls +(-1,0.5) and +(0.5,0).. (0.5,-0.5);

\draw[blue] (0.5,-0.15) -- (0.5,0.15);
\draw[blue,below] (0.5,-0.15) node{$\gamma$};

\draw (0,0) node {$\bullet$};
    \node[left] at (0,0) {$x$} ;
    \node[red,below] at (0.7,0.05) {$\phi$};

\draw (1,0) node {$\bullet$} ;
\node[right] at (1,0) {$x_1$};

\end{tikzpicture}
\end{center}
\caption{ }
\label{local2}    
\end{figure}

\textbf{The graph of the foliation $\FF$.} We consider the oriented graph $G(\FF)$ whose set of vertices is equal to $\Sing(\FF)$  and for every couple of vertices $x$ and $y$ there exists an edge from $x$ to $y$ if and only if there exists a leaf $\phi$ of $\FF$ such that $\alpha(\phi)=x$ and $\omega(\phi)=y$. For $t\in \R$ we consider the subgraphs $G^-_t(\FF)$ (resp. $G^+_t\FF)$) which is the maximal subgraph of $G(\FF)$ whose 
set of vertices is $X\cap A^{-1}((-\infty,t))$ (resp. $X\cap A^{-1}((t,+\infty))$). \\

\textbf{Attractive basin.} Let us consider $t\in\R$ and a connected component $\mathcal{C}$ of $G^-_t(\FF)$. We define the \emph{attractive basin} of $\mathcal{C}$, denoted $W^s(\mathcal{C})$, as the union of the leaves of $\FF$ whose omega-limit point is a singularity of $\mathcal{C}$. Notice that it is a subset of $\Sigma$. \\
In particular we have
$$W^s(\mathcal{C})=\bigcup_{x\in X\cap \mathcal{C}} W^s(x).$$

\begin{lem}\label{bordhyperbolic}
Let us consider a saddle point $x$ in the frontier of $W^s(\mathcal{C})$. There exists a neighborhood $V$ of $x$ such that each hyperbolic sector $U$ of $x$ in $V$ is either included in $W^s(\mathcal{C})$ or disjoint of it.
\end{lem}

\begin{proof}
Let us fix a neighborhood $V$ of $x$. We consider an unstable cone $\sigma^-$ of $x$ and $U$ a hyperbolic sector of $x$ in $V$ adjacent to $\sigma^-$. We denote $\phi$ the leaf of $\FF$ such that $\phi=\sigma^-\cap \overline{U}$. Let $\gamma:[0,1)\ra U$ be a small arc transverse to the foliation $\FF$ and such that $\gamma(0)\in\phi$. For every $t\in[0,1)$ we denote $\phi_t$ the leaf of $\FF$ passing through $\gamma(t)$.\\

We denote by $y^+$ the ending point of $\Gamma^+_\phi= \bigcap_{\epsilon>0}\overline{\bigcup_{t\in(0,\epsilon)} \phi_t}$.\\

By the local model there exists $s\in(0,1)$ such that every leaf $\phi_t$, $t\in(0,s)$, satisfies $\omega(\phi_t)=y^+$.\\

So if $y^+$ is in $\mathcal{C}$ then every leaf $\phi_t$, $t\in(0,s)$, is in $W^s(\mathcal{C})$ and if $y^+$ is not in $\mathcal{C}$ then no leaf $\phi_t$, $t\in(0,s)$, is in $W^s(\mathcal{C})$. Hence, up to a smaller neighborhood $V'$ of $x$ we can suppose that every leaf of the hyperbolic sector $U$ in $V'$ is either in $W^s(\mathcal{C})$ or disjoint of it. Moreover, $x$ has finitely many hyperbolic sectors so we can suppose that we have the same property for each one of them and we obtain the result.
\end{proof}

In the next section we will need a precise description of $W^s(\mathcal{C})$. We describe it with the following proposition. Let us consider a stable or unstable cone $\sigma$ of a saddle point $x$ in the frontier of $W^s(\mathcal{C})$, we say that $\sigma$ is \emph{adjacent} to $W^s(\mathcal{C})$ if one and only one of the two hyperbolic sectors adjacent to $\sigma$ is in $W^s(\mathcal{C})$ and we say that $\sigma$ is \emph{surrounded} by $W^s(\mathcal{C})$ if the two hyperbolic sectors of $x$ which are adjacent to $\sigma$ is $W^s(\mathcal{C})$.\\

\begin{prop}\label{Propbord}
The set $W^s(\mathcal{C})$, is open, connected and its frontier is the closure of a finite union of leaves of $\FF$ contained in stable or unstable cones of saddle points.\\
More precisely, for each saddle point $x$ in the frontier of $W^s(\mathcal{C})$, the stable and unstable cones satisfy the following properties.
\begin{enumerate}
\item Let us consider a stable cone $\sigma^+$ of $x$. If $\sigma^+$ is surrounded by $W^s(\mathcal{C})$ then the leaves of $\partial\sigma^+$ are the only leaves of $\sigma^+$ in the frontier of $W^s(\mathcal{C})$. If $\sigma^+$ is adjacent to $W^s(\mathcal{C})$ then the leaf of $\partial \sigma^+\cap \partial U$, where $U$ is the adjacent hyperbolic sector of $x$ in $W^s(\mathcal{C})$ adjacent to $\sigma^+$, is the only leaf of $\sigma^+$ in the frontier of $W^s(\mathcal{C})$. If none of the previous situation holds then $\sigma^+$ is disjoint from the frontier of $W^s(\mathcal{C})$.\\

\item Let us consider an unstable cone $\sigma^-$ of $x$. There is a finite set of leaves of $\sigma^-$, possibly empty if $\sigma^-$ is not adjacent to $W^s(\mathcal{C})$, in the frontier of $W^s(\mathcal{C})$.
\end{enumerate}
\end{prop}

Every property of Proposition \ref{Propbord} is obvious except the finiteness property which is deduced from Lemma \ref{fini}. \\

The first point of Proposition \ref{Propbord} is a straigthforward consequence of the definition of $W^s(\mathcal{C})$. Unfortunately, the second point of Proposition \ref{Propbord} can not be more precise and we draw three examples to illustrate this. After these examples, we prove that $W^s(\mathcal{C})$ is open in $\Sigma$ and that there are finitely many leaves of $\FF$ in its frontier.\\
 
\emph{First example.} Let us consider a sink $y$ of a gradient-like foliation $\FF$ such that there exists a saddle point $x$ in the frontier of $W^s(y)$ as in the Figure \ref{figurebasin}. In this example we draw in red the leaves of $W^s(y)$ and in blue the leaves in its frontier whose omega-limit point equals to $x$. There exists an unstable cone $\sigma^-$ of $x$ surrounded by $W^s(y)$. The sectors $U$ and $U'$ of $x$ which are adjacent to $\sigma^-$ are in $W^s(y)$ and by the first point of Proposition \ref{Propbord} there are two leaves $\phi$ and $\phi'$ in the stable cones of $x$ which are adjacent to $U$ and $U'$ and in the frontier of $W^s(y)$. In our example there exists a leaf $\psi$ of the cone $\sigma^-$ which is in the frontier of $W^s(y)$. It is enough to suppose that $A(x)>A(z)>A(y)$ to obtain this example. In our example we choose to draw $\psi$ in the boundary of the unstable cone $\sigma^-$ but it could be any leaf of $\sigma^-$.\\

\begin{figure}[H]
\begin{center}
\begin{tikzpicture}[scale=0.7]

\draw[below] (-0.2,0) node{$x$};

\draw [decoration={markings, mark=at position 0.5 with {\arrow{>}}}, postaction={decorate}] (0,-3) -- (0,0);

\draw [blue,decoration={markings, mark=at position 0.5 with {\arrow{>}}}, postaction={decorate}] (0.3,-3) -- (0,0);
\draw[blue,below] (0.3,-3) node{$\phi'$};
\draw [decoration={markings, mark=at position 0.5 with {\arrow{>}}}, postaction={decorate}] (-0.3,-3) -- (0,0);

\draw [decoration={markings, mark=at position 0.5 with {\arrow{>}}}, postaction={decorate}] (0,3) -- (0,0);

\draw [blue,decoration={markings, mark=at position 0.5 with {\arrow{>}}}, postaction={decorate}] (0.3,3) -- (0,0);
\draw[blue,above] (0.3,3) node{$\phi$};
\draw [decoration={markings, mark=at position 0.5 with {\arrow{>}}}, postaction={decorate}] (-0.3,3) -- (0,0);

\draw [decoration={markings, mark=at position 0.5 with {\arrow{<}}}, postaction={decorate}] (-3,0.3) -- (0,0);
\draw [decoration={markings, mark=at position 0.5 with {\arrow{<}}}, postaction={decorate}] (-3,0) -- (0,0);
\draw [decoration={markings, mark=at position 0.5 with {\arrow{<}}}, postaction={decorate}] (-3,-0.3) -- (0,0);

\draw [red,decoration={markings, mark=at position 0.5 with {\arrow{<}}}, postaction={decorate}] (3,0.3) -- (0,0);
\draw [red,decoration={markings, mark=at position 0.5 with {\arrow{<}}}, postaction={decorate}] (3,0) -- (0,0);
\draw [blue,decoration={markings, mark=at position 0.5 with {\arrow{<}}}, postaction={decorate}] (2,-0.2) -- (0,0);
\draw [red,decoration={markings, mark=at position 0.5 with {\arrow{<}}}, postaction={decorate}]  (3,-0.3)-- (2,-0.2);
\draw (3.3,0) node{$\sigma^-$};

\draw [red,decoration={markings, mark=at position 0.5 with {\arrow{>}}}, postaction={decorate}] (3.6,0) -- (5,0);
\draw [red,decoration={markings, mark=at position 0.5 with {\arrow{>}}}, postaction={decorate}]  (3,-0.3)-- (5,0);
\draw [red,decoration={markings, mark=at position 0.5 with {\arrow{>}}}, postaction={decorate}] (3,0.3) -- (5,0);
\draw [red,decoration={markings, mark=at position 0.5 with {\arrow{>}}}, postaction={decorate}] (3,0.5) -- (5,0);
\draw [red,decoration={markings, mark=at position 0.5 with {\arrow{>}}}, postaction={decorate}] (3,-0.5) -- (5,0);

\draw[red,right] (5,0) node{$y$};

\draw [red,decoration={markings, mark=at position 0.5 with {\arrow{>}}}, postaction={decorate}]  (0.5,-3) ..controls +(-0.4,3) and +(-3,+0.4)..(3,-0.5)  ;
\draw[red] (1.5,1.5) node{$U$};

\draw [red,decoration={markings, mark=at position 0.5 with {\arrow{>}}}, postaction={decorate}]  (0.5,3) ..controls +(-0.4,-3) and +(-3,-0.4).. (3,0.5)  ;
\draw[red] (1.5,-1.5) node{$U'$};

\draw [decoration={markings, mark=at position 0.5 with {\arrow{<}}}, postaction={decorate}] (-3,0.5) ..controls +(3,-0.4) and +(0.4,-3).. (-0.5,3)  ;

\draw [decoration={markings, mark=at position 0.5 with {\arrow{<}}}, postaction={decorate}] (-3,-0.5) ..controls +(3,0.4) and +(0.4,3).. (-0.5,-3)  ;

\draw[blue,below] (1,-0.1) node{$\psi$};
\draw (2,-0.2) node{$\bullet$};
\draw[below] (2,-0.2) node{$z$};

\end{tikzpicture}
\end{center}
\caption{ }
\label{figurebasin}
\end{figure}

\emph{Second example.} Let us consider a sink $y$ of a gradient-like foliation $\FF$ such that there exists a saddle point $x$ in the frontier of $W^s(y)$ as in the Figure \ref{figurebasin2}. In this example we draw in red the leaves of $W^s(y)$ and in blue the leaves of its frontier whose alpha-limit point is equal to $x$. The unstable cone $\sigma^-$ of $x$ is adjacent to $W^s(y)$. In this example there exist two leaves $\psi$ and $\psi'$  of the unstable cone $\sigma^-$ from $x$ to two singularities $z$ and $z'$ which are in the frontier of $W^s(y)$.

\begin{figure}[H]
\begin{center}
\begin{tikzpicture}[scale=0.7]

\draw (-0.2,0.2) node{$x$};

\draw [decoration={markings, mark=at position 0.5 with {\arrow{>}}}, postaction={decorate}] (0,-3) -- (0,0);
\draw [blue,decoration={markings, mark=at position 0.5 with {\arrow{>}}}, postaction={decorate}] (0.3,-3) -- (0,0);
\draw [decoration={markings, mark=at position 0.5 with {\arrow{>}}}, postaction={decorate}] (-0.3,-3) -- (0,0);

\draw [decoration={markings, mark=at position 0.5 with {\arrow{>}}}, postaction={decorate}] (0,1) -- (0,0);

\draw [decoration={markings, mark=at position 0.5 with {\arrow{<}}}, postaction={decorate}] (-1,0) -- (0,0);

\draw [decoration={markings, mark=at position 0.5 with {\arrow{<}}}, postaction={decorate}] (3,0.5) -- (0,0);
\draw [blue,decoration={markings, mark=at position 0.5 with {\arrow{<}}}, postaction={decorate}]  (3,0.3)-- (0,0);
\draw [red,decoration={markings, mark=at position 0.5 with {\arrow{<}}}, postaction={decorate}] (3,0) -- (0,0);
\draw [blue,decoration={markings, mark=at position 0.5 with {\arrow{<}}}, postaction={decorate}] (2,-0.2) -- (0,0);
\draw [red,decoration={markings, mark=at position 0.5 with {\arrow{<}}}, postaction={decorate}]  (3,-0.3)-- (2,-0.2);
\draw[blue] (2,-0.2) node{$\cdot$};
\draw (3.3,0) node{$\sigma^-$};
\draw[blue,below] (1,-0.1) node{$\psi'$};
\draw (2,-0.2) node{$\bullet$};
\draw[below] (2,-0.2) node{$z'$};
\draw[blue,right] (3,0.5) node{$\psi$};

\draw [red,decoration={markings, mark=at position 0.5 with {\arrow{>}}}, postaction={decorate}]  (0.5,-3) ..controls +(-0.1,1) and +(-2,0)..(3,-0.5)  ;

\draw [decoration={markings, mark=at position 0.5 with {\arrow{>}}}, postaction={decorate}]  (0.2,1) ..controls +(0,-1) and +(-3,-0.4).. (3,0.7)  ;
\draw[red] (1.5,-1.5) node{$U$};

\draw [red,decoration={markings, mark=at position 0.5 with {\arrow{>}}}, postaction={decorate}] (3.6,0) -- (5,0);
\draw [red,decoration={markings, mark=at position 0.5 with {\arrow{>}}}, postaction={decorate}]  (3,-0.3)-- (5,0);
\draw [blue,decoration={markings, mark=at position 0.5 with {\arrow{>}}}, postaction={decorate}] (3,0.3) -- (5,0.5);
\draw [red,decoration={markings, mark=at position 0.5 with {\arrow{>}}}, postaction={decorate}] (5,0.5) -- (5,0) ;
\draw [decoration={markings, mark=at position 0.5 with {\arrow{>}}}, postaction={decorate}] (3,0.5) -- (5,0.7);
\draw [decoration={markings, mark=at position 0.5 with {\arrow{>}}}, postaction={decorate}] (3,0.7) -- (5,0.9);
\draw [red,decoration={markings, mark=at position 0.5 with {\arrow{>}}}, postaction={decorate}] (3,-0.5) -- (5,0);

\draw[red,right] (5,0) node{$y$};
\draw (5,0.5) node{$\bullet$};
\draw[right] (5,0.5) node{$z$};
\end{tikzpicture}
\end{center}
\caption{ }
\label{figurebasin2}
\end{figure}

\emph{Third example.} Let us consider  a sink $y$ of a gradient foliation $\FF$ such that there exists a saddle point $x$ in the frontier of $W^s(y)$ as in Figure \ref{example3}. In this example we draw in red leaves of $W^s(y)$ and in blue the leaves on its frontier. The unstable cone $\sigma^-$ is not adjacent to $W^s(y)$. In this example there exist leaves in the interior of $\sigma^-$ whose omega-limit point is $y$ and two leaves $\psi$ of $\partial \sigma^-$ and $\psi'$ in the interior of $\sigma^-$ are in the frontier of $W^s(y)$ The omega limit-points of $\psi$ and $\psi'$  are two saddle points $z$ and $z'$ of $\FF$.

\begin{figure}[H]
\begin{center}
\begin{tikzpicture}[scale=0.7]

\draw[below] (-0.2,0) node{$x$};

\draw [decoration={markings, mark=at position 0.5 with {\arrow{>}}}, postaction={decorate}] (0,-3) -- (0,0);

\draw [decoration={markings, mark=at position 0.5 with {\arrow{>}}}, postaction={decorate}] (0.3,-3) -- (0,0);
\draw [decoration={markings, mark=at position 0.5 with {\arrow{>}}}, postaction={decorate}] (-0.3,-3) -- (0,0);

\draw [decoration={markings, mark=at position 0.5 with {\arrow{>}}}, postaction={decorate}] (0,3) -- (0,0);

\draw [decoration={markings, mark=at position 0.5 with {\arrow{>}}}, postaction={decorate}] (0.3,3) -- (0,0);
\draw [decoration={markings, mark=at position 0.5 with {\arrow{>}}}, postaction={decorate}] (-0.3,3) -- (0,0);

\draw [decoration={markings, mark=at position 0.5 with {\arrow{<}}}, postaction={decorate}] (-3,0.3) -- (0,0);
\draw [decoration={markings, mark=at position 0.5 with {\arrow{<}}}, postaction={decorate}] (-3,0) -- (0,0);
\draw [decoration={markings, mark=at position 0.5 with {\arrow{<}}}, postaction={decorate}] (-3,-0.3) -- (0,0);

\draw [blue,decoration={markings, mark=at position 0.5 with {\arrow{<}}}, postaction={decorate}] (3,0.3) -- (0,0);
\draw [red,decoration={markings, mark=at position 0.5 with {\arrow{<}}}, postaction={decorate}] (3,0) -- (0,0);
\draw [blue,decoration={markings, mark=at position 0.5 with {\arrow{<}}}, postaction={decorate}] (3,-0.3) -- (0,0);
\draw (3.3,0) node{$\sigma^-$};

\draw[red,right] (5,0) node{$y$};
\draw [red,decoration={markings, mark=at position 0.5 with {\arrow{>}}}, postaction={decorate}] (3.6,0) -- (5,0);
\draw [red,decoration={markings, mark=at position 0.5 with {\arrow{>}}}, postaction={decorate}]  (3,-0.3)-- (5,0);
\draw [red,decoration={markings, mark=at position 0.5 with {\arrow{>}}}, postaction={decorate}] (3,0.3) -- (5,0);

\draw[blue,decoration={markings, mark=at position 0.5 with {\arrow{>}}}, postaction={decorate}] (4,1) -- (3,0.3);
\draw[blue,decoration={markings, mark=at position 0.5 with {\arrow{>}}}, postaction={decorate}] (4,-1) -- (3,-0.3);
\draw[decoration={markings, mark=at position 0.5 with {\arrow{<}}}, postaction={decorate}] (3.5,3) -- (3,0.3);
\draw[decoration={markings, mark=at position 0.5 with {\arrow{<}}}, postaction={decorate}] (3.5,-3) -- (3,-0.3);

\draw[red,decoration={markings, mark=at position 0.5 with {\arrow{>}}}, postaction={decorate}] (4,0.8) ..controls +(-0.5,-0.5) and +(-0.5,0.2).. (5,0);
\draw[red,decoration={markings, mark=at position 0.5 with {\arrow{>}}}, postaction={decorate}] (4,-0.8) ..controls +(-0.5,0.5) and +(-0.5,-0.2).. (5,0);

\draw[decoration={markings, mark=at position 0.5 with {\arrow{>}}}, postaction={decorate}]  (0.5,-3) ..controls +(-0.4,2) and +(-2,0.4).. (2,-0.3)  ;
\draw[decoration={markings, mark=at position 0.5 with {\arrow{>}}}, postaction={decorate}] (2,-0.3) ..controls +(1,-0.4) and +(-1,+3).. (3.5,-3.3);
\draw (1.5,1.5) node{$U$};

\draw[decoration={markings, mark=at position 0.5 with {\arrow{>}}}, postaction={decorate}]  (0.5,3) ..controls +(-0.4,-2) and +(-2,-0.4).. (2,0.8)  ;
\draw[decoration={markings, mark=at position 0.5 with {\arrow{>}}}, postaction={decorate}] (2,0.8) ..controls +(1,0.4) and +(-1,-3).. (3.5,3.6);
\draw[decoration={markings, mark=at position 0.5 with {\arrow{>}}}, postaction={decorate}] (0,0) ..controls +(3.5,0.5) and +(-1,-3).. (3.5,3.3);
\draw (1.5,-1.5) node{$U'$};

\draw [decoration={markings, mark=at position 0.5 with {\arrow{<}}}, postaction={decorate}] (-3,0.5) ..controls +(2,-0.4) and +(0.4,-3).. (-0.5,3)  ;

\draw [decoration={markings, mark=at position 0.5 with {\arrow{<}}}, postaction={decorate}] (-3,-0.5) ..controls +(3,0.4) and +(0.4,3).. (-0.5,-3)  ;

\draw[blue,below] (1,-0.1) node{$\psi$};
\draw (3,-0.3) node{$\bullet$};
\draw[below] (3,-0.3) node{$z$};
\draw[blue,above] (1,0.1) node{$\psi'$};
\draw (3,0.3) node{$\bullet$};
\draw[above] (3,0.3) node{$z'$};
\end{tikzpicture}
\end{center}
\caption{ }
\label{example3}
\end{figure}

The remaining part of this section aims to give the proof of Proposition \ref{Propbord}. The proof is divided into several lemmas. Notice that the following lemma will be useful in the next section.

\begin{lem}\label{Sigma-}
Let us consider a saddle point $x$ of $\FF$. For each unstable cone $\sigma^-$ of $x$ there exists a unique connected component $\mathcal{C}$ of $G_{A(x)}^-(\FF)$ such that each leaf $\phi$ of the cone $\sigma^-$ is in the attractive basin of $\mathcal{C}$.
\end{lem}

\begin{proof}
Let us consider an unstable cone $\sigma^-$ of $x$ and $\phi$ a leaf of $\sigma^-$. We denote by $\mathcal{C}$ the connected component of $G_{A(x)}^-(\FF)$ which contains the singularity $\omega(\phi)$. By the local model there exist two chains of connexions $\Gamma_\phi^+$ and $\Gamma_\phi^-$, whose ending points are denoted $y^+$ and $y^-$, containing $\phi$ such that each leaf $\psi$ of $\sigma^-$ close enough to $\phi$ satisfies $\omega(\psi)=y^+$ or $\omega(\psi)=y^-$. We have that $\Gamma_\phi^+$ and $\Gamma_\phi^-$ contain $\phi$, so by definition $y^+$ and $y^-$ are also in $\mathcal{C}$ and we deduce that every leaf $\psi$ of $\sigma^-$ close enough to $\phi$ is in $W^s(\mathcal{C})$. \\

Moreover, if we consider a leaf $\phi'\in\sigma^-$ in the closure of a sequence $(\phi'_n)_{n\in\N}$ of leaves in $W^s(\mathcal{C})$ then by the local model, there is a chain of connexions $\Gamma$ in the closure of the sequence $(\phi'_n)_{n\in\N}$ which contains $\phi'$. Notice that $A(\omega(\phi'))<A(x)$ so $\omega(\phi')$ and the ending point of the chain $\Gamma$ are in the same connected component of $G_{A(x)}^-(\FF)$. So $\omega(\phi')$ is in $\mathcal{C}$ and then $\phi'$ is in $W^s(\mathcal{C})$.\\

Thus we obtain that the intersection of $\sigma^-$ and $W^s(\mathcal{C})$ is open and closed in $\sigma^-$. So, since $\sigma^-$ is connected we have that every leaf of $\sigma^-$ is in $W^s(\mathcal{C})$.
\end{proof}

\begin{rem}
Lemma \ref{Sigma-} does not hold if $A(x)$ is replaced by some $t<A(x)$. Indeed if we only suppose that $\mathcal{C}$ is a connected component of $G_{t}^-(\FF)$ with $t<A(x)$ then the unstable cone $\sigma^-$ is not necessarily included in the attractive basin of $\mathcal{C}$. See example \ref{figurebasin2} above where the leaf $\psi'$ is not a leaf of the attractive basin of the connected component $\{y\}$ of $G_{A(y)^+}^-(\FF)$.
\end{rem}

We prove that the attractive basin of a connected component of $G_t^-(\FF)$ with $t\in\R$ is open with the two following lemmas.

\begin{lem}\label{interior}
Let us consider $t\in\R$ and $\mathcal{C}$ a connected component of $G_t^-(\FF)$. Each singularity $x$ of $\mathcal{C}$ is in the interior of $W^s(\mathcal{C})$.
\end{lem}

\begin{proof}
Let us consider a singularity $x$ of $\FF$ in $\mathcal{C}$. It is either a sink, a source or saddle. We separate those three cases. \\

1. If $x$ is a sink, then each leaf in a neighborhood of $x$ is in $W^s(\mathcal{C})$. \\

2. If $x$ is a source, then each leaf $\phi$ in a neighborhood of $x$ satisfies $A(\omega(\phi))<A(x)$ and is in $W^s(\mathcal{C})$.\\

3. If $x$ is a saddle point, we consider a trivialization $(h,V)$ of $\FF$ at $x$ on a neighborhood $V$ given by Lemma \ref{bordhyperbolic}. By definition of $W^s(\mathcal{C})$, each leaf of $\Sigma^-_x$ and $\Sigma^+_x$ is in $W^s(\mathcal{C})$. Moreover, let us consider a leaf $\phi$ in the boundary of an unstable cone $\sigma^-$ of $x$ and $U$ the hyperbolic sector adjacent to $\phi$, by the local model there exists a chain of connexions from $x$ to a singularity $y$ such that every leaf of the hyperbolic sector $U$ admits $y$ as its omega limit. By definition, the singularity $y$ is in $\mathcal{C}$ hence the leaves of the hyperbolic sector $U$ are in $W^s(\mathcal{C})$. Since it holds for each hyperbolic sector of $x$ we obtain the result.

\end{proof}

Let us consider $t\in\R$ and $\mathcal{C}$ a connected component of $G^-_t(\FF)$. We describe the attractive basin of $\mathcal{C}$ and its frontier.

\begin{lem}\label{bordfini}
The intersection of $W^s(\mathcal{C})$ and $\Sigma\backslash X$ is open. 
\end{lem}

\begin{proof}
We consider a leaf $\phi$ in $W^s(\mathcal{C})$, by the local model there exists a small neighborhood $V$ of $\phi$ and two chains of connexions $\phi_1,...,\phi_k$ and $\phi'_1,...,\phi'_{k'}$ of $\FF$ which contain $\phi$ such that each leaf $\psi$ passing through $V$ satisfies $\omega(\psi)=\omega(\phi_k)$ or $\omega(\psi)=\omega(\phi'_{k'})$. By construction $\omega(\phi_k)$ and $\omega(\phi'_{k'})$ are in the same connected component of $G_t^-(\FF)$ and so every leaf passing through $V$ is in $W^s(\mathcal{C})$.
\end{proof}

The following corollary is a consequence of Lemmas \ref{interior} and \ref{bordfini}.

\begin{coro}
The attractive basin of $\mathcal{C}$ is an open surface in $\Sigma$.
\end{coro}

We deduce the following lemma.

\begin{lem}
The attractive basin of $\mathcal{C}$ is connected.
\end{lem}

\begin{proof}
For every $x\in X\cap \mathcal{C}$, $W^s(\mathcal{x})$ is connected as it is arc-connected. Let us consider two singularities $x$ and $y$ of $\mathcal{C}$ such that there exists an edge from $x$ to $y$. So there is a leaf $\phi$ such that $\alpha(\phi)=x$ and $\omega(\phi)=y$. The point $x$ is in the frontier of $W^s(y)$ and is either a saddle point or a source. Since we proved previously that $x$ is in the interior of $W^s(\mathcal{C})$ we deduce that $\overline{\phi}\subset W^s(x)\cup W^s(y)$. \\
We deduce easily that $W^s(\mathcal{C})=\bigcup_{x\in X\cap \mathcal{C}} W^s(x)$ is connected.
\end{proof}

We prove a last result which concludes the proof of Proposition \ref{Propbord}.  

\begin{lem}\label{fini}
The number of leaves included in the frontier of the attractive basin of $\mathcal{C}$ is finite.
\end{lem}

\begin{proof}[Proof of Lemma \ref{fini}]
Let us consider a leaf $\phi$ of the frontier of $W^s(\mathcal{C})$. The singularity $\omega(\phi)$ can not be a sink nor a source so it is a saddle point. So $\phi$ is a leaf of a stable cone of $\omega(\phi)$. So by the first point of Proposition \ref{Propbord} and the fact that the number of singularities of $\FF$ is finite, there exists a finite number of such leaves and we obtain the result.
\end{proof}

\begin{rem}
Each result about a connected components of $G_t^-(\FF)$, where $t\in \R$, has its symmetrical result for a connected component of $G_t^+(\FF)$.
\end{rem}

			\subsection{Some properties of the Barcode $\mathcal{B}(G(\FF),A,\ind(\FF,\cdot))$}

Let us consider a gradient-like foliation $\FF$ defined on the complement of a finite set $X$ of a compact surface $\Sigma$. We suppose that $X$ is equipped with an action function $A:X \ra \R$ such that for each leaf $\phi$ we have $A(\alpha(\phi))>A(\omega(\phi))$. We prove some properties about the Barcode $\mathcal{B}(G(\FF),A,\ind(\FF,\cdot))$, which will be denoted $\bm{\beta}_\FF$. \\

We set $A_m$ the minimum value of the action function $A$ on $X$ and $A_M$ its maximum value.\\

Notice that, by construction of the map $\mathcal{B}$, for every singularity $z$ of $\FF$ of index zero there is no bar in $\mathcal{B}(G(\FF),A,\ind(\FF,\cdot))$ one end of whose is equal to $A(z)$. Indeed, the point $z$ does not connect distinct connected components of the graphs $G^-_{A_f(z)}(\FF)$ and $G_{A_f(z}^+(\FF)$ so if we follow the construction in section \ref{construction}, there is no bar of category $1$ or $2$ induced by $z$. So, since $\ind(\FF,z)=0$ there is no bar of category $3$ induced by $z$ neither.\\

We consider the graph $G(\FF)$ associated to $\FF$ and the subgraphs $(G^-_t(\FF))_{t\in \R}$ and $(G^+_t(\FF))_{t\in\R}$. We will use some notations from section \ref{construction}. For $t\in\R$ we denote $\mathcal{C}^-_t$ and $\mathcal{C}^+_t$ the set of connected components of $G^-_t(\FF)$ and $G^+_t(\FF)$.\\
We set the maps $j_t:\mathcal{C}^-_t\ra \mathcal{C}^-_{t^+}$ and $j'_t:\mathcal{C}^+_t\ra \mathcal{C}^+_{t^-}$ where for $C\in\mathcal{C}^-_t$, $j_t(C)$ is the connected component of $G_{t^+}^-$ which contains $C$ and for $C'\in\mathcal{C}^+_t$ $j_t'(C')$ is the connected component of $G_{t^-}^+$ which contains $C'$. For a connected component $\mathcal{C}$ of $G^-_t(\FF)$ we consider $L(\mathcal{C})=\min\{A(y)\ | \ y \text{ vertex of } \mathcal{C}\}$.

\begin{lem}\label{sink-source}
For each sink $y$ of $\FF$ satisfying $A_m<A(y)$ there exists at least one saddle point $x$ of $\FF$ such that the barcode $\bm{\beta}_\FF$ contains a bar $b=(A(y),A(x)]$.\\
For each source $y$ of $\FF$ satisfying $A(y)<A_M$ there exists at least one saddle point $x$ of $\FF$ such that the barcode $\bm{\beta}_\FF$ contains a bar $b=(A(x),A(y)]$.
\end{lem}

In the situation of Lemma \ref{sink-source} we say that the pair $(x,y)$ is associated to the finite bar $b$ of $\bm{\beta}_\FF$.

\begin{proof}
Let us fix a value $t>A_m$ of the action function $A$. We label $y_1,...,y_n$ the sinks of $X$ whose action value is $t$. We prove that for every $i\in\{1,...,n\}$ there exists a saddle point $x_i$ of $\FF$ such that the barcode $\bm{\beta}_\FF$ contains a bar $(A(y_i),A(x_i)]$.\\

For $i\in[1,n]$ there is a map $c_i:s\mapsto \mathcal{C}_i^s$ defined for $s>t$ where $\mathcal{C}_i^s$ is the connected component of the graph $G^-_s(\FF)$ which contains $y_i$.\\

For a value $s>t$ such that $A^{-1}((t,s))= \emptyset$, the connected components $\mathcal{C}_1^s,...,\mathcal{C}_n^s$ are distinct. \\

For $s>t$ close to $t$, the elements $c_i(s)$, $i\in [1,n]$, are distincts and for $s>A_M$, we have that $c_i(s)=G(\FF)$ for every $i\in [1,n]$, hence we have $L(c_i(s))=t$ for $s>0$ close to $t$ 	and $L(c_i(s))=A_m$ for $t>A_M$.\\

Thus, for every $i\in [1,n]$, there exists an action value, denoted $s_i$, such that $A(y_i)=L(c_i(s_i))$ and $t>L(j_s(c_i(s_i)))$. In other words, there exists a saddle point $x_i$ of $\FF$ of action $s_i$ such that 
\begin{itemize}
\item $L(c_i(s_i))=t=A(y_i)$
\item $x_i$ connects the connected component $c_i(s_i)$ and another connected component $\mathcal{C}$ of $G_{s}^-(\FF)$ such that $L(\mathcal{C})<t$.
\end{itemize}
By construction of the barcode $\bm{\beta}_\FF$ there exists a bar $(A(y_i),s_i]$ of category $1$ (see section \ref{construction}).\\

Notice that the points $x_i$, $i\in\{1,...,n\}$ are not sources. Indeed, by contradiction we suppose that a source $z$ of action $s$ connects two distinct connected components $\mathcal{C}_1$ and $\mathcal{C}_2$ of the graph $G^-_{s}(\FF)$. Thus, by hypothesis, there exist two singularities $x_1\in \mathcal{C}_1$ and $x_2\in\mathcal{C}_2$ in the frontier of $W^u(z)$. By Proposition \ref{Propbord}, there exists a path of connexions $\Gamma$ between $x_1$ and $x_2$ such that every singularity $y$ in $\Gamma$ satisfies $A(y)<s$, so $x_1$ and $x_2$ are in the same connected component of $G^-_{s}(\FF)$ and we have a contradiction.\\

We obtain the symmetrical results for sources of $\FF$ by considering bars of category $2$.
\end{proof}

\begin{rem}
For a singularity $y$ of $\FF$ of index $1$, there may exist two saddle points $x$ and $x'$ with $A(x)=A(x')$ such that the couples $(x,y)$ and $(x',y)$ are associated to the same bar $b$. 
\end{rem}

\begin{lem}\label{sink-source-mM}
We label $y_1,...,y_n$ the sinks of $\FF$ of action $A_m$. There exist $n-1$ finite bars $J_1,...,J_{n-1}$ whose lower bound is equal to $A_m$ and upper bound is the action value of a saddle point of $\FF$ and one semi-infinite bar $(A_m,\infty)$ in the barcode $\bm{\beta}_\FF$.
\end{lem}

\begin{proof}
For $s>A_m$ such that $A^{-1}((A_m,s))= \emptyset$, we denote $\mathcal{C}_i^s$ the connected component of $G_s^-(\FF)$ which contains the singularity $y_i$.\\

For $i\in[1,n]$ there is a map $c_i:s\mapsto \mathcal{C}_i^s$ defined for $s>t$ where $\mathcal{C}_i^s$ is the connected component of the graph $G^-_s(\FF)$ which contains $y_i$.\\

For $s>A_m$ close to $A_m$, we have $\mathcal{C}_i^s=\{y_i\}$, $i\in [1,n]$.\\

Let us consider $s>A_m$, such that there exists a subset $K$ of $\{1,...,n\}$ of cardinal at least two such that the connected components $c_i(s)$, $i\in K$ are distinct but the connected component $j_s(c_i(s))$, $i\in K$, are equal. We have $L(c_i(s))=A_m$ for $i\in K$ so, by construction, it induces the existence of $\#K-1$ bars $(A_m,s]$ in the barcode $\bm{\beta}_\FF$.\\

Moreover, for each $i\neq j$ in $\{1,...,n\}$ there exists $s>A_m$ such that $j_s(c_i(s))=j_s(c_j(s))$ and $c_i(s)\neq c_j(s)$. Hence we obtain $n-1$ finite bars such that $A_m$ is the lower bound of the bar.\\

The existence of a semi-infinite bar $(A_m,\infty)$ in the barcode $\mathcal{B}(G(\FF),A,\ind(\FF,\cdot))$ is provided by the construction of $\mathcal{B}$ : it is a bar of category $0$.
\end{proof}

Let us consider a sink $y_m$ of the foliation $\FF$ such that $A(y_m)=A_m$ and a source $y_M$ of the foliation $\FF$ such that $A(y_M)=A_M$. We denote $X_{<0}\subset X$ the set of negative index singularities of $\FF$ and $X_1\subset X$ the set of singularities of index $1$.\\

By Lemmas \ref{sink-source} and \ref{sink-source-mM} there exists a map, which may be not unique,
 \begin{equation}\label{iota}
 \iota:X_{1}\backslash \{y_m,y_M\}\ra X_{<0},
 \end{equation}
where $\iota(y)=x$ is a saddle point $x$ of the foliation $\FF$ given by Lemmas \ref{sink-source} and \ref{sink-source-mM} such that the couple $(x,y)$ is associated to a finite bar of $\bm{\beta}_\FF$.\\

In particular, if $A$ is injective and each singularity $x$ of $X_{<0}$ has an index equal to $-1$ then, every bar $(a,b]$ of the barcode is naturally associated to a unique couple $x,y\in X$ by Lemma \ref{sink-source}. In this case, $\iota$ is unique and an injection. Indeed, by construction, for every saddle point $x$ of $X_{<0}$ there exists a unique bar one end of which is equal to $A(x)$. notice that this bar is either a finite interval whose infimum is $A(x)$, maximum is $A(x)$ or a semi-infinite interval whose infimum is $A(x)$.\\

We have the following result.

\begin{prop}\label{connexionsaddle}
Let $x$ be a saddle point of $\FF$ of action value $t\in\R$ such that $A^{-1}(t)=\{x\}$. We have
$$\# j^{-1}_t(\mathcal{C}_x)+\# j'^{-1}_t(\mathcal{C}'_x)\leq |\ind(\FF,x)|+2,$$
where $\mathcal{C}_x$ (resp. $\mathcal{C}'_x$) is the connected component  of $G_{t^+}^-(\FF)$ (resp. $G_{t^-}^+(\FF)$) which contains $x$.\\
\end{prop} 

In general, Proposition \ref{connexionsaddle} certifies that for a saddle point $x$ of $\FF$ such that $A^{-1}(t)=\{x\}$, there are $|\ind(f,x)|$ bars of which $t$ is an end. We can link this result to Proposition $28$ of \cite{ROUX3} which asserts that for a Hamiltonian function $H$ on a manifold the number of bars of which $t\in\R$ is an ending value is given by the dimension of the local Floer Homology at $t$.

\begin{proof}[Proof of Proposition \ref{connexionsaddle}]
We set $t=A(x)$ and $i=\ind(\FF,x)$. We label $\Sigma^-_x$ the set of the unstable cones of $x$ and $\Sigma^+_x$ the set of the stable cones of $x$. Both $\Sigma^-_x$ and $\Sigma^+_x$ are equipped with a cyclic order and a natural \textit{cyclic transformation} $\tau_x :\Sigma^-_x\cup\Sigma^+_x \ra \Sigma^-_x\cup\Sigma^+_x$ which sends $\Sigma^-_x$ into $\Sigma^+_x$ and $\Sigma^+_x$ into $\Sigma^-_x$ such that every element $\sigma^-\in \Sigma^-_x$ (resp. $\sigma^+\in \Sigma^+_x$) is sent to the element of $\Sigma^+_x$ (resp. $\Sigma^-_x$)  right after it in cyclic order. \\

By Lemma \ref{Sigma-}, for every $\sigma^-\in\Sigma^-_x$ there is a connected component $\mathcal{C}^-$of $G_{t}^-(\FF)$ such that $\omega(\phi)$ is a vertex of $\mathcal{C}^-$ for every leaf $\phi\in \sigma^-$. So we can define a map $\omega:\Sigma^-_x\ra \mathcal{C}^-_t(\FF)$ given by $\omega(\sigma^-)=\mathcal{C}^-$.  

Symmetrically, for every $\sigma^+\in\Sigma^+_x$ there is a connected component $\mathcal{C}^+$of $G_{t}^+(\FF)$ such that $\alpha(\phi)$ is a vertex of $\mathcal{C}^+$ for every leaf $\phi\in \sigma^+$. So we can define a map $\alpha:\Sigma^+_x\ra \mathcal{C}^+_t(\FF)$ given by $\alpha(\sigma^+)=\mathcal{C}^+$. \\

We will denote by $\im(\alpha)$ and $\im(\omega)$ the image sets of $\alpha$ and $\omega$ and the result consists in establishing the inequality 
\begin{equation}\label{inegalitéco}
\#\im(\alpha)+\#\im(\omega)\leq i+2.
\end{equation}

We introduce a combinatorial context to facilitate the proof.\\

Let us consider two sets $E^+$ and $E^-$ and a couple of maps $\alpha:\Sigma^+\ra E^+$ and $\omega\Sigma^- \ra E^-$ defined on sets $\Sigma^+\subset \Sigma^+_x$, $\Sigma^-\subset \Sigma^-_x$ such that there is, in cyclic order, alternatively an element of $\Sigma^+$ and an element of $\Sigma^-$ in $\Sigma^-_x\cup\Sigma^+_x$. In particular, $\Sigma^+$ and $\Sigma^-$ have the same cardinal and are naturally equipped with a cyclic transformation $\tau$. We define some useful notation.\\

Let $c$ be in the image of $\omega$. We set $J=\omega^{-1}(c)$ and we will say that 
\begin{itemize}
\item An element $\sigma^+\in\Sigma^+$ is \emph{adjacent to $J$ in $\Sigma^+$} if one and only one of the two elements $\tau(\sigma^+)$ and $\tau^{-1}(\sigma^+)$ is in $J$. More precisely we will say that an element $\sigma^+\in\Sigma^+$ is \emph{adjacent and before $J$ in $\Sigma^+$} if
$\tau(\sigma^+)$ is in $J$ and $\tau^{-1}(\sigma^+)$ is not. Symmetrically we will say that an element $\sigma^+\in\Sigma^+$ is \emph{adjacent and after $J$ in $\Sigma^+$} if $\tau^{-1}(\sigma^+)$ is in $J$ and $\tau(\sigma^+)$ is not.\\
\item An element $\sigma^+\in\Sigma^+$ is \emph{surrounded by $J$ in $\Sigma^+$} if the two elements $\tau(\sigma^+)$ and $\tau^{-1}(\sigma^+)$ are in $J$. 
\end{itemize}
A maximal set of consecutive elements of $J$ in $\Sigma^-$ will be called a maximal interval of $J$.\\

We prove the following lemma about the maps $\alpha$ and $\omega$.

\begin{lem}\label{combi}
Suppose that for every $c\in \im(\omega)$ and every $\sigma^+\in\Sigma^+$ adjacent to $\omega^{-1}(c)$ in $\Sigma^+$, there exists $\sigma'^+\in\Sigma^+\backslash\{\sigma^+\}$ adjacent to $\omega^{-1}(c)$ in $\Sigma^+$ such that $\alpha(\sigma'^+)=\alpha(\sigma^+)$. Then we have
$$\#\im(\alpha)+\#\im(\omega)\leq i+2,$$
where $i+1$ is the cardinal of the sets $\Sigma^+$ and $\Sigma^-$.
\end{lem}

\begin{proof} Let $i=0$, for a couple $(\alpha,\omega)$ defined on sets $\Sigma^+$ and $\Sigma^-$ of cardinal equal to $1$, the inequality is trivial.\\ 

Let $i\geq 1$, we suppose by induction that couples defined on sets of cardinal less than $i$ satisfying the hypothesis of Lemma \ref{combi} satisfy the result. \

We consider a couple $(\alpha,\omega)$ defined on sets $\Sigma^+$ and $\Sigma^-$ of cardinal equals to $i+1$ satisfying the hypothesis of Lemma \ref{combi}. We prove that $\#\im(\alpha)+\#\im(\omega)\leq i+2.$ \\

We divide the proof of the inequality into three cases.\\

\textit{Case 1.} Suppose that $\omega$ is constant. In this case we have by definition  $\# \im(\omega)=1$ and $\# \im(\alpha)\leq i+1$ so the result is trivial.\\

\textit{Case 2.} Suppose that for each $c\in \im(\omega)$ we have $\# \omega^{-1}(c)=1$. By hypothesis, we obtain that $\alpha$ is constant on $\Sigma^+$. So we compute
$$\#\im(\alpha)+\#\im(\omega)=1+(i+1),$$
and we obtain the result.\\

\textit{Case 3.} Suppose that there exists $c\in \im(\omega)$ such that $i+1>\#\omega^{-1}(c)>1$. We will modify the couple $(\alpha,\omega)$ into another couple $(\alpha_1,\omega_1)$ defined on subsets of $\Sigma^+$ and $\Sigma^-$ and we prove that the couple $(\alpha_1,\omega_1)$ satisfies the hypothesis of Lemma \ref{combi}. \\

We denote $J=\omega^{-1}(c)$ and we define new maps on $\Sigma^-_1=\Sigma^-\backslash J$ and $\Sigma^+_1=\Sigma^+\backslash \tau(J)$.\\
We set $\omega_1:\Sigma^-_1\ra E^-\backslash\{c\},$
defined as the restriction of $\omega$.\\
We define $\alpha_1:\Sigma^+_1  \ra E^+/ \alpha(\tau(J) \Delta  \tau^{-1}(J)),$
where we have 
\begin{itemize}
\item $\tau(J) \Delta  \tau^{-1}(J)$ is the symmetric difference of $\tau(J)$ and $\tau^{-1}(J)$. 
\item $E^+/ \alpha(\tau(J) \Delta  \tau^{-1}(J))$ is the set obtained by identifiyng the elements of $\alpha(\tau(J) \Delta  \tau^{-1}(J))$.
\item $\alpha_1$ is the natural map induced by $\alpha$.
\end{itemize}
Notice that the set $\tau(J) \Delta  \tau^{-1}(J)$ is composed of the elements of $\Sigma^+$ which are adjacent to $J$. Moreover, the sets $\Sigma^-_1$ and $\Sigma^+_1$ are not empty.\\

We prove that $(\alpha_1,\omega_1)$ satisfies the hypothesis of Lemma \ref{combi}.\\

Let us consider $c_1\in\im(\omega_1)$. We set $J_1=\omega_1^{-1}(c_1)$ and we consider $\sigma_1^+\in \Sigma^+_1 $ adjacent to $J_1$ in $\Sigma^+_1$. Our goal is to find an element of $\Sigma_1^+$ distinct from $\sigma^+_1$ and adjacent to $J_1$ in $\Sigma_1^+$ whose $\alpha_1$ value is equal to the $\alpha_1$ value of $\sigma^+_1$.\\
The element $\sigma_1^+$ may not be adjacent to $J_1$ in $\Sigma^+$. There are two possibilities.\\
1) $\sigma_1^+$  is adjacent to $J_1$ in $\Sigma^+$.\\
2) $\sigma_1^+$ is adjacent and before a maximal interval $K$ of $J$ in $\Sigma^+$ and there exists ${\sigma'}^+_1\in J$ which is adjacent and after the interval $K$ of $J$ in $\Sigma^+$ and adjacent and before  $J_1$ in $\Sigma^+$.\\

We separate these two cases.\\

In case 1), by hypothesis, there exists $\sigma^+_2\in\Sigma^+\backslash\{\sigma^+_1\}$ adjacent to $J_1$ in $\Sigma^+$ such that $\alpha(\sigma_1^+)=\alpha(\sigma_2^+)$. Again, there are two possibilities.\\

i) If $\sigma_2^+\notin \tau(J)$ then $\sigma_2^+$ is adjacent to $J_1$ in $\Sigma_1^+$ and $\alpha_1(\sigma_1^+)=\alpha_1(\sigma_2^+)$ by definition. We obtain the result.\\

ii) If $\sigma_2^+\in \tau(J)$ then $\sigma_2^+$ is adjacent and after a maximal interval $K_2$ of $J$ in $\Sigma^+$. So $\sigma_2^+$ is before an interval $K_1$ of $J_1$ in $\Sigma^+$ and after the interval $K_2$ of $J$ in $\Sigma^+$.
 We consider ${\sigma'}_2^+$ in $\Sigma^+$ just before the interval $K_2$ in $\Sigma^+$. We have that ${\sigma'}_2^+\notin\tau(J)$ and ${\sigma'}_2^+\neq \sigma_1^+$.
  Indeed, if we supposed that $\sigma^+_1={\sigma'}^+_2$ then since $\sigma^+_1$ is adjacent to $J_1$ in $\Sigma^+$ it is after an interval $K_1'$ of $J_1$ in $\Sigma^+$ and just before the interval $K_2$ of $J$ in $\Sigma^+$. Hence in the cyclic order we have in $\Sigma^+$ 
  $$K_1',\sigma_1^+={\sigma'}_2^+, K_2, \sigma^+_2 , K_1.$$ 
 So in $\Sigma_1^+$ we obtain in the cyclic order 
$$ K_1', \sigma_1^+,K_1.$$ 
Meaning that $\sigma_1^+$ is surrounded by $J_1$ in $\Sigma_1^+$ and so it contradicts the fact that $\sigma_1^+$ is adjacent to $J_1$ in $\Sigma_1^+$.\\
Moreover, by hyothesis, $\alpha(\sigma_1^+)=\alpha(\sigma_2^+)\in \alpha(\tau(J) \Delta  \tau^{-1}(J))$ so by construction $\alpha_1(\sigma_1^+)=\alpha_1({\sigma'}_2^+)$ and we obtain the result.\\

In case 2), ${\sigma'}^+_1$ is just before a maximal interval $K'_1$ of $J_1$ in $\Sigma^+$. Moreover, by hypothesis, there exists $\sigma^+_2\in\Sigma^+\backslash\{{\sigma'}^+_1\}$ adjacent to $J_1$ in $\Sigma^+$ such that $\alpha({\sigma'}_1^+)=\alpha(\sigma_2^+)$.
Again, there are two possibilities.\\

i) If $\sigma_2^+\notin \tau(J)$ then $\sigma_2^+$ is adjacent to $J_1$ in $\Sigma_1^+$ and distinct from $\sigma^+_1$. Indeed, if we suppose that $\sigma_2^+=\sigma^+_1$ then $\sigma^+_1$ is before the maximal interval $K$ of $J$ in $\Sigma^+$ and after a maximal interval $K_1$ of $J_1$ in $\Sigma^+$.
Hence, in $\Sigma^+$, we have and in the cyclic order
$$K_1,\sigma^+_1=\sigma_2^+,K,{\sigma'}^+_1,K'_1.$$
So in $\Sigma_1^+$ we obtain in the cyclic order
$$K_1,\sigma^+_1,K'_1.$$
Meaning that $\sigma_1^+$ is surrounded by $J_1$ in $\Sigma_1^+$ and it contradicts the fact that $\sigma_1^+$ is adjacent to $J_1$ in $\Sigma_1^+$.\\
Moreover, by hypothesis, $\alpha({\sigma'}^+_1)=\alpha(\sigma^+_2)\in \alpha(\tau(J) \Delta  \tau^{-1}(J))$ and by definition $\alpha(\sigma^+_1)\in\alpha(\tau(J) \Delta  \tau^{-1}(J))$. So, by construction, we have $\alpha_1(\sigma_1^+)=\alpha_1(\sigma_2^+)$ and we obtain the result.\\

ii) If $\sigma_2^+\in \tau(J)$ then $\sigma_2^+$ is adjacent and after a maximal interval $K_2$ of $J$ in $\Sigma^+$. We consider ${\sigma'}_2^+$ in $\Sigma^+$ just before the interval $K_2$ in $\Sigma^+$. We have $K\neq K_2$ then ${\sigma'}_2^+\neq \sigma^+_1$ and $\alpha({\sigma'}_2^+)= \alpha( \sigma^+_1)$. Indeed, if we suppose that $K_2=K$ then it contradicts the fact that $\sigma^+_2\neq {\sigma'}_2^+$ and we obtain the result.\\

Now, we prove that $\#\im(\alpha)+\#\im(\omega)\leq i+2$.\\

The image of $\omega$ is the union of the image of $\omega_1$ and the singleton $\{c\}$. So we have $\# \im(\omega)=\#\im(\omega_1)+1$.\\

Let us consider the natural projection $p:E^+\ra E^+/\alpha(\tau(J) \Delta  \tau^{-1}(J))$, we have the following diagram

$$\xymatrix{
    \Sigma_1^+  \ar[r]^\alpha \ar[rd]^{\alpha_1}  & E^+ \ar[d]^p\\
     & E^+/\alpha(\tau(J) \Delta  \tau^{-1}(J))
}$$
We denote $e\in E^+/\alpha(\tau(J)\Delta \tau^{-1}(J))$ such that $p(\alpha(\sigma))=e$, for every $p(c)\in \tau(J)\Delta \tau^{-1}(J)$.\\

The image of $\alpha$ satisfies 
$$\im(\alpha)=\alpha({\tau(J)\cup  \tau^{-1}(J)})\cup \alpha(\Sigma^+\backslash \{\tau(J)\cup  \tau^{-1}(J)\}).$$
By definition, we have 
$$\tau(J)\cup  \tau^{-1}(J)= (\tau(J) \Delta  \tau^{-1}(J) ) \cup ( \tau(J)\cap  \tau^{-1}(J)).$$
So, we deduce
 $$\alpha(\tau(J)\cup  \tau^{-1}(J))= \alpha(\tau(J) \Delta  \tau^{-1}(J))\cup \alpha(\tau(J)\cap  \tau^{-1}(J)).$$
We denote by $K$ the number of maximum intervals of $J$.Since $J\neq \Sigma^+$, we have $\#\tau(J) \Delta  \tau^{-1}(J)= 2K$ and $\# (\tau(J)\cap  \tau^{-1}(J))= \#J -K$. Indeed, the set $\tau(J) \Delta  \tau^{-1}(J)$ is the set of elements of $\Sigma^+$ which are adjacent to $J$ and the set $\tau(J)\cap  \tau^{-1}(J)$ is the set of elements of $\Sigma^+$ which are surrounded by $J$.\\

By hypothesis, for every element $\sigma^+\in \tau(J) \Delta  \tau^{-1}(J)$ there exists another element ${\sigma'}^+\in\tau(J) \Delta  \tau^{-1}(J)$ such that $\alpha({\sigma}^+)=\alpha({\sigma'}^+)$. So we obtain
$$\#\alpha(\tau(J) \Delta  \tau^{-1}(J)) \leq \frac{2K}{2}\leq K.$$
Hence we compute
\begin{align*}
\# \alpha(\tau(J)\cup  \tau^{-1}(J)) & \leq \#  \alpha(\tau(J) \Delta  \tau^{-1}(J)) +\# \alpha(\tau(J)\cap  \tau^{-1}(J)) \\
&\leq K +\#J -K\\
&\leq \# J.
\end{align*}

It remains to estimate the cardinal of $C=\im( \alpha)\backslash \{\alpha(\tau(J)\cup  \tau^{-1}(J)\})\}$. For every $c\in C$ and $\sigma \in \alpha^{-1}(c)$, we have $\sigma \notin \alpha(\tau(J)\cup  \tau^{-1}(J))$. So in particular, $\sigma\in \Sigma_1^+$ and $p(c)\neq e$. We deduce that 
$$p(c)=\alpha_1(\sigma).$$
It implies that 
$$p(C)\subset \im (\alpha_1)\backslash \{e\}.$$
Moreover, by definition $C \cap p^{-1}(e)=\emptyset$, so $p|_C$ is a bijection and, since $\{e\}\in \im(\alpha_1) $, we obtain
$$\# C\leq \# \im(\alpha_1) -1.$$
Thus we have
$$\im(\alpha) \leq \#J+\# \im(\alpha_1)-1.$$
The couple $(\alpha_1,\omega_1)$ is defined on a set of cardinal $i-\#J+1$ and satisfies the hypothesis of Lemma \ref{combi}, so we compute 

\begin{align*}
\# \im(\alpha)+\#\im(\omega)&\leq (\#J+\#\im(\alpha_1)-1)+\#\im(\omega_1)+1\\
& \leq \#\im(\alpha_1)+\#\im(\omega_1)+\# J\\
& \leq (i-\#J+2) +\#J\\
&\leq i+2,
\end{align*}
where the third inequality is given by the induction step.\\
 
\end{proof}

Notice that if $J$ is composed of at least $2$ maximal intervals, the previous inequality is strict.\\

To complete the proof of Proposition \ref{connexionsaddle} we will prove that the couple of maps $\alpha:\Sigma_x^+\ra \mathcal{C}_t^+(\FF)$ and $\omega:\Sigma^-_x\ra \mathcal{C}_t^-(\FF)$ defined at the begining of the proof satisfies the hypothesis of Lemma \ref{combi}.\\

Let $\mathcal{C}$ be a connected component of $G_t^-(\FF)$ in the image of $\omega$. We denote $B=W^s(\mathcal{C})$ and we want to desingularize its frontier $\Fr(B)$ as follows.\\

\emph{Desingularization.} We cut the surface $\Sigma$ along $\Fr(B)$ (see \cite{Bou08} for example) to obtain a manifold with boundary $\hat{B}$ and a natural projection $\pi :\hat{B} \ra \overline{B}$ such that 
\begin{itemize}
\item $\pi(\partial\hat{B})=\Fr(B)$.
\item $\pi(\hat{B}\backslash\partial\hat{B})\simeq B$.
\end{itemize} 

Let us draw simple examples to explain what we are doing. \\

\textit{First example.} In Figure \ref{desing}, we consider on the left a saddle point $y$ in $\Fr(B)$ such that
\begin{itemize}
\item There are two hyperbolic sectors $U'$ and $U''$ separated by a cone $\sigma$ in $B$. We denote $\phi'$ and $\phi''$ the leaves of $U'$ and $U''$ in $\Fr(B)$.
\item There exists a leaf $\psi$ in $\sigma\cap \Fr(B)$. 
\end{itemize} 
We draw in red the leaves in $B$ and in blue the leaves in $\Fr(B)$. We cut along leaves of $\Fr(B)$ in blue to obtain $\hat{B}$ on the right and we have
\begin{itemize}
\item $\pi^{-1}(y)=\{\hat{y}_1,\hat{y}_2\}$.
\item $\pi^{-1}(\psi)=\{\hat\psi_1,\hat\psi_2\}$.
\item $\pi^{-1}(\phi'')=\{\hat\phi'\}$.
\item $\pi^{-1}(\phi')=\{\hat\phi''\}$.
\end{itemize}

\begin{figure}[H]
\begin{center}
\begin{tikzpicture}[scale=1]

\draw [decoration={markings, mark=at position 0.5 with {\arrow{>}}}, postaction={decorate}] (-2,0.3) -- (0,0);
\draw [decoration={markings, mark=at position 0.5 with {\arrow{<}}}, postaction={decorate}]  (0.3,2) ..controls +(0,-2) and +(-2,+0)..(2,0.3);
\draw [decoration={markings, mark=at position 0.5 with {\arrow{>}}}, postaction={decorate}]  (-2,0.5) ..controls +(2,-0.3) and +(0,-2)..(-0.3,2)  ;

\draw [decoration={markings, mark=at position 0.5 with {\arrow{<}}}, postaction={decorate}] (0,2) -- (0,0);
\draw [decoration={markings, mark=at position 0.5 with {\arrow{>}}}, postaction={decorate}, blue] (-2,0) -- (0,0);
\draw [decoration={markings, mark=at position 0.5 with {\arrow{<}}}, postaction={decorate}, blue] (0,0) -- (2,0);
\draw [decoration={markings, mark=at position 0.5 with {\arrow{>}}}, postaction={decorate}, blue] (0,0) -- (0,-1.5);
\draw [decoration={markings, mark=at position 0.5 with {\arrow{>}}}, postaction={decorate}, red] (0,-1.5) -- (0,-2);
\draw [decoration={markings, mark=at position 0.5 with {\arrow{>}}}, postaction={decorate}, red] (0,0) -- (0.3,-2);
\draw [decoration={markings, mark=at position 0.5 with {\arrow{>}}}, postaction={decorate}, red] (0,0) -- (-0.3,-2);
\draw [decoration={markings, mark=at position 0.5 with {\arrow{>}}}, postaction={decorate}]  (-2,0.5) ..controls +(2,-0.3) and +(0,-2)..(-0.3,2)  ;
\draw [decoration={markings, mark=at position 0.5 with {\arrow{>}}}, postaction={decorate},red]  (-2,-0.3) ..controls +(2,0) and +(0.1,2)..(-0.5,-2)  ;
\draw [decoration={markings, mark=at position 0.5 with {\arrow{<}}}, postaction={decorate},red]  (2,-0.3) ..controls +(-2,0) and +(-0.1,2)..(0.5,-2)  ;

\draw [decorate, decoration = {brace}]  (0.3,-2.1) -- (-0.3,-2.1);
\draw[below] (0,-2.1) node{$\sigma$};
\draw (0,-1.5) node{$\bullet$};
\draw[right,blue] (0,-1.5) node{$\psi$};
\draw (0,0) node{$\bullet$};
\draw[above] (0,0.1) node{$y$};
\draw[red] (-1,-1) node{$U'$};
\draw[red] (1,-1) node{$U''$};
\draw[left,blue] (-2,0) node{$\phi'$};
\draw[right,blue] (2,0) node{$\phi''$};

\draw (3,0) node{$\leadsto$};

\draw [decoration={markings, mark=at position 0.5 with {\arrow{>}}}, postaction={decorate},blue] (4,-0.2) -- (5.5,-0.2);
\draw [decoration={markings, mark=at position 0.5 with {\arrow{>}}}, postaction={decorate},blue] (5.5,-0.2) -- (6,-1.5);
\draw [decoration={markings, mark=at position 0.5 with {\arrow{<}}}, postaction={decorate},blue] (6.5,-0.2) -- (8,-0.2);
\draw [decoration={markings, mark=at position 0.5 with {\arrow{>}}}, postaction={decorate},blue] (6.5,-0.2) -- (6,-1.5);

\draw [decoration={markings, mark=at position 0.5 with {\arrow{>}}}, postaction={decorate},red] (6,-1.5) -- (6,-2);
\draw [decoration={markings, mark=at position 0.5 with {\arrow{>}}}, postaction={decorate},red] (5.5,-0.2) -- (5.2,-2);
\draw [decoration={markings, mark=at position 0.5 with {\arrow{>}}}, postaction={decorate},red] (6.5,-0.2) -- (6.8,-2);
\draw [decoration={markings, mark=at position 0.5 with {\arrow{>}}}, postaction={decorate},red]  (4,-0.4) ..controls +(1,0) and +(0.1,1)..(5,-2)  ;
\draw [decoration={markings, mark=at position 0.5 with {\arrow{>}}}, postaction={decorate},red]  (8,-0.4) ..controls +(-1,0) and +(-0.1,1)..(7,-2)  ;

\draw [decorate, decoration = {brace}]  (6.8,-2.1) -- (5.2,-2.1);
\draw[below] (6,-2.1) node{$\hat\sigma$};
\draw (6,-1.5) node{$\bullet$};
\draw (5.5,-0.2) node{$\bullet$};
\draw (6.5,-0.2) node{$\bullet$};
\draw[left,blue] (4,-0.2) node{$\hat\phi'$};
\draw[right,blue] (8,-0.2) node{$\hat\phi''$};
\draw[blue] (6.4,-1.2) node{$\hat\psi_2$};
\draw[blue] (5.7,-1.2) node{$\hat\psi_1$};
\draw[above] (5.6,-0.2) node{$\hat y_1$};
\draw[above] (6.4,-0.2) node{$\hat y_2$};
\draw[red] (4.5,-1) node{$U'$};
\draw[red] (7.5,-1) node{$U''$};

\end{tikzpicture}
\end{center}
\caption{ }
\label{desing}
\end{figure}

\textit{Second example.} In Figure \ref{desing2} we consider a saddle point $y$ in $\Fr(B)$ such that 
\begin{itemize} 
\item There is a stable cone surrounded by $B$ which contains two leaves $\phi_1$ and $\phi_2$ in $\Fr(B)$.
\item There are two leaves $\phi'$ and $\phi''$ of stable cones of $y$ in $\Fr(B)$.
\end{itemize}
We draw in red the leaves in $B$, in blue the leaves in $\Fr(B)$. We cut along the leaves of $\Fr(B)$ to obtain $\hat{B}$ on the right and we have
\begin{itemize}
\item $\pi^{-1}(y)=\{\hat{y}_1,\hat{y}_2\}$.
\item $\pi^{-1}(\psi_1)=\{\hat\psi_1\}$.
\item $\pi^{-1}(\psi_2)=\{\hat\psi_2\}$
\item $\pi^{-1}(\phi'')=\{\hat\phi'\}$.
\item $\pi^{-1}(\phi')=\{\hat\phi''\}$.
\end{itemize}

 \begin{figure}[H]
\begin{center}
\begin{tikzpicture}[scale=1]

\draw [decoration={markings, mark=at position 0.5 with {\arrow{>}}}, postaction={decorate}] (-2,0.3) -- (0,0);
\draw [decoration={markings, mark=at position 0.5 with {\arrow{<}}}, postaction={decorate}]  (0.3,2) ..controls +(0,-2) and +(-2,+0)..(2,0.3);
\draw [decoration={markings, mark=at position 0.5 with {\arrow{>}}}, postaction={decorate}]  (-2,0.5) ..controls +(2,-0.3) and +(0,-2)..(-0.3,2)  ;

\draw [decoration={markings, mark=at position 0.5 with {\arrow{<}}}, postaction={decorate}] (0,2) -- (0,0);
\draw [decoration={markings, mark=at position 0.5 with {\arrow{>}}}, postaction={decorate}, blue] (-2,0) -- (0,0);
\draw [decoration={markings, mark=at position 0.5 with {\arrow{<}}}, postaction={decorate}, blue] (0,0) -- (2,0);
\draw [decoration={markings, mark=at position 0.5 with {\arrow{<}}}, postaction={decorate}] (0,0) -- (0,-1.5);
\draw [decoration={markings, mark=at position 0.5 with {\arrow{<}}}, postaction={decorate}, blue] (0,0) ..controls +(-0.2,-0.2) and +(-0.2,0.2).. (0,-1.5);
\draw [decoration={markings, mark=at position 0.5 with {\arrow{<}}}, postaction={decorate}, blue] (0,0) ..controls +(0.2,-0.2) and +(0.2,0.2).. (0,-1.5);
\draw [decoration={markings, mark=at position 0.5 with {\arrow{>}}}, postaction={decorate}, red] (0,-1.5) -- (0,-2);
\draw [decoration={markings, mark=at position 0.5 with {\arrow{>}}}, postaction={decorate}, red] (0,0) -- (-1.5,-2);
\draw [decoration={markings, mark=at position 0.5 with {\arrow{>}}}, postaction={decorate}, red] (0,0) -- (1.5,-2);
\draw [decoration={markings, mark=at position 0.5 with {\arrow{>}}}, postaction={decorate}]  (-2,0.5) ..controls +(2,-0.3) and +(0,-2)..(-0.3,2)  ;

\draw [decoration={markings, mark=at position 0.5 with {\arrow{>}}}, postaction={decorate},red]  (-2,-0.3) ..controls +(2,0) and +(1,2)..(-1.7,-2)  ;
\draw [decoration={markings, mark=at position 0.5 with {\arrow{>}}}, postaction={decorate},red]  (2,-0.3) ..controls +(-2,0) and +(-1,2)..(1.7,-2)  ;
\draw [decoration={markings, mark=at position 0.5 with {\arrow{>}}}, postaction={decorate}, red] (0,-1.5) ..controls +(-0.2,0.2) and +(0.2,0.2).. (-0.5,-2);
\draw [decoration={markings, mark=at position 0.5 with {\arrow{>}}}, postaction={decorate}, red] (0,-1.5) ..controls +(0.2,0.2) and +(-0.2,0.2).. (0.5,-2);

\draw (0,-1.5) node{$\bullet$};
\draw[right,blue] (0.1,-1.2) node{$\psi_2$};
\draw[left,blue] (-0.1,-1.2) node{$\psi_1$};
\draw (0,0) node{$\bullet$};
\draw[above] (0,0.1) node{$y$};

\draw[left,blue] (-2,0) node{$\phi'$};
\draw[right,blue] (2,0) node{$\phi''$};

\draw (3,0) node{$\leadsto$};

\draw [decoration={markings, mark=at position 0.5 with {\arrow{>}}}, postaction={decorate},blue] (4,-0.2) -- (5.5,-0.2);
\draw [decoration={markings, mark=at position 0.5 with {\arrow{<}}}, postaction={decorate},blue] (5.5,-0.2) -- (6,-1.5);
\draw [decoration={markings, mark=at position 0.5 with {\arrow{<}}}, postaction={decorate},blue] (6.5,-0.2) -- (8,-0.2);
\draw [decoration={markings, mark=at position 0.5 with {\arrow{<}}}, postaction={decorate},blue] (6.5,-0.2) -- (6,-1.5);

\draw [decoration={markings, mark=at position 0.5 with {\arrow{>}}}, postaction={decorate},red] (6,-1.5) -- (6,-2);
\draw [decoration={markings, mark=at position 0.5 with {\arrow{>}}}, postaction={decorate},red] (6,-1.5) ..controls +(-0.2,0.2) and +(0.2,0.2).. (5.5,-2);
\draw [decoration={markings, mark=at position 0.5 with {\arrow{>}}}, postaction={decorate},red] (6,-1.5) ..controls +(0.2,0.2) and +(-0.2,0.2).. (6.5,-2);
\draw [decoration={markings, mark=at position 0.5 with {\arrow{>}}}, postaction={decorate},red] (5.5,-0.2) -- (5.2,-2);
\draw [decoration={markings, mark=at position 0.5 with {\arrow{>}}}, postaction={decorate},red] (6.5,-0.2) -- (6.8,-2);
\draw [decoration={markings, mark=at position 0.5 with {\arrow{>}}}, postaction={decorate},red]  (4,-0.4) ..controls +(1,0) and +(0.1,1)..(5,-2)  ;
\draw [decoration={markings, mark=at position 0.5 with {\arrow{>}}}, postaction={decorate},red]  (8,-0.4) ..controls +(-1,0) and +(-0.1,1)..(7,-2)  ;

\draw (6,-1.5) node{$\bullet$};
\draw (5.5,-0.2) node{$\bullet$};
\draw (6.5,-0.2) node{$\bullet$};
\draw[left,blue] (4,-0.2) node{$\hat\phi'$};
\draw[right,blue] (8,-0.2) node{$\hat\phi''$};
\draw[blue] (6.4,-1.2) node{$\hat\psi_2$};
\draw[blue] (5.7,-1.2) node{$\hat\psi_1$};
\draw[above] (5.6,-0.2) node{$\hat y_1$};
\draw[above] (6.4,-0.2) node{$\hat y_2$};

\end{tikzpicture}
\end{center}
\caption{ }
\label{desing2}
\end{figure}

We have that $\hat{B}$ is a manifold with boundary whose boundary is a union of disjoint circles. By Proposition \ref{Propbord}, each circle of $\partial \hat B$ is composed of chains of connexions of the foliation $\FF$ such that for every leaf $\phi$ in $\Fr(B)$ we have
\begin{itemize}
\item $\#\pi^{-1}(\phi)=1$ if $\phi$ is adjacent to $B$.
\item $\#\pi^{-1}(\phi)=2$ if $\phi$ is in the interior of $\overline{B}$.
\end{itemize}

\begin{rem}
As we saw in the second example with the leaves $\psi_1$ and $\psi_2$, a leaf $\phi$ can be adjacent to $B$ and in a stable cone of a saddle point of $\Fr(B)$ which is surrounded by $B$
\end{rem}

We set $J=\omega^{-1}(\mathcal{C})$ and recall that we set $t=A(x)$ at the begining. By definition, the saddle point $x$ is in the boundary of $B$. We consider a stable cone $\sigma^+\in\Sigma^+$ of $x$ adjacent to $J$ in $\Sigma^+$. To prove that the applications $\alpha$ and $\omega$ satisfy the hypothesis of Lemma \ref{combi} it is sufficient to prove that there exists another stable cone ${\sigma'}^+$ of $x$ adjacent to $J$ in $\Sigma^+$ such that $\alpha({\sigma'}^+)=\alpha({\sigma}^+)$.\\

The cone $\sigma^+$ is adjacent to $B$ so, by Proposition \ref{Propbord}, there is a unique leaf $\phi^+_0$ in $\sigma^+\cap \Fr(B)$. The cone $\sigma^+$ is adjacent to $J$ in $\Sigma^+$, so we can denote $\pi^{-1}(\phi^+_0)=\{\hat\phi_0^+\}\subset \partial \hat{B}$.\\

Moreover, $x$ is the only singularity of action $t$, so each singularity $\hat y$ of $\partial \hat B\backslash \pi^{-1}(x)$ satisfies $A_f(\pi(\hat y))>t$. So, if we consider a circle $\gamma$ of $\partial \hat B$, the singularities in a same connected component of $\gamma\backslash \pi^{-1}(x)$ are the lift of singularities which are in the same connected component of $G_t^+(\FF)$. Notice that $\gamma\backslash \pi^{-1}(x)$ may be composed of more than one connected components. \\

We consider the circle $\gamma_0$ of $\partial \hat B$ containing $\hat\phi_0^+$. The connected component of $\gamma_0\backslash\pi^{-1}(x)$ containing $\hat\phi_0^+$ contains a leaf $\hat{\phi'}_0^+$ distinct from $\hat \phi_0^+$ such that $\omega(\pi(\hat{\phi'}_0))=x$.  There are two cases.\\

1) If $\pi(\hat{\phi'}_0)$ is in a stable cone of $x$ adjacent to $J$ in $\Sigma^+$ then we obtain the result.\\

2) Suppose that $\pi(\hat{\phi'}_0)$ is in a stable cone $\sigma^+_1$ of $x$ surrounded by $J$ in $\Sigma^+$. By construction, there exists another leaf $\hat\phi_1^+\in \partial\hat B$ distinct from $\hat{\phi'}_0^+$ such that $\pi(\hat{\phi}_1^+)$, which may be equal to $\pi(\hat{\phi'}_0^+)$, is in $\sigma^+_1$.

By Lemma \ref{Sigma-}, we have that $\alpha(\pi({\hat\phi}_1^+))=\alpha(\pi(\hat{\phi'}_0^+))$. So we apply the same arguments to ${\hat\phi}_1^+$ and then there exists a leaf $\hat{\phi'}_1^+\subset\partial\hat B$ such that $\hat{\phi'}_1^+$ and ${\phi}_1^+$ are in the same connected component of $\gamma_1\backslash\{\pi^{-1}(x)\}$ and $\omega(\pi(\hat{\phi'}_1^+))=x$. If ${\phi'}_1^+$ is in a stable cone of $x$ adjacent to $J$ in $\Sigma^+$ we stop the process and if not we do the same discussion for ${\phi'}_1^+$ as we did for ${\phi'}_0^+$ in case 2).\\

Since $x$ has finitely many stable cones, the process stops after a finite number of times and we obtain a leaf $\phi'^+$ in the frontier of the attractive basin of $\mathcal{C}$ distinct from $\phi^+$ which is in a stable cone ${\sigma'}^+$ of $x$ adjacent to $J$ in $\Sigma^+$ such that $\alpha({\sigma'}^+)=\alpha(\sigma^+)$ and this ends the proof of Proposition \ref{connexionsaddle}.
\end{proof}

By construction of $\mathcal{B}$, we deduce the following corollary. 

\begin{coro}\label{numberbar}
Let us suppose that the singularities of $\FF$ have distinct action values and consider a saddle point $x$ of $\FF$. There exist exactly $-\ind(f,x)$ bars $J_1,...,J_{-\ind(f,x)}$ of which $A(x)$ is an end point.\\
Moreover, for each source or sink $y$ of $X$ there exists exactly one bar $J$ of which $A(y)$ is an end point.
\end{coro}

\begin{proof}
Let us consider a saddle point $x$ of $\FF$ of action $t$. We denote $\mathcal{C}_x$ (resp. $\mathcal{C}'_x$ the connected component of $G_{t^+}^-(\FF)$ (resp. $G_{t^-}^+(\FF)$) which contains $x$. By construction, there are $\#j_t^{-1}(\mathcal{C}_x)$ bars of category $1$ such that $A(x)$ is the maximum and $\#{j'}_t^{-1}(\mathcal{C}'_x)$ bars of category $2$ such that $A(x)$ is the infimum. Finally, by Proposition \ref{connexionsaddle} we have $k=|\ind(\FF,x)|- \# j^{-1}_t(\mathcal{C}_x)-\# j'^{-1}_t(\mathcal{C}'_x)\geq 0$ so there exists $k$ bars $(A_f(x),+\infty)$ of category $3$.
Thus there are exactly $-\ind( \FF,x)$ bars of which $A_f(x)$ is an end.
\end{proof}

Recall that $X_{<}$ is the set of singularities of $X$ of negative index and $X_1$ the set of singularities of $X$ of index $1$. We set $S_{<0}=A(X_{<0})$ and we consider a sink $y_m$ of the foliation $\FF$ such that $A(y_m)=A_m$ and a source $y_M$ of the foliation $\FF$ such that $A(y_M)=A_M$ and an application $\iota: X\backslash\{y_m,y_M\}\ra X_{<0}$ given by \ref{iota}. We will prove the following result about the semi-infinite bars of $\bm{\beta}_\FF$. \\

\begin{lem}\label{barsemi}
There exist exactly $2g+2$ semi-infinite bars in $\bm{\beta}_\FF$, where $g$ is the genus of $\Sigma$.
\end{lem}

\begin{proof}
By Lemma \ref{sink-source} for each value $t\in S_{<0}$ the number of finite bars of which one end is the action value of $t$ is equal to 
$$\left(\sum_{C\in \mathcal{C}^-_{t^+}}(\#j_{t}^{-1}(\{C\})-1)+\sum_{C'\in \mathcal{C}^+_{t-}} (\# {j'}_t^{-1}(\{C'\})-1)\right).$$

Moreover, the existence of the application $\iota$ asserts that the total number of finite bars is equal to the number of singularities of index $1$ minus two. So we compute
\begin{equation}\label{sumnumberbar}
\begin{split}
\sum_{y\in X_1}1&=2+\sum_{y\in X_1\backslash \{y_m,y_M\}}\ind(\FF,y)\\
&=2+\sum_{t\in S_{<0}}\left(\sum_{C\in \mathcal{C}^-_{t^+}}(\#j_{t}^{-1}(\{C\})-1)+\sum_{C\in \mathcal{C}^+_{t-}} (\# {j'}_t^{-1}(\{C'\})-1)\right).
\end{split}
\end{equation}

Moreover, we have
\begin{align*}
2-2g &=\sum_{x\in X} \ind(\FF,x)\\
&=2+\sum_{x\in X\backslash \{y_m,y_M\}} \ind(\FF,x).\\
&=2+\sum_{x\in X_{<0}}\ind(\FF,x) +\sum_{y\in X_1\backslash \{y_m,y_M\}} \ind(\FF,y).
\end{align*}
Where the last equality is given by separating the fixed points of negative index and the fixed points of index  $1$.\\

 For $t\in S_{<0}$ we define 
$$k_t=-\sum_{\substack{x\in X_{<0} \\ A(x)=t}} \ind(\FF,x).$$

By equation \ref{sumnumberbar}, we get 
\begin{align*}
2-2g &=2+\sum_{t\in S_{<0}}\left(\sum_{\substack{x\in X_{<0} \\ A(x)=t}} \ind(\FF,x)\right)\\
& \ \ \ \ +\sum_{t\in S_{<0}}\left(\sum_{C\in \mathcal{C}^-_{t^+}}(\#j_{t}^{-1}(\{C\})-1)+\sum_{C'\in \mathcal{C}^+_{t-}} (\# {j'}_t^{-1}(\{C'\})-1)\right)\\
&=2- \sum_{t\in S_{<0}}\left( k_t-\sum_{C\in \mathcal{C}^-_{t^+}}(\#j_{t}^{-1}(\{C\})-1)-\sum_{C'\in \mathcal{C}^+_{t-}}( \# ({j'}_t^{-1}(\{C'\})-1)\right)\\
&=2-(\#\{\text{semi-infinite bars}\}-2).
\end{align*}
The last equation is given by the construction of the bars of category $3$ in the construction of $\mathcal{B}$. Indeed, for action value $t$, the number of bars $(t,+\infty)$ in the barcode is equal, by definition, to $k_t-\sum_{C\in \mathcal{C}^-_{t^+}}(\#j_{t}^{-1}(\{C\})-1)-\sum_{C'\in \mathcal{C}^+_{t-}}( \# ({j'}_t^{-1}(\{C'\})-1)$. The $-2$ comes from the two semi-infinite bars $(A_m,+\infty)$ and $(A_M,+\infty)$. So the number of semi-infinite bars is equal to $2g+2$.
\end{proof}

\begin{rem}
Let us suppose that for each value $t$ in the image of $A$, the set $A^{-1}(\{t\})$ is a singleton and each singularity $x$ of $X_{<0}$ has index $-1$. In this case the proof is simpler to understand because $\iota$ is an injection.

The number of semi-infinite bars is equal to $\#(X_{<0}\backslash \im(\iota))+2$ and we compute
\begin{align*}
2-2g &=\sum_{x\in X} \ind(\FF,x)\\
&=2+\sum_{y\in X_1\backslash \{y_m,y_M\}} \ind(\FF,y)+\sum_{x\in X_{<0}}\ind(\FF,x) \\
&=2 +\sum_{y\in X_1\backslash \{y_m,y_M\}}1-\sum_{x\in \im(\iota)}1-(\#\{\text{semi-infinite bars}\}-2)\\
&=4-\#\{\text{semi-infinite bars}\}.
\end{align*}

We obtain the result.
\end{rem}

			 	\section{A barcode with an order on a maximal unlinked set of fixed points}\label{order}
						
We consider a homeomorphism $f$ of a compact surface $\Sigma$ with a finite number of fixed points. We fix a maximal unlinked set  of fixed points $X$ of $f$. We denote by $\D$ the unit disk.\\

For a set $U\subset \D$ we will denote by $\Adh_\D(U)$ its closure in $\D$ and by $\Adh_{\R^2}{U}$ its closure in $\R^2$. A line $\gamma$ will be a proper oriented topological embedding of the interval $(0,1)$. If an oriented line $\gamma: (0,1)\ra \D$ separates $\D$ in two connected components, we will consider its left hand side, denoted $L(\ti\gamma)$, and its right-hand side, denoted $R(\ti\phi)$.  \\

Let us consider a maximal isotopy $I$ from $\id$ to $f$ such that $\Sing(I)=X$. We equipped $\Sigma\backslash X$ with a hyperbolic metric $m$ such that the universal cover $\ti{\Sigma\backslash X}$ of $\Sigma\backslash X$ with the pull back metric is isomorphic to $\D$.\\

We consider a singularity $x\in X$ and a path $\gamma:[0,1]\ra \Sigma$ such that $\lim_{t\ra 1} \gamma(t) =x$ and $\gamma((0,1])\subset \Sigma\backslash X$. We fix a lift $\ti \gamma:(0,1) \ra\D$ of $\gamma|_{(0,1)}$.
We consider a horospherical neighborhood $V$ of $x$ in $(\Sigma\backslash X,m)$. Meaning that $\pi^{-1} (V)$ is a disjoint union of horodisks where a horodisk is a disk internally tangent to the unit circle. Notice that $V\cup \{x\}$ is a topological neighborhood of $x$ in $\Sigma$.\\

Moreover, $\ti\gamma((0,\epsilon])$ is connected so there exists a unique horoball $\ti V\subset \pi^{-1} (V)$ which contains $\ti\gamma((0,\epsilon])$. By definition, the closure of a horoball intersects the boundary $\Ss^1$ of $\D$ in exactly one point and we set $\ti x = \Adh_{\R^2}(V)\cap \Ss^1$. Since the alpha-limit point of $\ti\gamma$ is a point of $\Ss^1$ we obtain that $\lim_{t\ra 0} \ti\gamma(t)=\ti x$. \\

Thus, if we consider a path $\gamma:[0,1]\ra \Sigma$ such that $\alpha(\gamma)=x\in X$, $\omega(\gamma)=y\in X$ and $\gamma((0,1))\subset \Sigma\backslash X$, then for every lift $\ti \gamma$ of $\gamma$ is line in $\D$ such that there are two points $\ti x$ and $\ti y$ of $\Ss^1$ which satisfy
$$\lim_{t\ra 0} \ti \gamma(t)=\ti x, \text{ and } \lim_{t\ra 1} \ti\gamma(t)=\ti y.$$
The point $\ti x$ (resp. $\ti y$) will be denoted $\alpha(\ti \gamma)$ (resp. $\omega(\ti \phi)$). We refer to Ratcliffe's book \cite{RAT}, chapter $9.8$ for more details.\\
			
For two distinct points $\ti x$ and $\ti y$ of $\Ss^1$, we define $[\ti x,\ti y]\subset \Ss^1$ the arc joining $\ti x$ to $\ti y$ for the usual orientation. We set $(\ti x,\ti y)=[\ti x,\ti y]\backslash\{\ti x,\ti y\}$.\\

Let us consider two proper paths $\phi:[0,1]\ra \Sigma\backslash X$ and $\phi':[0,1]\ra \Sigma\backslash X$ whose alpha and omega limit points are in $X$ and satisfy $\phi((0,1))\subset  \Sigma\backslash X$ and $\phi'((0,1))\subset  \Sigma\backslash X$. We say that $\phi$ and $\phi'$ \textit{strongly intersect} if there exist two lifts $\ti \phi$ and $\ti\phi'$ of $\phi$ and $\phi'$ as follows.\\
 
If we denote $\alpha(\ti\phi)=\ti x$, $\omega(\ti\phi)=\ti y$, $\alpha(\ti\phi')=\ti x'$ and $\omega(\ti\phi')=\ti y'$, then we have 
\begin{itemize}
\item $\ti x, \ \ti y, \ \ti x', \ \ti y'$ are distinct,
\item  $\left\{\begin{matrix}
\ti x' \in (\ti x, \ti y)\\
\ti y' \in (\ti y, \ti x)
\end{matrix}\right.$ 
or $\left\{\begin{matrix}
\ti x' \in (\ti y, \ti x)\\
\ti y' \in (\ti x, \ti y)
\end{matrix}\right.$.
\end{itemize}

The homeomorphism $\ti f$ can be extended on the closed unit disk $\Adh_{\R^2}(\D)$. A line $\gamma\subset \D$ is said to be a \textit{Brouwer line} of $\ti f$ if it separates $\D$ in two connected components such that the one on the left-hand side, denoted $L(\gamma)$, contains $\ti f(\gamma)$ and the one on the right-hand side, denoted $R(\gamma)$, contains $\ti f^{-1}(\gamma)$. In particular, we have $\ti f(\Adh_\D(L( \gamma)))\subset L(\gamma)$. We have the following definition.\\

\begin{definition}
For a couple $x,y\in X$, an oriented path $\gamma: [0,1]\ra \Sigma$ such that $\alpha(\gamma)=x\in X$, $\omega(\gamma)=y\in X$ and $\gamma((0,1))\subset \Sigma\backslash X$ is called a \textit{connexion} from $x$ to $y$ if every lift $\ti{\gamma}$ of $\gamma$ is an oriented Brouwer line of $\ti{f}$. 
\end{definition}

\begin{rem}
If we consider a maximal isotopy $I=(f_t)_{\tin}$ from $\id$ to $f$ and a foliation $\FF$ positively transverse to $I$ then each leaf $\phi$ of $\FF$ is a connexion in the previous sense. Indeed, if we consider a leaf $\phi$ of $\FF$, then a lift $\ti \phi$ of $\phi$ is an oriented line which separates $\D$ in two connected components $L(\ti\phi)$ and $R(\ti \phi)$ such that $\ti f(\ti \phi)\subset L(\ti\phi)$ and we deduce that $\ti f(\Adh_\D(L( \ti \phi)))\subset L(\ti\phi))$.
\end{rem}

We will prove the following lemma which will be useful in section \ref{equality}. \\

\begin{lem}\label{intersection-positive}
Let us consider $x,x',y,y'\in X$  fixed points of $f$ such that there exist a connexion $\phi$ from $x$ to $y$ and a connexion $\phi'$ from $x'$ to $y'$.
If $\phi$ and $\phi'$ strongly intersect then there exists a connexion from $x$ to $y'$ and a connexion from $x'$ to $y$.
\end{lem}

\begin{rem}\label{rem-intersection}
The result stands for non area-preserving homeomorphisms isotopic to the identity and no foliations are involved in the statement. Moreover, the fixed points $x,y,x',y'$ do not have to be distincts to obtain the result. Nevertheless, if we suppose that $f$ is a Hamiltonian homeomorphism, then for every connexion $\phi$ between two fixed points $x$ and $y$, we have $A_f(x)>A_f(y)$ where $A_f$ is the action function of $f$, see \ref{action-decrease} in the preliminaries. So, if we consider four fixed points satisfying the hypothesis of Lemma \ref{intersection-positive}, then, by hypothesis, we have $x\neq y$ and $x'\neq y'$ and the result implies that $x \neq y'$ and $x'\neq y$. \\
\end{rem}

We will need a result of Kerékj\'{a}rt\'{o} \cite{KER} which asserts that each connected component of the intersection of two Jordan domains is a Jordan domain. A \textit{Jordan domain} is the relatively compact connected component of the complement of a simple closed curve of the plane, called a \textit{Jordan curve}. We refer to \cite{CAL7} for the proof of the following result.

\begin{theorem}\label{Jordan-plan}
Let $U$ and $U'$ be two Jordan domains of the plane. Every connected component of $U\cap U'$ is a Jordan domain.
\end{theorem}

We now prove Lemma \ref{intersection-positive}.

\begin{proof}[Proof of Lemma \ref{intersection-positive}]
We consider the lifts $\ti\phi$ and $\ti\phi'$ of $\phi$ and $\phi'$ given by the hypothesis. We denote 
$\alpha(\ti\phi)=\ti x$, $\omega(\ti\phi)=\ti y$, $\alpha(\ti\phi')=\ti x'$ and $\omega(\ti\phi')=\ti y'$.\\

By symmetry we can suppose that $\ti x'\in(\ti y, \ti x)$ as in Figure \ref{fig-intersection-positive}.\\

We consider the oriented loops $\Gamma_{\ti\phi}=\ti\phi\cup[\ti y,\ti x]$ and $\Gamma_{\ti \phi'}=\phi'\cup [\ti y', \ti x']$ in $\R^2$. The loops $\Gamma_{\ti\phi}$ and $\Gamma_{\ti \phi'}$ are the frontier of the domains $\Adh_{\R^2}(L(\ti\phi))$ and $\Adh_{\R^2}(L(\ti\phi'))$ and are Jordan curves.  \\

For every $\ti z=\e^{i\theta}\in (\ti y,\ti x)$ there exist $\epsilon>0$ and $\eta>0$ such that 
$$ |\theta'-\theta|<\eta \text{ and }
1-\epsilon<r<1
\Rightarrow r\e^{i\theta'}\in L(\ti\phi).$$
Symmetrically, for every $\ti z=\e^{i\theta}\in (\ti y',\ti x')$ there exist $\epsilon>0$ and $\eta>0$ such that 
$$ |\theta'-\theta|<\eta \text{ and }
1-\epsilon<r<1
\Rightarrow r\e^{i\theta'}\in L(\ti\phi').$$

We denote $\ti x'=e^{i\alpha}$ and $\ti y=e^{i\beta}$ with $\beta<\alpha <\beta+2\pi$. There exists a continuous map $\psi :(\beta,\alpha) \ra [0,1]$ such that for every $ \theta \in(\beta,\alpha)$, we have
$$\psi(\theta)<r<1 \Rightarrow r\e^{i\theta}\in L(\ti\phi)\cap L(\ti\phi').$$
We will consider the small croissant  $K\subset \D$,  as in figure \ref{fig-intersection-positive}, defined by 
$$K=\{z=r\e^{i\theta} | \ \theta \in (\beta,\alpha), \ \psi(\theta)<r<1\}.$$
 
By Theorem \ref{Jordan-plan}, the connected component $U$ of $L(\ti\phi'\cap L(\ti\phi')$ which contains $K$ is a Jordan domain in $\R^2$. Since $K\subset U$, the boundary of $\Adh_{\R^2}(U)$ is the union of an arc in $\Ss^1$ and an oriented curve $J$ in $\D$ from $\ti x'$ to $\ti y$.\\

Since $\ti\phi$ and $\ti\phi'$ are Brouwer lines of $\ti f$, we have that $\ti f(\Adh_\D(U))$ is connected and satisfies $\ti f(\Adh_\D(U))\subset L(\ti\phi)\cap L(\ti\phi')$. Moreover, we have $\ti f(K)\subset U$, so we deduce 
$$\ti f(\Adh_\D(U))\subset U.$$
In other words, the line $J$ is a Brouwer line for $\ti f$. Hence $J$ induces a connexion in $\Sigma$ from $x'$ to $ y$.\\
 
By considering the intersection $R(\ti \phi)\cap R(\ti\phi')$, with the same arguments we obtain a connexion in $\Sigma$ from $x$ to $y'$. 
\begin{figure}[H]
\begin{center}
\begin{tikzpicture}[scale=0.6]

  \draw (0,0) circle (4); 
  \draw (-2.5,-2.5) node{$\D$};

	\draw[green] (-4,0) ..controls +(1,2.5) and +(-2.5,-1).. (0,4);
	\draw[green] (-2.3,2.3) node{$K$};
	\draw[green] (-2,3.46) --(-1.4,3.42);
	\draw[green] (-2.82,2.82) --(-2.6,2.6);
	\draw[green] (-3.46,2) --(-3.42,1.4);
			
  \draw[red] (0,4) -- (0,1.5); 
  \draw[red] (0,1.5) ..controls +(0,-0.5) and +(0,0.5).. (-1,1);
  \draw[red] (-1,1) ..controls +(0,-0.5) and  +(0,-0.5).. (0,1);
  \draw[red] (0,1) .. controls +(0,0.5) and +(-0.5,0.5).. (1,1.5);
  \draw[red] (1,1.5) .. controls +(1,-1) and +(-1,1).. (-1,-1);
  \draw[red] (-1,-1) .. controls +(1,-1) and +(0,2) .. (1.5,-0.2);
  \draw[red,<-] (1.5,-0.2) ..controls +(0,-2) and +(0,4).. (0,-4); 
  \draw (0,-3) node[left,red]{$\ti{\phi}$};

\begin{scope}[rotate = 90]
    \draw[blue,decoration={markings, mark=at position 0.5 with {\arrow{>}}}, postaction={decorate}] (0,4) -- (0,1.5); 
    \draw[blue] (0,1.5) ..controls +(0,-0.5) and +(0,0.5).. (-1,1);
    \draw[blue] (-1,1) ..controls +(0,-0.5) and  +(0,-0.5).. (0,1);
    \draw[blue] (0,1) .. controls +(0,0.5) and +(-0.5,0.5).. (1,1.5);
    \draw[blue] (1,1.5) .. controls +(1,-1) and +(-1,1).. (-1,-1);
    \draw[blue] (-1,-1) .. controls +(1,-1) and +(0,2) .. (1.5,0);
    \draw[blue] (1.5,0) ..controls +(0,-2) and +(0,4).. (0,-4);
    \draw (0,3) node[below,blue]{$\ti{\phi'}$};
  \end{scope}

	\draw[blue,left] (-4,0) node{$\ti x'$};
	\draw[blue,right] (4,0) node{$\ti y'$};
	\draw[red,below] (0,-4) node{$\ti x$};
	\draw[red,above] (0,4) node{$\ti y$};
\end{tikzpicture}
\end{center}
\caption{ }
\label{fig-intersection-positive}
\end{figure}

\end{proof}

For the remainderof the section, we also suppose that $f$ is a Hamiltonian homeomorphism, we will define a barcode associated to $X$.\\

The notion of connexion induces an order on $X$ where, for $x,y\in X$,  we say that $x>y$ if there exists a connexion from $x$ to $y$.\\

We saw in section \ref{foliation} that the index function $\ind(\FF,\cdot)$ defined on $X$ does not depend on the choice of the foliation $\FF$ positively transverse to $I$ and we denote this index function $\ind(I,\cdot)$. We will study the following barcode.\\

\begin{definition}
We define the graph of connexion $G(>)$ whose set of vertices is equal to $X$ and in which there is an edge between two vertices $x$ and $y$ if and only if $x> y$ in $\Sigma$. We denote by $\bm{\beta}_>$ the barcode $\mathcal{B}(G(>),A_f|_X,\ind(I,\cdot)).$
\end{definition}

\begin{rem} 
If we consider a foliation $\FF\in\FF(I)$, the graph $G(\FF)$ is a subgraph of $G(>)$. 
\end{rem}

		\section{Equalities of the previous constructions and independence of the foliation}\label{equality}

In this section, we fix a Hamiltonian homeomorphism $f$ of a closed and oriented surface $\Sigma$ with a finite number of fixed points. Let $X\subset \Fix(f)$ be a maximal unlinked set of fixed points and $I=(f_t)_{t\in[0,1]}$ an isotopy from $\id$ to $f$ such that $\Sing(I)=X$. The action functional of $f$ will be denoted $A_f$.\\
For a foliation $\FF$ positively transverse to the isotopy $I$, we denote by $G(\FF)$ the graph, defined in section \ref{2barcode}, whose set of vertices is the set $X$ such that for every couple of vertices $x$ and $y$ of $G(\FF)$ there exists an edge from $x$ to $y$ if and only if there exists a leaf in $\FF$ from $x$ to $y$. We will also consider the subgraphs $(G^-_t(\FF))_{t\in \R}$ and $(G^+_t(\FF))_{t\in\R}$ given by the natural filtration of $G(\FF)$ by $A_f$. \\

We will use some notation of section \ref{construction}. For $t\in\R$ we denote $\mathcal{C}^-_t$ and $\mathcal{C}^+_t$ the sets of connected components of $G^-_t(\FF)$ and $G^+_t(\FF)$ and for a connected component $\mathcal{C}$ of $G^-_t(\FF)$ we define $L(\mathcal{C})=\min\{A_f(y)\ | \ y \text{ vertex of } \mathcal{C}\}$.\\

We consider the graph of connexion $G(>)$, defined in section \ref{order}, whose set of vertices is equal to $X$ such that there is an edge between two vertices $x$ and $y$ if and only if there exists a connexion between $x$ and $y$ in $\Sigma$.\\

We will consider the barcode $\mathcal{B}(G(\FF),A_f|_X,\ind(\FF,\cdot))$, denoted $\bm{\beta}_\FF$, constructed in section \ref{2barcode} and associated to the foliation $\FF$. Recall that the index function $\ind(\FF,\cdot)$ defined on $X$ does not depend on the choice of $\FF\in\FF(I)$ and we denote this index function $\ind(I,\cdot)$. We will consider the barcode  $\mathcal{B}(G(>),A_f|_X,\ind(I,\cdot))$, denoted $\bm{\beta}_>$, defined in section \ref{order}.\\

In the first section, we compare these two barcodes and in the second section we compare the barcode $\bm{\beta}_\FF$ with the barcode constructed in section \ref{génériqueF} in the generic case.

			\subsection{Equality between the barcode $\bm{\beta}_\FF$ and the barcode $\bm{\beta}_>$}\label{1egality}
			
In this section, we prove the following theorem.

\begin{theorem}\label{egalité-ordre}
The barcode $\bm{\beta}_\FF$ does not depend on the choice of $\FF\in\FF(I)$ and satisfies $\bm{\beta}_\FF=\bm{\beta}_>.$
\end{theorem}

For the proof, we fix a foliation $\FF\in\FF(I)$ positively transverse to the isotopy $I$. We will need the following lemma.

\begin{lem}[Fundamental]\label{fundamental}
For each $t\in \R$, the set of connected components of $G^-_t(\FF)$ defines the same partition of $X\cap A_f^{-1}((-\infty,t))$ than the set of connected components of $G^-_t(>)$.
\end{lem}

If we suppose Lemma \ref{fundamental} true, the proof of Theorem \ref{egalité-ordre} is straigthforward since the action functions and the index functions are equals and the constructions of the barcodes $\bm{\beta}_>$ and $\bm{\beta}_\FF$ in section \ref{construction} depend only on the connected components of the graphs $G(\FF)$ and $G(>)$ and the action and index values of the singularities of $X$.\\

\begin{proof}[Proof of Lemma \ref{fundamental}]
For $t\in\R$ and for each fixed points $x,y\in X$ of action less than $t$, we want to prove the equivalence of the following two properties.

\begin{enumerate}[(i)]
\item The elements $x,y$ are in the same connected component $\mathcal{C_>}$ of $G^-_t(>)$,
\item The elements $x,y$ are in the same connected component $\mathcal{C}_\FF$ of $G^-_t(\FF)$.
\end{enumerate}

\textit{ (ii) $\Rightarrow$ (i).} Since $x$ and $y$ are in the same connected component $\mathcal{C}$ of $G^-_t(>)$, there exists a family $(x_i)_{0\leq i\leq k}$ of singularities of $X$ such that 
\begin{itemize}
\item $x_0=x$ and $x_k=y$.
\item $A_f(x_i)<t$ for every $i\in\{0,...,k\}$
\item For every $i\in\{1,...,k\}$ there exists a leaf $\phi_i$ of $\FF$ either from $x_i$ to $x_{i-1}$ or from $x_{i-1}$ to $x_i$
\end{itemize}

Moreover, each leaf of the foliation $\FF$ is by definition a connexion, so $(\phi_i)_{0\leq i\leq k}$ is a family of connexions and then the singularities $x_i$, $i\in\{0,...,k\}$, are in the same connected component of $G^-_t(>)$ and we obtain the result.\\

\textit{ (i) $\Rightarrow$ (ii).}
We will prove the following lemma.

\begin{lem}\label{sousfonda}
Let us consider $\FF\in\FF(I)$ and $(x,y)\in X^2$. If $x> y$ then $x$ and $y$ are in the same connected component of $G^-_{t}(\FF)$ for every $t>A_f(x)$.
\end{lem}

Let us assume that Lemma \ref{sousfonda} is true and consider $x,y\in X$ and $t>A_f(x)$ such that $(i)$ is satisfied, we prove that $(ii)$ holds.\\

Since $x$ and $y$ are in the same connected component $\mathcal{C}$ of $G^-_t(>)$, there exists a family $(x_i)_{0\leq i\leq k}$ of singularities of $X$ such that 
\begin{itemize}
\item $x_0=x$ and $x_k=y$.
\item $A_f(x_i)<t$ for every $i\in\{0,...,k\}$
\item For every $i\in\{1,...,k\}$ there exists a connexion $\phi_i$ either from $x_i$ to $x_{i-1}$ or from $x_{i-1}$ to $x_i$
\end{itemize}
Moreover, by Lemma \ref{sousfonda} $x_{i-1}$ and $x_i$ are in the same connected component of $G_t^-(\FF)$ for every $i\in\{1,...,k\}$ so $x$ and $y$ are also in the same connected component of $G_t^-(\FF)$.
\end{proof} 
 
 To complete the proof of Lemma \ref{fundamental}, it remains to prove Lemma \ref{sousfonda}. We introduce the notion of local rotation set and we refer to \cite{ROUX1,ROUX5} for more details.\\
 
Let $\Sigma$ be a connected oriented surface. We write $f:(W,z_0)\ra (W',z_0)$ for an orientation preserving homeomorphism between two neighborhoods $W$ and $W'$ of $z_0\in\Sigma$ such that $f(z_0)=z_0$. Such a local homeomorphism $f$ is called an \textit{orientation preserving local homeomorphism at $z_0$}. We recall the definition of local isotopies of Le Calvez \cite{CAL5}: a \textit{local isotopy} $I=(f_t)_{\tin}$ from $\id$ to $f$ is a continuous family of local homeomorphisms $(f_t)_{\tin}$ fixing $z_0$ such that 
\begin{enumerate}[-]
\item each $f_t$ is a homeomorphism of a neigborhood $V_t\subset W$ of $z$ into a neighborhood $V_t'\subset W'$ of $z$ ;
\item the sets $\{(z,t)\in \Sigma\times[0,1] | z\in V_t\}$ and $\{(z,t)\in \Sigma\times[0,1] \ | \ z\in V_t'\}$ are open in $\Sigma\times[0,1]$ ;
\item the map $(z,t)\mapsto f_t(z)$ is continuous on $\{(z,t)\in \Sigma\times[0,1] \ | \ z\in V_t\}$ ;
\item the map $(z,t)\mapsto f_t^{-1}(z)$ is continuous on $\{(z,t)\in \Sigma\times[0,1] \ | \  z\in V_t'\}$ ;
\item we have $f_0=\id_{V_0}$ and $f_1=f|_{V_1}$ ;
\item for all $t\in[0,1]$, we have $f_t(z_0)=z_0$.
\end{enumerate}

Let us consider a local orientation preserving homeomorphism $f:(W,z_0)\ra (W',z_0)$ and $I=(f_t)_{\tin}$ a local isotopy from $\id$ to $f$. We want to define the local rotation set of the isotopy $I$ at $z_0$. Given two neighborhoods $V\subset U$ of $z_0$ included in $W$ and an integer $n\geq1$ we define 
$$E(U,V,n)=\{z\in U \ | \  z\notin V, f^n(z)\notin V, f^i(z)\in U \text{ for all } 1\leq i \leq n\}.$$
We define the rotation set of $I$ relative to $U$ and $V$ by 
$$\rho_{U,V}(I)=\bigcap_{m\geq 1}\overline{\bigcup_{n\geq m} \{\rho_n(z)\ | \ z\in E(U,V,n)\}} \subset [-\infty,\infty],$$
where $\rho_n(z)$ is the average change of angular coordinate along the trajectory of $z$ during $n$ iterates. We define the local rotation set of $I$ to be 
$$\rho_s(I,z_0)=\bigcap_U\overline{\bigcup_V \rho_{U,V}(I)} \subset [-\infty,\infty],$$
where $V\subset U \subset W$ are neighborhoods of $z_0$.\\

The local rotation set is an invariant of local conjugacy in the following sense: let us say that an isotopy $I'=(f_t')_{\tin}$ is locally conjugated to $I$ if there exists a homeomorphism $\phi: W\ra W''$ between two neighborhood of $z_0$ which preserves the orientation and fixes $z_0$ such that for each $\tin$ we have $f_t'=\phi \circ f_t\circ \phi^{-1}$. For each neighborhoods $V$ and $U$ of $z_0$ such that $V\subset  U\subset W$ we have
$$\rho_{U,V}(I)=\rho_{\phi(U),\phi(V)}(\phi I\phi^{-1}).$$
In particular we deduce that
$$\rho_s(I)=\rho_s(\phi I\phi^{-1}).$$

Let us consider a homeomorphism of the plane $f$ isotopic to the identity which preserves the orientation and fixes the origin and an isotopy $I=(f_t)_{\tin}$ from $\id$ to $f$ which fixes the origin. Recall that $R=(R_t)_{\tin}$ is the isotopy of the rotation of angle $2\pi$ such that $R_t(z)=z\e^{2i\pi t}$ for each $z\in\R^2$ and $\tin$.  For each $p\in \Z$ and $q\in\Z$ we have
$$\rho_s(R^p I^q)=q\rho_s(I)+p.$$

\begin{ex}
Let us consider a fiber rotation $h_\alpha: (r,\theta) \ra (r,\theta+\alpha(r))$ on the plane where $\alpha:(0,\infty)\ra \R$ is continuous and an isotopy $I=(h_t)_{\tin}$ such that $h_t(r,\theta)=(r,\theta+t\alpha(r))$ for $\tin$. The local rotation set  $\rho_s(I)$ of $I$ at the origin is equal to the set of accumulation points of $\alpha$ at $0$.
\end{ex}

In the case where $f$ is area-preserving (or a hamiltonian homeomorphism), Gambaudo and Pécout \cite{GAM1} proved that the local rotation set is not empty. Moreover, if we suppose that $\Fix(f)$ is finite, $0$ is not accumulated by fixed points then if the local rotation set is not empty it does not contain an integer in its interior. Notice that the result holds if we suppose that $f$ satisfies the local intersection property, meaning that for each non contractible loop $\gamma$ of $W\backslash\{z_0\}$ we have $f(\gamma)\cap \gamma\neq \emptyset$. 

The \textit{rotation number} classify the homotopy classes of the isotopies at $z_0$. Let us consider a local orientation preserving homeomorphism $f:(W,z_0)\ra (W',z_0)$ of a surface $\Sigma$ such that $f(z_0)=z_0$ and $I=(f_t)_{\tin}$ a local isotopy from $\id$ to $f$ which fix $z_0$. Let us consider a closed disk $D\subset \Sigma$ containing $z_0$ in its interior. For every point $z\in D\backslash \{z_0\}$ close enough to $z_0$, the trajectory of $z_0$ along $I$ is a loop included in $D\backslash\{z_0\}$. There exists an integer $k\in\Z$ such that this trajectory is freely homotopic in $D\backslash\{z_0\}$ to $(\partial D)^k$. The integer $k$ depends only on the choice of the isotopy $I$, it is the \textit{rotation number} $k=\rho(I,z_0)$ of $I$ at $z_0$. We consider the isotopy $R_\infty=(r_t)_{\tin}$ of the plane where $r_t$ is the rotation of angle $2\pi t$ i.e $r_t(r,\theta)=(r,\theta +2\pi t)$ in radial coordinates. The isotopy $R_\infty$ extends into an isotopy on the sphere $\R^2\sqcup\{\infty\}$ and we have $\rho(R_\infty,\infty)=1$ while $\rho(R_\infty,0)=-1$. We refer to \cite{CAL5} for more details.\\

\begin{proof}[Proof of Lemma \ref{sousfonda}]

We consider two fixed points $x,y\in X$ such that $x> y$. Every attractive or repulsive basin in this proof will be defined relatively to the foliation $\FF$. We will divide the proof in three cases, in the first one $x$ will be a sink of $\FF$, in the second one $x$ will be a saddle and in the last one $x$ will be a source of $\FF$.\\

\textit{First case.} 
We suppose that $x$ is a sink. We will prove that there is no connexion from $x$ to another singularity of $\FF$.\\

We say that the orbit of a $q$-periodic point $z$ of $f$ is \textit{contractible} if the concatenation of the trajectories of the points $f^{k}(z)$, $k\in\{0,...,q-1\}$, along the isotopy $I$ is a contractible loop , denoted $\gamma_z$, in $\Sigma$. The loop $\gamma_z$ is called the trajectory of the periodic orbit of $z$.  We say that a contractible $q$-periodic orbit has \textit{type} $(p,q)$ associated to $I$ at $x\in\Fix(f)$ if its trajectory along $I$ is homotopic to $p\Gamma$ in $\Sigma\backslash\Sing(I)$, where $\Gamma$ is the boundary of a sufficiently small Jordan domain containing $x$.\\

We will use the following version of a result of Yan and we refer to \cite{YAN1}, Theorem $1.1$, for a proof.

\begin{theorem}\label{thmYan}
Let us consider a fixed point $z$ of $f$ of Lefschetz index equal to $1$ fixed by the isotopy $I$, and such that the rotation set $\rho_s(I,z)$ is reduced to $\{0\}$. The point $z$ is accumulated by periodic points. More precisely, the following property holds: there exists $\epsilon>0$, such that, for every neighborhood of $z$, either for every irreducible $p/q\in(0,\epsilon)$, or for every irreducible $p/q\in(-\epsilon,0)$, there exists a contractible periodic orbit $O_{p/q}$ of type $(p,q)$.
\end{theorem}

Let us prove that there is no connexion from $x$ to another singularity of $\FF$. We suppose that there exists a connexion $\phi$ from $x$ to another singularity $y$ of $\FF$, we want to find a contradiction. \\

The singularity $x$ is a sink of the foliation $\FF$ which is positively transverse to the isotopy $I$, so the local rotation set $\rho_{s,I}(x)$ of $x$ is included in $(-\infty,0]$. 
Moreover, it is not difficult to prove that the existence of the connexion $\phi$ implies that the rotation set $\rho_{s,I}(x)$ of $x$ is included in $[0,+\infty)$. Indeed, locally, a connexion whose alpha-limit is $x$ is a \textit{positive arc}, which means that in polar coordinates where $\gamma$ corresponds to the semi-line $\{\theta=0\}$, for every point $z$ close enough to $x$, the variation of $\theta$ along the trajectory is positive.  We refer to Theorem $3.2.4$ and section $2.4$ of \cite{ROUX1} for more details.\\
So the local rotation set of $x$ for the isotopy $I$ is reduced to the integer $\{0\}$. \\

If the Lefschetz index of $x$ is not equal to $1$, by a result of Le Roux, see \cite{ROUX1} Theorem $4.1.1$, the foliation $\FF$ and the homeomorphism $f$ have the same index at $x$ for the  isotopy $I$. But, by hypothesis, $x$ is a sink of the foliation $\FF$ so $\ind(\FF,x)=1$ and we obtain a contradiction.\\

If the Lefschetz index of $x$ is equal to $1$ then we can apply Theorem \ref{thmYan}. More precisely, the singularity $x$ is a sink of $\FF$ so there exists $\epsilon>0$ such that $x$ is accumumated by periodic orbits $O_{p/q}$ of type $(p,q)$ where $p/q\in (0,\epsilon)$. So the rotation number $\rho_{s,I}(x)$ is not reduced to $\{0\}$ and we obtain a contradiction.\\

\textit{Second case.} We suppose that $x$ is a saddle point. We prove the result by contradiction. \\

We suppose that there exists a connexion $\gamma:[0,1] \ra \Sigma$ from $x$ to some $y$ and that $x$ and $y$ are not in the same connected component of $G^-_{A_f(x)^+}(\FF)$.
We denote by $W^s(\mathcal{C}^-_x)\subset \Sigma$ the attractive basin of $\mathcal{C}^-_x$, where $\mathcal{C}^-_x$ is the connected component of $G^-_{A_f(x)^+}(\FF)$ which contains $x$.
By Lemma \ref{interior}, the fixed point $x$ is in the interior of $W^s(\mathcal{C}^-_x)$ and every singularity $z$ of $\FF$ in the frontier of $W^s(\mathcal{C}^-_x)$ satisfies $A_f(z)>A_f(x)$. 
The existence of the connexion $\gamma$ implies that $A_f(x)>A_f(y)$ so the fixed point $y$ is not in the frontier of $W^s(\mathcal{C}^-_x)$ and is in $\Sigma\backslash \overline{W^s(\mathcal{C}^-_x)}$.\\
We consider the universal cover $\ti{\Sigma\backslash X}$ of $\Sigma\backslash X$ which is identified to the unit disk $\D$ and $\pi:\D\ra \Sigma\backslash X$ the universal covering map. Let $U$ be a connected component of $\pi^{-1}(\mathcal{C}^-_x\backslash X)$ and $\ti\gamma$ a lift of the connexion $\gamma$ such that there exists $\epsilon>0$ such that $\ti{\gamma}((0,\epsilon))\subset U$. By hypothesis on $y$, there exists $\epsilon'>0$ such that $\ti{\gamma}((1-\epsilon',1))\notin U$. \\
Recall that $\pi$ naturally extends to $\Ss^1$. Moreover, we saw that $\lim_{t \ra 0} \ti\gamma(t)$ and $\lim_{t \ra 1} \ti\gamma(t)$ are well-defined on $\Ss^1$ and will be denoted $\ti x$ and $\ti y$.\\

The set $U$ is an open connected set of $\D$ whose frontier is a union of lifted leaves. Then there exists a lifted leaf $\ti \psi$ of $\Fr(U)$ which separates $\D$ such that $U$ is on one side and $\ti y$ is in the other. By hypothesis, the points $\omega(\ti\psi)$ and $\alpha(\ti\psi)$ are distinct from $\ti x$ and $\ti y$, indeed, $x$ is in the interior of $W^s(\mathcal{C}_x^-)$ and $y\notin \Fr(W^s(\mathcal{C}_x^-))$. Since $\ti \psi$ is the lift of a connexion, we obtain that $\psi$ and $\gamma$ are two connections which intersect strongly, hence by Lemma \ref{intersection-positive} there exists a connexion from $x$ to $\omega(\psi)$ which is impossible because by definition $\omega(\psi)\in \Fr(W^s(\mathcal{C}_x^-))$ and then $A_f(\omega(\psi))>A_f(x)$.\\
	
\textit{Third case.} Suppose that $x$ is a source of $\FF$. The point $y$ is either in the frontier of the repulsive basin $W^u(x)$ of $x$ or in the complement of $\overline{W^u(x)}$. We separate these two cases.\\

1) We suppose that $y$ is in the frontier of $W^u(x)$. There exists a chain of connexions from $x$ to $y$ and so we deduce the result. Indeed, by definition, the singularity $y$ is accumulated by leaves $(\phi_j)_{j\in J}$ of $W^u(x)$ whose alpha-limit point is $x$. By the local model described in section \ref{geometricsprop}, the closure of these leaves contains a chain of connexions which starts at $x$ and also contains $y$. \\

2) We suppose that $y$ is in the complement of $\overline{W^u(x)}$. We consider $U$ a connected component of $\pi^{-1}(W^u(x)\backslash \{x\})$ where $\pi:\D\ra \Sigma\backslash X$ the covering map defined in the second case. We can consider a lift $\ti \gamma$ of the connexion $\gamma$ such that there exists $\epsilon>0$ such that $\ti\gamma((0,\epsilon))\subset U$. By hypothesis on $y$, there exists $\epsilon'>0$ such that $\ti{\gamma}((1-\epsilon',1))\notin U$. The limits $\lim_{t \ra 0} \ti\gamma(t)$ and $\lim_{t \ra 1} \ti\gamma(t)$ are well-defined on $\Ss^1$ and will be denoted $\ti x$ and $\ti y$.\\

We apply similar arguments as in the second case. The set $U$ is an open connected set of $\D$ whose boundary is a union of lifted leaves. Then there exists a lifted leaf $\ti \psi$ of $\Fr(U)$ separating $\D$ such that $U$ is on one side and $\ti y$ is on the other. By hypothesis, the points $\omega(\ti\psi)$ and $\alpha(\ti\psi)$ are distinct from $\ti x$ and $\ti y$, indeed, $x$ is in the interior of $W^s(\mathcal{C}_x^-)$ and $y\notin \Fr(W^s(\mathcal{C}_x^-))$. Since $\ti \psi$ is the lift of a connexion, we obtain that $\psi$ and $\gamma$ are two connections which intersect strongly, hence by Lemma \ref{intersection-positive} there exists a connexion from $\alpha(\psi)$ to $y$. The singularity $\alpha(\psi)$ can not be a source nor a sink so it is a saddle point of $\FF$, hence we apply the result of the second case which asserts that the existence of a connexion from $\alpha(\psi)$ to $y$ implies that $\alpha(\psi)$ and $y$ are in the same connected component of $G_{A_f(\alpha(\psi))^+}^-(\FF)$. Moreover, by hypothesis, $\alpha(\psi)$ and $x$ are in the same connected component of $G_{A_f(x)^+}^-(\FF)$ thus we deduce that $x$ and $y$ are also in the same connected component of $G_{A_f(x)^+}^-(\FF)$.

\end{proof}

			\subsection{Equality between the barcode $\bm{\beta}_\FF$ and the barcode $B_{\gen}(f,\FF)$ in the simplest case}\label{2egality}

We consider a Hamiltonian homeomorphism $f$ of a closed and oriented surface $\Sigma$ with a finite number of fixed points.
We suppose that $\Fix(f)$ is finite and unlinked, each fixed point $x\in\Fix(f)$ satisfies $\ind(f,x)\in\{-1,1\}$ and that the action function $A_f:\Fix(f)\ra\R$ is injective. Let $I=(f_t)_{t\in[0,1]}$ be a maximal isotopy from $\id$ to $f$ fixing all fixed points of $f$. We denote $A_f$ the action functional of $f$.\\

Recall that a foliation $\FF\in\FF_{\gen}(I)$ does not have connexions between saddle points, and the stable and unstable cones of a saddle point $x$  of $\FF$ are both composed of a unique leaf which will be referred to as the stable and unstable leaves of $x$.\\

Let us consider the graph $G_{\gen}(\FF)$ given by Definition \ref{graphgeneric}. 
Remember that $G_{\gen}(\FF)$ is the graph whose set of vertices is the set $\Fix(f)$ and whose edges correspond to leaves $\phi$ of $\FF$ such that $\ind_{CZ}(f,\alpha(\phi))=\ind_{CZ}(f,\omega(\phi))-1$, where $\ind_{CZ}(f,\cdot)$ is the Conley-Zehnder index, defined in section \ref{génériqueF}, equals to $1$ on sources and sinks and equals to $-1$ on saddle points. Notice that $G_{\gen}$ is distinct from the graph $G(\FF)$ given in the introduction of this section.\\

In this section we want to compare the barcode $B_{\gen}(\FF)$ to the barcode 
$$\mathcal{B}(G(\FF),A_f|_X,\ind(\FF,\cdot)),$$ 
denoted $\bm{\beta}_\FF$,  constructed in section \ref{2barcode}. We will prove the following result.
		
\begin{theorem}\label{egalité-generique}
Let us consider a Hamiltonian homeomorphism $f$ on a compact surface $\Sigma$. We suppose that $\Fix(f)$ is finite and unlinked, each fixed point $x\in\Fix(f)$ satisfies $\ind(f,x)\in\{-1,1\}$ and the action function is injective. We consider a maximal isotopy $I$ such that $\Sing(I)=\Fix(f)$ then for a foliation $\FF\in\FF_{\gen}(I)$ we have
$$B_{\gen}(\FF)=\bm{\beta}_\FF.$$
\end{theorem}

We recall the definition of the functor $\beta$. Let $\textbf{V}=(V_s)_{s\in\R}$ be a persistence module. Let us consider the set of $t\in\R$ in the spectrum of $V$ such that $\mathrm{dim}(\mathrm{Ker}(i_{t^-,t^+}))$ is equal to $1$ and label its elements $b_1,...,b_n$. For each $b_j$, there exists a unique $a_j\in\R$ with the following property: Let $x\in V_{b_j^-}$ represents a non-zero element in $\mathrm{Ker}(i_{t^-,t^+})=1$, the element $x$ is in the image of $i_{a^+_j,b^-_j}$ but $x$ is not in the image of $i_{a^-_j,b^-_j}$; if we label the remaining elements of the spectrum of $\textbf{V}$ by $\{c_1,...,c_m\}$ then the barcode $\beta(\textbf{V})$ consists of the list of intervals $(a_j,b_j]$ and $(c_k,+\infty)$, where $1\leq j\leq n$ and $1\leq k\leq m$.\\

We will consider $G_{\gen,t}^-(\FF),G_{\gen,t}^+(\FF)$ the associated filtered graphs and we denote by $(H_*^t)_{t\in\R}$ the persistence module of the chain complex $(C_i^t)_{i\in\{0,1,2\},t\in\R}$ of Definition \ref{complexchain}. Finally, we denote by $B_{\gen}(\FF)$ the barcode $\beta((H_*^t)_{t\in\R})$ where $\beta$ is the functor defined in section \ref{introbarcodes}. \\

To avoid any confusion, we will always refer to the chain complex by $C_i^t$ where $i\in\{0,1,2\}$ and $t\in\R$. We will refer to connected components of the graph $G(\FF)$ by $\mathcal{C}$ or $\mathcal{C}'$ and to a connected component of the graph $G_{\gen}(\FF)$ by $\mathcal{C}_{\gen}$. For $t\in \R$ and a connected component $\mathcal{C}$ of $G_t^-(\FF)$ we denote $L(C)$ the minimum of the action function on the sinks of $\mathcal{C}$ and for $t\in \R$ and a connected component $\mathcal{C}'$ of $G_t^+(\FF)$ we denote $D(C)$ the maximum of the action function on the sources of $\mathcal{C}'$.  Moreover, to simplify the notation, we provide the filtered chain complex $(C^i_t)_{i\in\{0,1,2\},t\in\R}$ with a natural scalar product $\langle .|.\rangle$ associated to the canonical basis. Meaning that we consider the bilinear function $\langle .|.\rangle$ on the space $C^i_t$ such that for every couple of fixed points $x$ and $y$ of $f$ in $C^i_t$, we have $\langle x,y\rangle=1$ if and only if $x=y$ and $\langle x,y\rangle=0$ otherwise.\\

Theorem \ref{egalité-generique} allows us to prove Property \ref{independance-generique} of section \ref{génériqueF} which states that the barcode $B_{\gen}(f,\FF)$ does not depend on the choice of $\FF\in\FF_{\gen}(I)$. We recall that assumptions of Theorem  \ref{egalité-generique} are satisfies in this particular case.\\

\begin{proof}[Proof of Proposition \ref{independance-generique}]
By Theorem \ref{egalité-generique}, for each foliation $\FF\in \FF_{\gen}(I)$ we have that $B_{\gen}(\FF)=\bm{\beta}_\FF$. Moreover, by Theorem \ref{egalité-ordre}, the barcode $\bm{\beta}_\FF$ does not depend on the choice of the foliation $\FF\in \FF_{\gen}(I)$. So we obtain that $B_{\gen}(f,\FF)$ does not depend on the choice of $\FF\in\FF_{\gen}(I)$ which is the result.
\end{proof}

We fix a foliation $\FF\in\FF_{\gen}$ for the remaining of the section.\\

\begin{proof}[Proof of Theorem \ref{egalité-generique}]
By Remark \ref{bargeneric}, each action value of $A_f$ is the end of a unique bar of the barcode $B_{\gen}(\FF)$ and by Corollary \ref{numberbar}, we have the same result for the barcode $\bm{\beta}_\FF$ so it is enough to prove the inclusion $\bm{\beta}_\FF\subset B_{\gen}(\FF)$ to prove that these barcodes are equal.\\

Moreover, Corollary \ref{numberbar} states that exactly one end point of every bar of the barcode $\bm{\beta}_\FF$, except the bars $(\min A_f,+\infty)$ and $(\max A_f,+\infty)$, is the action value of a saddle point of $\FF$. So it is enough to prove that finite bars of the barcode $\bm{\beta}_\FF$ are also bars of the barcode $B_{\gen}$ to prove the inclusion $\bm{\beta}_\FF\subset B_{\gen}(\FF)$. Indeed, the remaining bars of the barcode would be the same semi-infinite bars as they would be associated to the same saddle points.\\

We will prove that for every saddle point $x$ of $f$, if the bar $J$ of $\bm{\beta}_\FF$, of which $A_f(x)$ is an end, is a finite bar, then it is also a bar of the barcode $B_{\gen}(\FF)$. Notice that,  by construction, the bars $(\min A_f,+\infty)$ and $(\max A_f,+\infty)$ of $\bm{\beta}_\FF$ are also bars of the barcode $B_{\gen}(\FF)$.\\

For the remainder of the proof, we consider a saddle point $x$ of $f$, we denote by $t$ its action value and by $\mathcal{C}_x$ the connected component of $G_{t^+}^-(\FF)$ which contains $x$. By Lemma \ref{connexionsaddle} the set of connected components of $G_t^-(\FF)$ which are included in $\mathcal{C}_x$, which were labeled $j^{-1}_t(\mathcal{C}_x)$ in section \ref{construction}, has one or two elements. We separate those two cases.\\

\textit{Case 1.} The set $j^{-1}_t(\mathcal{C}_x)$ consists of two connected components of $G_{t}^-(\FF)$ denoted $\mathcal{C}$ and $\mathcal{C}'$. By symmetry, we can suppose that $L(\mathcal{C})>L(\mathcal{C}')$ and, by the construction described in section \ref{construction}, there is a bar $(L(\mathcal{C}),t]$ in the barcode $\bm{\beta}_\FF$. Let us prove that this bar is also a bar of the barcode $B_{\gen}(\FF)$. It means that there is an element of $\mathrm{Ker}(i_{t^-,t^+})$ which is in the image of $i_{L(\mathcal{C})^+,t^-}$ but not in the image of $i_{L(\mathcal{C})^-,t^-}$.\\

By hypothesis, the omega-limit points of the unstable leaves of $x$ are distinct sinks $y$ and $y'$ of $\FF$ where $y\in \mathcal{C}$ and $y'\in \mathcal{C}'$. We have $\partial_1^{t^+}(x)=y+y'$ so the element $y+y'\in C_0^{t^+}$ satisfies $[y+y']\in \mathrm{Ker}(i_{t^-,t^+})$. It remains to prove that $[y+y']$ is in the image of $i_{L(\mathcal{C})^+,t^-}$ and not in the image of $i_{L(\mathcal{C})^-,t^-}$. For that, we will consider another cycle in $C_0^{t^+}$ representing $[y+y']$ in homology. \\

We will use some geometric lemmas.\\

We will call a \textit{path of leaves} a path $\Gamma$ in $\Sigma$ which is the concatenation of leaves of $\FF$. The singularities of a path of leaves will refer to the alpha-limit points and omega-limit points of those leaves. 

\begin{lem}\label{scie}
Let us consider $s\in \R$, and two sinks $y_1$ and $y_2$ of $\FF$ in the same connected component $\mathcal{C}_{s}$ of $G^-_{s}(\FF)$. There exists a path of leaves from $y_1$ to $y_2$ whose singularities are alternatively sinks and saddle points of $\mathcal{C}_{s}$.
\end{lem} 

\begin{proof}[Proof of Lemma \ref{scie}]
By definition of the connected component $\mathcal{C}_{s}$ of $G^-_{s}$, there exists a path of leaves $\Gamma$ from $y_1$ to $y_2$ in $\Sigma$. The path $\Gamma$ may contain sources. For a source $z$ in $\Gamma$ we will modify $\Gamma$ into a path which does not contain $z$. \\

If there is a source $z$ in $\Gamma$, there exist two leaves $\phi_1\subset \Gamma$ and $\phi_2\subset \Gamma$ whose alpha-limit points are equal to $z$ and omega-limit points are either saddle points or sinks of $\FF$ that we denote $x_1$ and $x_2$. The singularities $x_1$ and $x_2$ are in the repulsive basin of $z$ for $\FF$ so, by Lemma \ref{bordbassin}, there exists a path $\gamma$ of leaves of $G^-_s(\FF)$ from $x_1$ to $x_2$ whose singularities are alternatively saddle points and sinks of $\FF$.  \\
We cut the union $\phi_1\cup\{z\}\cup\phi_2$ from the path $\Gamma$ and replace this portion by the path $\gamma$ given by Lemma \ref{bordbassin}. We obtain a new path $\Gamma'$ from $y$ to $y'$ such that the source $z$ is not in $\Gamma'$.\\

We do the same process for every source of $\Gamma$ and we finally obtain a path from $y$ to $y'$ as wanted.
\end{proof}

We prove the following lemma.\\

\begin{lem}\label{homologous}
For every $s\in \R$ and every couple of sinks $y_1$ and $y_2$ of $\FF$ in the same connected component $\mathcal{C}_{s}$ of $G^-_{s}(\FF)$ we have $[y_1]=[y_2]$ in $H_0^s$.
\end{lem}

\begin{proof}[Proof of Lemma \ref{homologous}]
Let us consider $s\in \R$, and two sinks $y_1$ and $y_2$ of $\FF$ in the same connected component $\mathcal{C}_{s}$ of $G^-_{s}(\FF)$. By Lemma \ref{scie} there exists a path of leaves in $\Sigma$ from $y_1$ to $y_2$ whose singularities are alternatively sinks and saddle points of $\mathcal{C}_s$. We denote by $(x_i)_{0\leq i \leq n]}$ the saddle points of the path $\Gamma$ and by a simple computation we obtain
$$\partial_1^{s}\left(\sum_{i=0}^n x_i\right)=y_1+y_2.$$
So, by definition, $[y_1]=[y_2]$ in $H_0^s$.
\end{proof}

Let us come back to the first case of the proof of Theorem \ref{egalité-generique}. By Lemma \ref{homologous} each sink $z\in \mathcal{C}$ of $\FF$  (resp. each sink $z'\in \mathcal{C}'$ of $\FF$) satisfies $[z]=[y]$ (resp. $[z']=[y']$) in $H_0^{t}$. So for every couple of sinks $z\in\mathcal{C}$ and $z\in\mathcal{C}'$ of $\FF$, the element $z+z'\in C_0^{t^+}$ satisfies $[z+z']=[y+y']\in \mathrm{Ker}(i_{t^-,t^+})$. We denote by $z_\mathcal{C}$ and $z_{\mathcal{C}'}$ the sinks of $\mathcal{C}$ and $\mathcal{C}'$ such that $A_f(z_\mathcal{C})= L(\mathcal{C}), \  A_f(z_{\mathcal{C}'})= L(\mathcal{C}')$.\\
We supposed that $A_f(z_\mathcal{C})=L(\mathcal{C})>A_f(z_{\mathcal{C}'})=L(\mathcal{C}')$ so the sink $z_\mathcal{C}$ is not a cycle in $C_0^{L(\mathcal{C})^-}$ so the element $[z_\mathcal{C}+z_{\mathcal{C}'}]$ is not in the image of $i_{L(\mathcal{C})^-,t^-}$.
Moreover, the sinks $z_\mathcal{C}$ and $z_{\mathcal{C}'}$ are not in the same connected component of $G_{t^-}(\FF)$ and so we deduce that $[z_\mathcal{C}+z_{\mathcal{C}'}]$ is in the image of $i_{L(\mathcal{C})^+,t^-}$.\\

So, by construction, there exists a bar $(L(C),t]$ in the barcode $B_{\gen}(\FF)$.\\

\textit{Case 2.} The set $j^{-1}_t(\mathcal{C}_x)$ is a unique element. We will consider the connected components of the subgraphs $(G^+_t)_{t\in \R}$ instead of connected components of the subgraphs $(G^-_t)_{t\in \R}$.\\
We consider the connected component $\mathcal{C}_x'$ of $G^+_{t^-}(\FF)$ which contains $x$. By Lemma \ref{connexionsaddle} the set of connected components of $G_t^+(\FF)$ included in $\mathcal{C}_x'$, which were labeled $j'^{-1}_t(\mathcal{C}_x^-)$ in section \ref{construction}, is composed of $1$ or $2$ elements. We separate those two cases.\\

1) Suppose that $j'^{-1}_t(\mathcal{C}_x')$ is composed of one connected component, then, by construction, there is no finite bar $J$ in the barcode $\bm{\beta}_\FF$ of which $t$ is an end point. We have nothing to prove in this case.\\

2) Now we suppose that $j'^{-1}_t(\mathcal{C}_x')$ is composed of two connected components of the graph $G_{t^+}^+(\FF)$ denoted $\mathcal{C}$ and $\mathcal{C}'$. By symmetry we can suppose that $D(\mathcal{C})<D(\mathcal{C}')$ and by construction there is a bar $(t,D(\mathcal{C})]$ in the barcode $\bm{\beta}_\FF$. Let us prove that this bar is also a bar of the barcode $B_{\gen}(\FF)$. It means that there is an element in $\mathrm{Ker}(i_{D(\mathcal{C})^-,D(\mathcal{C})^+})$ which is in the image of $i_{t^+,D(\mathcal{C})^-}$ but not in the image of $i_{t^-,D(\mathcal{C})^-}$. We will need the following lemma about the repulsive basin $W^s(\mathcal{C})$ of $\mathcal{C}$.\\

\begin{lem}\label{sommesource}
We label $x_1,...,x_n$ the saddle points in the frontier of $\overline{W^s(\mathcal{C})}$. Then, for every $T>D(\mathcal{C})$, the element $Y=\sum_{\substack{ y\in \mathcal{C} \\ y \text{ source }}} y $ of $C^T_2$ satisfies
$$\partial_2^T(Y)=\sum_{i=1}^n x_i.$$
\end{lem}

\begin{proof}
For each source $y$ of $\mathcal{C}$, by definition, $\partial_2^T(y)$ is equal to the sum of the saddle points in the frontier of the repulsive basin of $y$. These saddle points have either one or both of their stable leaves in $W^s(\mathcal{C})$.  We separate those cases.\\

Firstly, we label $x_1,...,x_n$ the saddle points of $\FF$ of which only one stable leaf have its alpha-limit point in $\mathcal{C}$. For every $i\in[0,n]$ we have $\langle\partial_2^T(Y)|x_i\rangle =1$ for every $i\in[0,n]$. Moreover, the action values of these saddle points is less then or equal to $t$ and it is simple to see that this belong to the frontier of $\overline{W^s(\mathcal{C})}$.\\

Secondly, we label $x'_1,...,x'_m$  the saddle points such that both stable leaves have their alpha-limit points in $\mathcal{C}$. For every $i\in[0,m]$ we have $\langle\partial_2^T(Y)|x'_i\rangle =2$. Those saddle points are nondegenerate saddle points of $\FF$ so that they are in the interior $\overline{W^s(\mathcal{C})}$ and not in its frontier.  Indeed, both stable cones of a saddle point $x$ whose action satisfies $A_f(x)>t$ are leaves of $\FF$ whose alpha-limit points are in the same connected component of $G_t^+(\FF)$. \\\\

Finally, we compute $\partial_2^T(Y)$ as follows.

\begin{align*}
\partial_2^T(Y) &= \sum_{\substack{y\in \mathcal{C}, \\ y \text{ source }}} \partial_2^T(y)\\
&=\sum_{i=1}^n x_i +\sum_{i=1}^m 2x'_i\\
&=\sum_{i=1}^n x_i .
\end{align*}

And we obtain the result.
\end{proof}

Let us denote $c=D(\mathcal{C})$, we will consider the element
$Y=\sum_{\substack{ y\in \mathcal{C} \\ y \text{ source }}} y $ in $C_2^{c^+}$.
By Lemma \ref{sommesource}, $Y$ satisfies
$$\partial_2^{c^+}(Y)=\sum_{i=1}^n x_i,$$
where $x_1,...,x_n$ are the saddle points of the frontier of $\overline{W^s(\mathcal{C})}$. So we have $[\sum_{i=1}^n x_i]\in \mathrm{Ker}(i_{c^-,c^+})$. \\

By hypothesis, the saddle point $x$ is one of the saddle points $(x_i)_{i\in[1,n]}$  and each $x_i$ satisfies $A(x_i)\leq A(x)$ since $\mathcal{C}$ is a connected component of $G_t^+(\FF)$.  So we have that $[\sum_{i=1}^n x_i]$ is in the image of $i_{t^+,c^-}$.\\

Moreover, the singularity $x$ is not homologous in $C_1^{c^-}$ to a chain of singularities of $C_1^{t^-}$. Indeed, if it was the case then, by definition, it would exist $X'\in C_1^{t^-}$ and $Y'\in C_2^{c^-}$ such that $x=X'+\partial_2^{c^-}(Y')$. \\

We set $y_x\in \mathcal{C}$ and $y_x'\in \mathcal{C}'$ the only two sources of $\FF$ such that $x$ is in the frontier of the sets $W^u(y_x)$ and $W^u(y_x')$.
The equality $x=X'+\partial_2^{c^-}(Y')$ would imply that $\langle Y'|y_x\rangle=1$ or $\langle Y' | y_x'\rangle=1$, which is impossible because, by hypothesis, $A_f(y_x')>A_f(y_x)=c$. Indeed, if we have $\partial_2^{c^-}(Y')=x-X'$ then, there exists a source $y$ such that $\langle Y| y\rangle =1$ and $\langle \partial_2^{c^-} (y) |x\rangle=1$ , which means that $x$ is in the frontier of $W^u(y)$. So $y$ is either equal to $y_x$ or $y_x'$.\\

So, we have the same result for $\sum_{i=1}^n x_i$ and so $[\sum_{i=1}^n x_i]\notin \mathrm{Im}(i_{t^-,c^-})$.\\ 

Thus, by construction, there exists a bar $(A(x),D(\mathcal{C})]$ in the barcode $B_{\gen}(\FF)$.\\
\end{proof}

Now we can prove Theorem \ref{C2close} from section \ref{génériqueF} stated as follows.\\

\begin{theorem}\label{C2closeP}
If we consider a Hamiltonian diffeomorphism $f$ with a finite number of fixed points which is $C^2$-close to the identity and generated by an autonomous Hamiltonian function then the barcode $B_{\gen}(\FF)$ is equal to the Floer homology barcode of $f$.
\end{theorem}

\begin{proof}
If we suppose that the autonomous Hamiltonian function $H$ is $C^2$ close to a constant then the Floer homology of $H$ is equal to the Morse homology of $H$, we refer to \cite{AUD} for a proof. We deduce that the Morse Homology barcode $\beta(\HM_*^t(H))_{t\in \R})$ of $H$ is equal to the Floer Homology barcode $\beta(\HF_*^t(H))_{t\in \R})$ of $H$, where $\beta$ is the functor defined in section \ref{introbarcodes} which associate a persistence module to its barcode. \\

The time one map $f_1=f$ of the Hamiltonian flow is $C^1$ close the the identity and its set of fixed points is unlinked.\\

Moreover, the gradient-lines of $H$ provides a $C^1$ foliation $\FF$ positively transverse to the natural Hamiltonian isotopy induces by $H$.  This isotopy is maximal and so fixes every fixed points of $f$. The foliation $\FF$ is gradient-like and there is no cone of leaves at the saddle points of $\FF$. \\

Moreover, the construction of the map $\mathcal{B}$ in section \ref{construction} follows the ideas of the Morse homology theory then we can assert that the barcode $\bm{\beta}_\FF$ is equal to the barcode $\beta(\HM_*^t(H))_{t\in \R}$,\\

Thus, by Theorem \ref{egalité-generique} we have
$$B_{\gen}(f)=\bm{\beta}_\FF=\beta(\HM_*^t(H))_{t\in \R}=\beta(\HF_*^t(H))_{t\in \R}.$$

So we obtain the result.
\end{proof}

\bibliographystyle{plain}

\bibliography{bibliographie.bib}

\end{document}